\def\WHO{nobody}
\def\version{20.03.2023: 3pm}
\def\users{us}  %
\def\users{final-layout}

\documentclass[12pt]{article}
\usepackage{upgreek}
\usepackage{graphicx}

\usepackage{fullpage}
\usepackage{xcolor}
\usepackage{bm,amsmath,amsthm,hyperref,amsfonts,amssymb,color}
\usepackage{mathrsfs} %
\usepackage{psfrag}     %

\usepackage{ifthen}
\ifthenelse{\equal{\users}{final-layout}}{}{
\usepackage{fancyhdr}
\pagestyle{fancy}
\headheight=60pt
\headwidth=18cm
\definecolor{gray}{gray}{0.5}
\rhead{\color{gray}Thermoplasticity large strain\\
T.Roub\'\i\v cek \& G.Tomassetti}
\chead{}
\lhead{Version \version, file: \jobname.tex\\compiled:
\number\day.\number\month.\number\year\ at
\the\hour:\ifnum\minute<10 0\fi\the\minute\ h
\Large\color{red}\WHO's working on the file
}
}

\definecolor{labelkey}{rgb}{1.,.2,0.}

\newcount\hour \newcount\minute
\hour=\time
\divide \hour by 60
\minute=\time
\loop \ifnum \minute > 59 \advance \minute by -60 \repeat

\usepackage[normalem]{ulem}

\usepackage[normalem]{ulem}
\newcommand{\stkout}[1]{\ifmmode\text{\sout{\ensuremath{#1}}}\else\sout{#1}\fi}
\usepackage{ifthen}
\usepackage{color}

\ifthenelse{\equal{\users}{final-layout}}{
	\newcommand{\REPLACE}[2]{#2}
	\newcommand{\INSERT}[1]{}
	\newcommand{\COMMENT}[1]{}
	\newcommand{\COMMENTGT}[1]{}
	\newcommand{\TODO}[1]{}
	\newcommand{\INTERNAL}[1]{}
	\newcommand{\QUESTION}[1]{}
	\newcommand{\DELETE}[1]{}

	\newcommand{\REM}[1]{\marginpar{\bfseries\tiny{\color{blue}}}}
    \newcommand{\MARGINOTE}[1]{}
}
{\newcommand{\REPLACE}[2]{{\color{red}\stkout{#1}\uline{#2}\color{black}}}
	\newcommand{\INSERT}[1]{{\color{blue}\uuline{#1}\color{black}}}
	\newcommand{\COMMENT}[1]{{\color{red}\uuline{#1}\color{black}}}
	\newcommand{\COMMENTGT}[1]{{\hfill\large\color{red}***{#1}***\color{black}\hfill}}
	\newcommand{\TODO}[1]{{\color{red}\uuline{#1}\color{black}}}
	\newcommand{\INTERNAL}[1]{\footnote{#1}}
	\newcommand{\QUESTION}[1]{{\color{brown}\uuline{#1}\color{black}}}
	\newcommand{\DELETE}[1]{{\color{red}\stkout{#1}\color{black}}}

	\newcommand{\REM}[1]{\marginpar{\bfseries\tiny{\color{blue}#1}}}
\newcommand{\MARGINOTE}[1]{\marginpar{\color{red}\tiny\texttt{#1}}}
}

\newcommand\DT[1]{\mathchoice
                 {{\buildrel{\hspace*{.1em}\text{\LARGE.}}\over{#1}}}
                 {{\buildrel{\hspace*{.1em}\text{\Large.}}\over{#1}}}
                 {{\buildrel{\hspace*{.1em}\text{\large.}}\over{#1}}}
                 {{\buildrel{\hspace*{.1em}\text{\large.}}\over{#1}}}}
\newcommand\pdt[1]{\frac{\partial{#1}}{\partial t}} %
\newcommand{\lineunder}[2]{\LU{\begin{array}[t]{c}\underbrace{#1}\vspace*{.5em}\end{array}}{\mbox{\footnotesize\rm #2}}}
\newcommand{\LU}[2]{\begin{array}[t]{c}#1\vspace*{-1em}\\_{#2}\end{array}}
\newcommand{\linesunder}[3]{\LSU{\begin{array}[t]{c}\underbrace{#1}\vspace*{.5em}\end{array}}{\mbox{\footnotesize\rm #2}}{\mbox{\footnotesize\rm #3}}}
\newcommand{\LSU}[3]{\begin{array}[t]{c}#1\vspace*{-1em}\\_{#2}\vspace*{-.5em}\\_{#3}\end{array}}
\newcommand{\morelinesunder}[4]{\LSUU{\begin{array}[t]{c}\underbrace{#1}\vspace*{.5em}\end{array}}{\mbox{\footnotesize\rm #2}}{\mbox{\footnotesize\rm #3}}{\mbox{\footnotesize\rm #4}}}
\newcommand{\LSUU}[4]{\begin{array}[t]{c}#1\vspace*{-1em}\\_{#2}\vspace*{-.5em}\\_{#3}\vspace*{-.5em}\\_{#4}\end{array}}
\newcommand{\Item}[2]{\parbox[t]{.055\textwidth}{#1}\hfill%
      \parbox[t]{.945\textwidth}{#2}\vspace*{.8mm}} 
\newcommand{\divS}{\mathrm{div}_{\scriptscriptstyle\textrm{\hspace*{-.1em}S}}^{}}
\newcommand{\nablaS}{\nabla_{\scriptscriptstyle\textrm{\hspace*{-.3em}S}}^{}}
\def\Vdots{\!\mbox{\setlength{\unitlength}{1em}
\begin{picture}(0,0)
\put(-.07,0){.}
\put(-.07,.3){.}
\put(-.07,.6){.}
\end{picture}\hspace*{.2em}}}

\numberwithin{equation}{section}

\newcounter{myfigure}
\newenvironment{my-picture}[3]{\refstepcounter{myfigure}\label{#3}\setlength{\unitlength}{1cm}\begin{picture}(#1,#2)}{\end{picture}}

\usepackage{mathrsfs}   %
\usepackage{eucal}      %

  \def\bbI{{\mathbb I}}

\def\FG{\boldsymbol}
 \def\bb{{\FG b}}  
 \def\ee{{\FG e}} \def\ff{{\FG f}} 
\def\gg{{\FG g}}  
\def\jj{{\FG j}}   
\def\mm{{\FG m}} \def\nn{{\FG n}}  
   
 \def\tt{{\FG t}}  
\def\vv{{\FG v}} \def\ww{{\FG w}} \def\xx{{\FG x}} 
\def\yy{{\FG y}}  
 \def\BB{{\FG B}}  
\def\DD{{\FG D}} %
\def\FF{{\FG F}} 
 \def\KK{{\FG K}} \def\LL{{\FG L}}

 \def\TT{{\FG T}} %
 \def\WW{{\FG W}} \def\XX{{\FG X}}

\newcommand{\R}{\mathbb R}
\newcommand{\N}{\mathbb N}
\newcommand{\Nabla}{{\bm\nabla}}
\newcommand{\Fetop}{\Fe^{\!\!\top}}
\newcommand{\Fe}{{\FF_{\hspace*{-.2em}\mathrm e^{^{^{}}}}}}

\newcommand{\Fee}{\FF_{\hspace*{-.2em}\mathrm e}}

\newcommand{\Fp}{\FF_{\hspace*{-.2em}\mathrm p^{^{^{}}}}}

\newcommand{\widetildeLp}{\widetilde{\LL\,}_{\hspace*{-.2em}\mathrm p^{^{^{}}}}}
\newcommand{\Fpzero}{\FF_{\hspace*{-.2em}\mathrm p^{^{^{}}},0}}
\newcommand{\Fpp}{\FF_{\hspace*{-.2em}\mathrm p}}

\newcommand{\pl}{\partial}
\newcommand{\eq}[1]{(\ref{#1})}
\newcommand{\Lp}{\LL_{\mathrm p^{^{^{}}}}}
\newcommand{\Tp}{\TT_{\mathrm p^{^{^{}}}}}

\renewcommand{\d}{\mathrm d}  

\newcommand{\barvarOmega}{\hspace*{.2em}{\overline{\hspace*{-.2em}\varOmega}}}

\newcommand\dx{\operatorname{d}\!\xx}
\newcommand\dt{\operatorname{d}\! t}

\newcommand\xiek{{\bm\xi_{\EPS k}}}

\usepackage{mathrsfs}
\usepackage{bbm}
\usepackage{cancel}
\usepackage[colorinlistoftodos,textsize=footnotesize,textwidth=1.4\marginparwidth]{todonotes}
\usepackage{mathtools} 
\usepackage{amsmath}
\usepackage{autobreak}
\allowdisplaybreaks
\usepackage{enotez}
\usepackage{fullpage}
\usepackage{hyperref}
\usepackage{amssymb,amsmath,amsthm,float}
\usepackage{xcolor}
\usepackage{amsmath}
\usepackage{import}
\usepackage{color}
\usepackage{bm}

\newtheorem{theorem}{Theorem}[section]

\newtheorem{definition}[theorem]{Definition}
\newtheorem{example}[theorem]{Example}

\newtheorem{remark}[theorem]{Remark}

\def\GRAVITY{\bm{g}}
\def\rhoR{\varrho_\text{\sc r}^{}}
\def\rhoRxi{\varrho_\text{\sc r}^{\bm\xi}}
\def\rhoRxieps{\varrho_\text{\sc r}^{\bm\xi_\EPS}}
\def\rhoRxizero{\varrho_\text{\sc r}^{\bm\xi_0}}
\def\rhoRxiepsk{\varrho_\text{\sc r}^{\bm\xi_{\EPS k}}}
\def\DTrhoR{\DT\varrho_\text{\sc r}^{}}

\def\COUPLING{\gamma}
\def\LAM{\lambda}
\def\W{w}
\def\OMEGA{\omega}
\def\EPS{\varepsilon}

\def\ALPH{\INSERT{\alpha}}
\def\ALPH{{\alpha}}
\def\ONEALPH{2}
\def\DIS{\DD}
\def\TWO{2}

\def\EXP{\mu}
\def\EE{{\bm e}}

\begin{document}\begin{sloppypar}
	
\allowdisplaybreaks

\vspace*{2em}

\begin{center}

{\LARGE\bf %
Inhomogeneous finitely-strained\\[-.1em]
thermoplasticity with hardening %
\\[.3em]
by an Eulerian approach\ }\footnote{This research has been partially from
the M\v SMT \v CR (Ministry of Education of the Czech Rep.) project
CZ.02.1.01/0.0/0.0/15-003/0000493, and the CSF project GA23-04676J,
and the institutional support RVO: 61388998 (\v CR). Supported by the
Italian INdAM-GNFM (Istituto Nazionale di Alta Matematica ¡V Gruppo Nazionale per la Fisica Matematica), the Grant of Excellence Departments,
MIUR-Italy (Art.1, commi 314-337, Legge
232/2016), and the Grant ``Mathematics of active materials: from mechanobiology to smart devices'' (PRIN 2017, prot. 2017KL4EF3) funded by the Italian MIUR.%
}

\bigskip\bigskip

{\sc Tom\'a\v{s} Roub\'\i\v cek${}^{(1),(2)}$\ \ \&\ \ Giuseppe Tomassetti${}^{(3)}$}

\bigskip

 ${}^{(1)}$Mathematical Institute, Charles University, 
             \\Sokolovsk\'a~83, CZ-186~75~Praha~8, Czech Republic,\\
  email: tomas.roubicek@mff.cuni.cz\\[0.8em]

    ${}^{(2)}$Institute of Thermomechanics, Czech Academy of Sciences,\\
    Dolej\v skova 5, CZ-182 00 Praha 8, Czech Republic.
    
\bigskip

 ${}^{(3)}$Department of Industrial, Electronic, and Mechanical Engineering, \\Roma Tre University,\\Via Vito Volterra 62,
 00146 Roma, Italy

email: giuseppe.tomassetti@uniroma3.it

\bigskip\bigskip

\begin{minipage}[t]{37em}{\small{\bf Abstract.}
\baselineskip=13pt
A standard elasto-plasto-dynamic model at finite strains
based on the Lie-Liu-Kr\"oner multiplicative decomposition,
formulated in rates, is here enhanced to cope with
spatially inhomogeneous materials by using
the reference (called also return) mapping. Also an isotropic
hardening can be involved. Consistent thermodynamics
is formulated, allowing for both the free and the
dissipation energies temperature dependent. The model
complies with the energy balance and entropy inequality.
A multipolar Stokes-like viscosity and plastic rate gradient are
used to allow for a rigorous analysis towards existence of weak
solutions by a semi-Galerkin approximation.

\medskip

\medskip


}
\end{minipage}
\end{center}

\section{Introduction}

Plasticity at finite (also called large) strain and its thermomechanics is
a vital part of nonlinear continuum mechanics. It has been
vastly studied in literature, often in Lagrangian
(i.e.\ reference) configuration since there is a certain
belief that this kind of description is more suitable 
for solids than the Eulerian formulation in actual
evolving configurations.

Plasticity in many materials (especially metals) exhibits also
(isotropic) {\it hardening} acting on plastic distortion.
The philosophy of this article is
to complete the %
plasticity or creep as in \cite{Roub22QHLS} by temperature like, without plasticity,
done in \cite{Roub22TVSE} and use the reference (also called return) mapping
\cite{kamrin2012reference} in order to keep track of the deformation gradient
and to allow for inhomogeneous material. We also include the plastic distortion
tensor in order to involve isotropic hardening. 
The temperature-dependence
of the Maxwellian-type damping (or the plastic yield stress) allows
for melting/solidification phase transition like in \cite{Roub??SPTC}
for a linearized model.

Let us summarize the main attributes of the devised model:\\
\Item{{\bf---}}{Concept of {\it hyperelastic materials}
(whose conservative-stress
response comes from an energy) combined with the  Stokes damping into the
{\it Kelvin-Voigt viscoelastic rheology} and with {\it plasticity with hardening}.}
\Item{\bf---}{Deformation in actual configuration (i.e.\ {\it Eulerian
approach}) and corresponding evolution of the deformation gradient, cf.\
\eq{ultimate} below,}
\Item{{\bf---}}{The {\it rate formulation} in terms of velocity and deformation
gradient is used while the deformation itself does not explicitly occur.}
\Item{\bf---}{Lie-Liu-Kr\"oner {\it multiplicative decomposition} of the
deformation gradient to the elastic and the inelastic (plastic) strains, cf.\ 
\eq{Green-Naghdi}, with the plastic distortion being {\it isochoric}, i.e.\
having determinant equal 1,}
\Item{{\bf---}}{{\it Inhomogeneous material} in terms of free energy, dissipation
potential, and mass density in referential configuration by involvement
of the mapping from actual to a reference configuration.}
\Item{{\bf---}}{Mechanical consistency in the sense that
{\it frame indifference} of the free energy (which is in particular
{\it nonconvex} in terms of deformation gradient) is admitted, as well
as its {\it singularity} under infinite compression in relation with
{\it local non-interpenetration}.}
\Item{{\bf---}}{Thermodynamical consistency of the thermally coupled system
in the sense that the {\it total energy is conserved} in a closed system,
the {\it Clausius-Duhem entropy inequality}
holds, and temperature stays non-negative.}
\Item{{\bf---}}{Gradient theories, here applied on the dissipative potential,
in particuar the nonconservative part of the stress in the Kelvin-Voigt model
containing a higher-order component reflecting the concept of nonsimple
{\it multipolar media} is exploited.}
\Item{\bf---}{A dissipation potential acting on symmetric velocity gradient
and on the plastic distorsion rate, and making the latter
temperature dependent allows for modelling of usual plastic weakening under
higher temperatures.}
\Item{\bf---}{A variant with no hardening, describing a certain sort of
perfect plasticity \COMMENT{} with gradient of plastic distrortion rate 
or a creep in a Jeffreys viscoelastic rheology, is allowed. In the latter
case, {\it solidification or melting phase transitions} like
in \cite{Roub??SPTC} for a linearized model can be implemented.}
\Item{\bf---}{The evolution based on the 
evolution of the plastic distortion and the deformation gradient through the reference mapping,
in addition on the momentum equilibrium and on the flow rule of plastic distortion
through the plastic distorsion rate, and the heat-transfer equation.}
\Item{{\bf---}}{The model allows for rigorous mathematical analysis as far as
existence and certain regularity of energy-conserving weak solutions
concerns.}
\vspace*{.3em}

\noindent
As far as the non-negativity of temperature, below we will be able
to prove only that at least some solutions enjoy this attribute, although
there is an intuitive belief that all possible solutions will make it and
a hope that more advanced analytical techniques would rigorously prove it.

Although formulated in actual configuration, the data are taken
in a referential configuration. This is quite standard in
Eulerian continuum mechanics \cite{GuFrAn10MTC,Maug92TPF} and is more
suitable for analysis of thermodynamically coupled system as
in \cite{Roub22TVSE} because the stress does not involve the
pressure comming from the free energy. Considering data in referential
configuration allows naturally to consider spatially {\it inhomogenous
material}. Yet, the Eulerian description then needs to involve
the {\it reference mapping} (denoted by $\bm\xi$) from the actual
to the reference configuration. Actually, in theoretical continuum
mechanics, $\bm\xi$ is more often called a return mapping but it
should be noted that the notion of a ``return mapping'' is used in
plasticity in 
\cite{SNPeOw08CMP,Hash20NCMF,simo2006computational} in a different meaning
as a certain projection algorithm.

Having involved the reference mapping, we have immediately the
deformation gradient at disposal, cf.\ \eqref{defF}. Having
also plastic distortion explicitly in the formulation of the
model due to the mantioned hardening, we have also the elastic
distortion at disposal, cf.\ \eqref{Green-Naghdi}. Thus,
in contrast to \cite{Roub22QHLS}, we do not need to involve
the evolution-and-transport equation for elastic
distortion \eqref{DTFe}.
We condsider also inertial effects (i.e.\ the full {\it dynamical problem}),
which here helps also in analysis in estimation of adiabatic effects,
cf.\ Remark~\ref{rem-inert} below, although it brings various technicalities
itself.

The main notation used in this paper is summarized in the following table:

{\small
\vspace*{1em}
\hspace*{-1.7em}\fbox{
\begin{minipage}[t]{16em}%

$\vv$ Eulerian velocity (in m/s),

$\yy$ deformation (in m),

$\rhoR=\rhoR(\XX)$ referential mass density, %

 $\bm\xi$ reference mapping,

$\FF$ deformation gradient,

$\Fe$ elastic part of $\FF$,

$\Fp$ plastic distortion,

$\theta$ temperature (in K)

$\TT$ Cauchy stress (symmetric, in Pa),

$\mathfrak{H}$ hyperstress (in Pa\,m),

$\jj$ heat flux (in W/m$^2$),

$\ff$ traction load,

$\det(\cdot)$ determinant of a matrix,

${\rm Cof}(\cdot)$ cofactor matrix,

$\R_{\rm sym}^{d\times d}=\{A\in\R^{d\times d};\ A^\top=A\}$,

$\R_{\rm dev}^{d\times d}=\{A\in\R_{\rm sym}^{d\times d};\ {\rm tr}\,A=0\}$,

\end{minipage}\hspace*{0em}
\begin{minipage}[t]{24em}%

$\Lp=\DT\Fp\Fpp^{-1}$ plastic distortion rate (in s$^{-1}$),

$\psi=\psi(\XX,\FF,\theta)$ referential free energy (in J/m$^3$=Pa),

$\zeta=\zeta(\theta;\Lp)$ plastic dissipation potential,

$\varphi=\varphi(\XX,\Fe)$ referential stored energy (in J/m$^3$),

$\COUPLING=\COUPLING(\XX,\Fe,\theta)$ referential heat part of free energy,

$\OMEGA=\OMEGA(\XX,\Fe,\theta)$ thermal part of the internal energy,%

$\ee(\vv)=\frac12(\Nabla\vv)^\top\!+\frac12\Nabla\vv$ small strain rate (in s$^{-1}$),

$\DIS=\nu_0|\ee(\vv)|^{p-2}\ee(\vv)$ dissipative part of Cauchy stress,

$\W$ heat part of internal energy (enthalpy, in J/m$^3$),

$\cdot$ or $:$ scalar products of vectors or matrices, 

$\Vdots\ \ $ scalar products of 3rd-order tensors,

$\kappa=\kappa(\XX,\Fe,\theta)$ %
thermal conductivity (in W/m$^{-2}$K$^{-1}$),

$c=c(\XX,\Fe,\theta)$ referential heat capacity (in Pa/K),

$\nu_1>0$ a bulk hyper-viscosity coefficient,

$\GRAVITY$ external bulk load (gravity acceleration in m/s$^{2}$),

$(^{_{_{\bullet}}})\!\DT{^{}}=\pdt{}{^{_{_{\bullet}}}}+(\vv{\cdot}\Nabla)^{_{_{\bullet}}}$
convective time derivative.

\end{minipage}
}
}

\vspace*{1em}

The formulation of the model in actual Eulerian configuration and its
energetics and thermodynamics is presented in Section~\ref{sec-model}.
Then, in Section~\ref{sec-anal}, the rigorous analysis by a suitable
regularization and a (semi) Faedo-Galerkin approximation is performed.

\section{The model}\label{sec-model}

After recalling basic kinematics of finite-strain continuum mechanics
towards plasticity models, we will now formulate the model and derive
formally its energetics and thermodynamical consistency.

\subsection{Finite-strain kinematics}
In finite-strain continuum mechanics, the basic kinematical ingredient is a
time-evolving deformation $\yy(\cdot,t):\Omega\to\R^d$ (where $d$ is the chosen space dimension) as a mapping from the reference
configuration of the body $\Omega\subset\R^d$ into the physical space $\R^d$. In general, the image $\yy(\Omega,t)$ of the reference configuration is a time dependent domain. However, for analytical reasons, we shall confine ourselves to situations when the body is constrained at its boundary in such a way that the region $\varOmega=\yy(\Omega,t)$ occupied by the body does not depend on time. We shall denote by $\XX\in\Omega$ the typical point of the reference configuration (the so-called ``Lagrangian'' space variable) and by $\xx\in\varOmega$ the typical point of the current configuration (the so-called ''Eulerian'' space variable). 

It is customarily assumed that the deformation $\yy(\cdot,t)$ be an invertible function for each time $t$. We shall denote by $\bm\xi(\cdot,t):\varOmega\to\Omega$ the inverse of $\bm y(\cdot,t)$, so that
\begin{equation}\label{inverse}
	\bm\xi(\yy(\XX,t),t)=\XX.
\end{equation}
Given any spatial time dependent field $f(\xx,t)$ defined for all $\xx\in\varOmega$,
we shall denote its {\it material description} $f_\yy$ by
\begin{equation}\label{pullback}
	f_{\yy}(\XX,t)=f(\yy(\XX,t),t).
\end{equation}
Likewise, if $g(\XX,t)$ is a referential time dependent field, we use the notation
\begin{equation}\label{pushforward}
g^{\bm\xi}(\xx,t)=g(\bm\xi(\xx,t),t)	
\end{equation}
to denote the {\it spatial description} of the same field. With this %
notation, we can express the spatial velocity in terms of the deformation as
$\vv=(\pdt{}\yy)^{\bm\xi}$.
Moreover, we can define the material derivative of a spatial field $f$ as 
\begin{equation}\label{defmatder}
	\DT f=\Big(\frac{\partial f_{\yy}}{\partial t}\Big)^{\bm\xi}.
\end{equation}
The material derivative obeys the Leibnitz product rule, i.e.\
for two spatial fields $f$ and $g$, it holds
$\DT{\overline{fg}}=\DT f g+f\DT g$. Furthermore, by the chain rule,
\begin{equation}\label{matDer}
\DT f=\pdt f+\vv{\cdot}\nabla f.	
\end{equation}
Differentiating \eqref{inverse} with respect to $t$ and using the chain rule yields 
\begin{equation}\label{dt}
\DT{\bm\xi}:=\frac{\partial\bm\xi}{\partial t}+(\bm v\cdot\nabla)\bm\xi={\bm0}.	
\end{equation}
The spatial tensor field 
\begin{equation}\label{defF}
\FF=(\nabla\bm\xi)^{-1},
\end{equation}
provides the spatial description of the deformation gradient $\nabla\yy$, that is, $\FF=(\nabla\bm y)^{\bm\xi}$.
From the definition \eqref{matDer} we have $\DT{\overline{\nabla\bm\xi}}=\nabla\frac{\partial\bm\xi}{\partial t}+(\vv{\cdot}\nabla)\nabla\bm\xi$, whence, since $\bm\xi$ satisfies \eqref{dt}, $\DT{\overline{\nabla\bm\xi}}=-\nabla((\bm v{\cdot}\nabla)\bm\xi)+(\vv{\cdot}\nabla)\nabla\bm\xi$. By working with components, one realizes that $\nabla((\bm v{\cdot}\nabla)\bm\xi)=(\nabla\bm\xi)\nabla\bm v+(\bm v{\cdot}\nabla)\nabla\bm\xi$. Thus 
\begin{equation}\label{dtnabla}
\DT{\overline{\nabla\bm\xi}}=-(\nabla\bm\xi)\nabla\vv.	
\end{equation}
By \eqref{defF} we have $\FF\nabla\bm\xi=\mathbb I$ where $\mathbb I$ is the identity matrix. Taking the material derivative yields $\DT\FF\nabla\bm\xi+\FF\DT{\overline{\nabla\bm\xi}}=0$, that is $\DT\FF\FF^{-1}=-\FF\DT{\overline{\nabla\bm\xi}}$. Then, by \eqref{dtnabla} we arrive at the following
{\it transport equation} for the
deformation gradient:
\begin{align}
\DT\FF=(\nabla\vv)\FF\,.
  \label{ultimate}\end{align}
From this, we obtain the {\it evolution-and-transport equation} for the
Jacobian $J=\det\FF$ as
\begin{align}%
\label{DT-det}
\DT J%
&={\rm Cof}\FF{:}\DT\FF=J\FF^{-\top}{:}\DT\FF
=J\bbI{:}\DT\FF\FF^{-1}=J\bbI{:}\Nabla\vv= J{\rm div}\,\vv\,,
\end{align}
as well as the
{\it evolution-and-transport equation} for $1/\!J$ as
\begin{align}%
\DT{\overline{1/J}}\ 
=-({\rm div}\,\vv)/J\,.
\label{DT-det-1}\end{align}
Note that \eqref{DT-det-1} can also be rewritten as
\begin{equation}
	\pdt{}(1/J)=-\vv{\cdot}\nabla(1/J)-{\rm div}\vv/J=-{\rm div}((1/J)\vv)\,.\label{DT-det-1+}
\end{equation}
The mass density (in kg/m$^3$) is an extensive variable, and its transport
(expressing the conservation of mass) writes as the {\it continuity
equation} $\pdt{}\varrho+{\rm div}(\varrho\vv)=0$,
or, equivalently, the {\it mass transport equation}
$\DT\varrho=-\varrho\,{\rm div}\,\vv$.
One can determine the density $\varrho$ instead of 
this transport equation from the algebraic relation
\begin{align}
\varrho(\xx,t)=\rhoR(\bm\xi(\xx,t))/J(\xx,t)\,,
\label{density-algebraically}\end{align}
where $\rhoR$ is the mass density in the reference
configuration. Indeed, relying on \eq{DT-det}, we have
\begin{align}%
\DT\varrho/\varrho=
\big(\rhoR\,\,\DT{\overline{1/J}}
+\DTrhoR/J\big)J/\rhoR=-\DT J/J=-{\rm div}\,\vv
\label{towards-cont-eq}\end{align}
because $\DTrhoR=0$.

\subsection{Kinematics of finite-strain plasticity}
Our treatment of finite-strain plasticity will be based on the Kr\"oner-Lie-Liu
\cite{Kron60AKVE,LeeLiu67FSEP} {\it multiplicative
decomposition}\index{multiplicative decomposition}
\begin{align}\label{Green-Naghdi}
\FF=\Fe\Fp\,,
\end{align}
where $\Fp$ is a {\em plastic distortion} tensor. The plastic distortion tensor $\Fp$ is interpreted as a transformation of the reference configuration
into an intermediate stress-free configuration, which is then mapped into the current configuration by the
{\it elastic strain}\index{elastic strain} $\Fe$. 

Applying the material derivative to
\eq{Green-Naghdi} and using \eq{ultimate}, we obtain
$(\Nabla\vv)\FF=\DT\FF=\DT\Fe\Fp+\Fe\DT\Fp$. Multiplying both sides by $\FF^{-1}=\Fpp^{-1}\Fee^{-1}$, we obtain the following decomposition of the velocity gradient:
\begin{align}
\Nabla\vv=\!\!\!\!\!\!\morelinesunder{\DT\Fe\Fee^{-1}\!\!\!}{elastic\ \ }{distortion}{rate}\!\!\!+\Fe\!\!\!\!\!\!\!\morelinesunder{\DT\Fp\Fpp^{-1}\!\!\!\!}{plastic}{distortion}{rate $=:\Lp$}\!\!\!\!\Fee^{-1},
\label{dafa-formula}\end{align}
cf.\ e.g.\
\cite{BesGie94MMID,Dafa84PSCS,GurAna05DMSP,GuFrAn10MTC,HasYam13IFST,KhaHua95CTP,Lee69EPDF,MauEps98GMSE,RajSri04TRC,XiBrMe00CFEP}.
The form of the rate $\Lp=\DT\Fp\Fpp^{-1}$ which occurs in the dissipation
potential
is also compatible with the so-called plastic indifference, cf.\ e.g.\
\cite{Miel03EFME}. %
Then, by making use \eqref{dafa-formula}, we have
\begin{equation}\label{DTFe}
	\DT\Fe=(\nabla\vv)\Fe-\Fe\Lp. 
\end{equation}
Analogously to \eq{DT-det}, we have
\begin{align}\label{D/DT-of-determinant}
\DT{\overline{\det\Fp}}=(\det\Fp){\rm tr}(\DT\Fp\Fpp^{-1})=(\det\Fp){\rm tr}\Lp\,\ \ \ \text{ with }\ \Lp=\DT\Fp\Fpp^{-1}\,.
\end{align}
From this we can see that, to ensure $\det\Fp=1$, it suffices to ensure
${\rm tr}\Lp=0$ provided the initial inelastic deformation is isochoric.
It is important that the constraint ${\rm tr}\Lp=0$ is linear, in contrast to
the non-affine constraint $\det\Fp=1$.
Thus, we can implement it exactly in the dissipative
part by considering a (temperature-dependent) dissipation potential for $\Lp$
imposing this constraint, cf.\ \eqref{ass-inelastic-dissip} below.
Having $\det\Fp=1$ during the whole evolution,
we have $J=\det\FF=\det(\Fe\Fp)=\det\Fe\det\Fp=\det\Fe$ and 
\eq{density-algebraically} can be written as
\begin{align}
\varrho(\xx,t)=\det(\nabla\bm\xi(\xx,t))\rhoR({\bm\xi}(\xx),t)=\frac{\rhoR({\bm\xi}(\xx),t)}{\det\Fe(\xx,t)}\,.
\label{density-algebraically+}\end{align}

\subsection{Balance equations}
We derive the relevant balance equation through the principle of virtual powers. Having in mind dissipative effects associated to higher gradients
of the velocity and to the gradient of the plastic distortion rate, we consider a form in the internal power that depends on the second gradient of the virtual velocity and on the first gradient of the plastic distortion rate. Specifically, given a smooth subregion $\mathcal P\subset\varOmega$ we define the virtual power expended within $\mathcal P$ as
\begin{equation}\label{internal-power}
	\mathcal I(\mathcal P)[\widetilde\vv,\widetildeLp]=
	\int_{\mathcal P}\TT{:}\nabla\widetilde\vv
	+\mathfrak{H}\vdots\nabla^2\widetilde\vv+
	\Tp{:}\widetildeLp+\mathfrak{H}_{\rm p}\vdots\nabla\widetildeLp\,\d\xx,
\end{equation}
where $\widetilde\vv$ and $\widetildeLp$ are, respectively, a virtual velocity and a virtual plastic distortion field. Here $\TT$ is the Cauchy stress, $\mathfrak H$ is a hyperstress, $\Tp$ is the plastic microstress and $\mathfrak H_{\rm p}$ is a plastic higher-order stress. Next, we define the external power expended on $\mathcal P$ as
\begin{equation}
\mathcal W(\mathcal P)[\widetilde\vv,\widetildeLp]=\int_{\mathcal P}\bb\cdot\widetilde\vv \,\d\xx+\int_{\partial\mathcal P}\tt\cdot\widetilde\vv+\bm m\cdot\partial_\nn\widetilde\vv+\KK{:}\widetildeLp\,\d S\,.
\end{equation}
Here $\bb$ is the body force, $\tt$ is the surface traction, $\mm$ is a diffused hypertraction, and $\KK$ is a plastic microscopic traction conjugate to $\widetildeLp$. 
We then apply the principle of virtual powers by requiring that 
\begin{equation}\label{pvp2}
\mathcal I(\mathcal P)[\widetilde v,\widetildeLp]=\mathcal W(\mathcal P)[\widetilde v,\widetildeLp]
\end{equation}
for every part $\mathcal P\subset\varOmega$ and for every choice of the virtual velocities $(\widetilde\vv,\widetildeLp)$. Such requirement yields the balance equations
\begin{subequations}\label{pvp}
\begin{align}
	&{\rm div}(\TT-{\rm div}\mathfrak H)+\bb=0,\\
	&-\Tp+{\rm div}\mathfrak H_{\rm p}=0,
\end{align}
\end{subequations}
as well as the following relations between the stresses and the tractions 
\begin{subequations}
\begin{align}\label{pvp3}
&(\TT-{\rm div}\mathfrak H)\nn-\divS({\mathfrak H\nn})=\tt,\\ 
&\mathfrak H{:}(\nn{\otimes}\nn)=\mm,\\
&\mathfrak H_p\nn=\KK,
\end{align}
\end{subequations}
where $\divS={\rm tr}(\nablaS)$ is the surface divergence. 

Neglecting inertia, energy balance states that for every time dependent subregion $\mathcal P_t$ convecting with the body (i.e.\ a region obtained as the image, under the deformation, of a constant referential domain),
\begin{equation}\label{bal11}
\frac{\d}{\d t}\int_{\mathcal P_t}	J^{-1}e\,\d\xx
+\int_{\partial\mathcal P_t}\jj{\cdot}\nn\,\d S=\mathcal W(\mathcal P_t)[\vv,\Lp]+\int_{\mathcal P_t}q\,\d\xx\,,
\end{equation}
where $J=\det\FF$, $e$ is the internal energy per unit reference volume, $\jj$ is the heat flux, $\vv$ is the velocity, $\Lp=\DT\FF_p\Fp^{-1}$ is the rate of plastic flow as in \eq{dafa-formula}, and $q$ is the bulk heat supply. Using \eqref{pvp2} with $\widetilde\vv=\vv$ and $\widetildeLp=\Lp$ allows us to replace the external power expenditure $\mathcal W(\mathcal P_t)[\vv,\Lp]$ appearing on the right-hand side of \eqref{bal11} with the internal power expenditure $\mathcal I(\mathcal P_t)[\vv,\Lp]$. Then, by a standard localization argument based on the arbitrariness of $\mathcal P_t$ we obtain
\begin{equation}\label{enBal}
J^{-1}\DT e+{\rm div}\jj=\TT{:}\nabla\vv
	+\mathfrak{H}\vdots\nabla^2\vv+
	\Tp{:}\Lp+\mathfrak{H}_{\rm p}\vdots\nabla\Lp+q\,.
\end{equation}
We next turn our attention to the second law of thermodynamics in the form of entropy imbalance, which states that for evert part $\mathcal P_t$ convecting with the body
\begin{equation}
	\int_{\mathcal P_t}J^{-1}\eta\,\d\xx+\int_{\partial\mathcal P_t}\frac{\jj\cdot\nn}\theta\,\d S\ge \int_{\mathcal P_t}\frac q\theta\,\d\xx\,,
\end{equation}
where $\eta$ is the entropy per unit volume in the reference configuration and $\theta$ is the absolute temperature. Again, a standard localization argument yields
\begin{equation}
J^{-1}\DT\eta\ge -{\rm div}\left(\frac\jj 
\theta\right)+\frac q\theta=-\frac 1\theta{\rm div}\jj+\frac 1 {\theta^2}\nabla\theta\cdot\jj+\frac q\theta\,,
\end{equation}
whence
$J^{-1}\theta\DT\eta\ge-{\rm div}\jj+\frac 1 \theta\nabla\theta\cdot\jj+q$.	
On introducing the free energy 
\begin{equation}\label{free-energy}
\psi=e-\theta\eta,
\end{equation}	
 the above inequality can be rewritten as
\begin{equation}\label{entropyImbalance}
J^{-1}\DT e-J^{-1}\eta\DT\theta-J^{-1}\DT\psi\ge-{\rm div}\jj+\frac 1 \theta\nabla\theta\cdot\jj+q\,.
\end{equation}	
On combining \eqref{entropyImbalance} with \eqref{enBal}, we obtain
\begin{equation}\label{inequality}
\TT{:}\nabla\vv
	+\mathfrak{H}\vdots\nabla^2\vv+
	\Tp{:}\Lp+\mathfrak{H}_{\rm p}\vdots\nabla\Lp	\ge J^{-1}\eta\DT\theta+J^{-1}\DT\psi+\frac 1 \theta\nabla\theta\cdot\jj.
\end{equation}
We now assume that, at a given material point $\XX\in\Omega$, the free energy depends on the triplet $(\Fe,\Fp,\theta)$, through a function $\widehat\psi(\XX,\Fe,\Fp,\theta)$. Such dependence may change from one point to another. Then the time-dependent spatial field $\psi$ is given by $\psi(\xx,t)=\widehat\psi(\bm\xi(\xx,t),\Fe(\xx,t),\Fp(\xx,t),\theta(\xx,t))$. Using the notation \eqref{pushforward} we can write such relation in the compact form $\psi=\widehat\psi^{\bm\xi}(\Fe,\Fp,\theta)$. On taking the material time derivative and on using \eq{DTFe}, we find
 \begin{align}
 \DT\psi&=[\widehat\psi'_{\Fe}]^{\bm\xi}(\Fe,\Fp,\theta){:}\DT\Fe+[\widehat\psi'_{\Fp}]^{\bm\xi}(\Fe,\Fp,\theta){:}\DT\Fp+[\widehat\psi'_\theta]^{\bm\xi}(\Fe,\Fp,\theta)\DT\theta
 \nonumber\\
 &=[\widehat\psi'_{\Fe}]^{\bm\xi}(\Fe,\Fp,\theta)\Fetop{:}\nabla\vv-\Fetop [\widehat\psi'_{\Fe}]^{\bm\xi}(\Fe,\Fp,\theta){:}\Lp
 \nonumber
 \\
 &\hspace{1.5em}+[\widehat\psi'_{\Fp}]^{\bm\xi}(\Fe,\Fp,\theta)\Fp^\top{:}\Lp
 +[\widehat\psi'_\theta]^{\bm\xi}(\Fe,\Fp,\theta)\DT\theta.\label{mat-der-psi}
 \end{align}
 Here, using again the notation \eqref{pushforward}, we have set $[\widehat\psi'_\Fe]^{\bm\xi}(\Fe,\Fp,\theta)=\frac{\partial\widehat\psi}{\partial\Fe}(\bm\xi,\Fe,\Fp,\theta)$, and we have employed a similar notation for the other partial derivatives. On substituting \eqref{mat-der-psi} into \eqref{inequality} we arrive at the following specialized form of the second law:
\begin{align}\nonumber
	&(\TT-J^{-1}[\widehat\psi'_{\Fe}]^{\bm\xi}(\Fe,\Fp,\theta)\Fetop){:}\nabla\vv+\mathfrak H\vdots\nabla^2\vv
	\nonumber\\
	&+(\Tp+J^{-1}\Fetop[\widehat\psi'_{\Fe}]^{\bm\xi}(\Fe,\Fp,\theta)
	-J^{-1}[\widehat\psi'_{\Fp}]^{\bm\xi}(\Fe,\Fp,\theta)){:}\Lp\nonumber
        \\
        &+\mathfrak H_{\rm p}\Vdots\nabla\Lp-J^{-1}(\eta+[\widehat\psi'_{\theta}]^{\bm\xi}(\Fe,\Fp,\theta))\DT\theta-\frac 1 \theta\nabla\theta\cdot\jj\ge 0.\label{diss-ineq}
\end{align}
Motivated by \eqref{diss-ineq}, we adopt the following set of constitutive equations  
\begin{subequations}\label{const}
\begin{align}
 &\TT=J^{-1}[\widehat\psi'_{\Fe}]^{\bm\xi}(\Fe,\Fp,\theta)\Fetop+\nu_0|\ee(\vv)|^{p-2}\ee(\vv)\,,\\
 &\Tp\in {\rm dev}(J^{-1}[\widehat\psi'_{\Fp}]^{\bm\xi}(\Fe,\Fp,\theta)\Fp^\top-J^{-1}\Fetop[\widehat\psi'_{\Fe}]^{\bm\xi}(\Fe,\Fp,\theta))+J^{-1}\partial_{\Lp}\!\zeta(\theta;\Lp)\,,\\
 &\mathfrak H=\nu_1|\nabla\ee(\vv)|^{p-2}\nabla\ee(\vv)\,,\\
 &\mathfrak H_{\rm p}=\nu_2|\nabla\Lp|^{q-2}\nabla\Lp\,,
 \\
 &\jj=-\kappa^{\bm\xi}(\Fe,\theta)\nabla\theta\,,\ \ \text{ and}\\
 &\eta=-[\widehat\psi'_\theta]^{\bm\xi}(\Fe,\Fp,\theta)\,.
\end{align}
\end{subequations}
Here $\nu_i$, $i=0,1,2$ are positive constants, $q$ and $p$ are exponents greater than the space dimension, and
$\ee(\vv)={\rm sym}(\nabla\vv)$ is the symmetric part of the velocity gradient.
Furthermore, ${\rm dev}(\cdot)$ stands for the deviatoric part, $\partial_{\Lp}\!\zeta(\theta;\Lp)$ is the subdifferential of a possibly non-smooth dissipation potential
\begin{align}\label{ass-inelastic-dissip}
\zeta(\theta;\cdot):\R_{\rm dev}^{d\times d}\to[0,+\infty)\ ,
\ \ \zeta(\theta;0)=0,\ \
\end{align}
which we assume to be convex with respect to the plastic distortion rate $\Lp$. Moreover, $\kappa^{\bm\xi}(\Fe,\theta)=\kappa(\bm\xi,\Fe,\theta)$, where $\kappa(\XX,\Fe,\theta)>0$ is the thermal conductance, which we assume to be unaffected by plastic strain.
Having $\Lp$ valued in $\R_{\rm dev}^{d\times d}$ 
together with the assumption $\det\Fpzero=1$ where $\Fpzero$ is the initial
value of $\Fp$, we will have $\det\Fp=1$ during the whole evolution;
cf.\ \eqref{D/DT-of-determinant} or also  \cite[Sect.\,91.3]{GuFrAn10MTC}. 

Concerning the body forces $\bb$, we shall assume
$\bb=\varrho(\bm g-\DT\vv)$,
where $\gg:\varOmega\to\R^d$ is a prescribed acceleration field. Since mass density is given by \eqref{density-algebraically+}, the body force can be written as
$\bb={\rhoRxi(\gg-\DT\vv)}{\det(\nabla\bm\xi)}$.	
On substituting \eqref{const} into \eqref{pvp} we obtain the following system
of partial differential equations and inclusions:
\begin{subequations}\label{system-mech}
\begin{align}
	&\!\nonumber{\rhoRxi\DT\vv}\det(\nabla\bm\xi)+{\rm div}(\TT_{\rm e}+\DD-{\rm div}\mathfrak H)={\rhoRxi\gg}\det(\nabla\bm\xi)\ \ \
        \text{with }\ \TT_{\rm e}=J^{-1}[\widehat\psi'_{\Fe}]^{\bm\xi}(\Fe,\Fp,\theta)\Fetop,\\ &\hspace*{12em}\DD=\nu_0|\ee(\vv)|^{p-2}\ee(\vv)\,,\ \text{ and }\ \mathfrak H=\nu_1|\nabla\ee(\vv)|^{p-2}\nabla\ee(\vv),\\
&\!\partial_{\Lp}\!\zeta(\theta,\Lp)\ni {\rm dev}\Big(\Fetop[\widehat\psi'_{\Fe}]^{\bm\xi}(\Fe,\Fp,\theta){-}[\widehat\psi'_{\Fp}]^{\bm\xi}(\Fe,\Fp,\theta)\Big)
+J{\rm div}(\nu_2|\nabla\Lp|^{q-2}\nabla\Lp).\!\!
\end{align}
\end{subequations}
We now turn our attention to the energy balance \eqref{enBal}. We assume that the free energy splits additively into a part that depends on $\Fe$ and $\theta$ and a part that depends on $\Fp$ only, the latter accounting for hardening effects. 
\begin{equation}\label{psi-split}
  \widehat\psi(\XX,\Fe,\Fp,\theta)=\widehat\varphi(\XX,\Fe)+\widehat\COUPLING(\XX,\Fe,\theta)
  +\widehat\phi(\XX,\Fp)\ \ \ \text{ with }\ \ \widehat\COUPLING(\XX,\Fe,0)=0\,.
\end{equation}
Then accordingly to \eqref{free-energy}, the internal energy is given by $e=\widehat e^{\bm\xi}(\Fe,\Fp,\theta)$ with 
\begin{equation}\label{defe}
	\widehat e(\XX,\Fe,\Fp,\theta)=\widehat\varphi(\XX,\Fe)+\widehat\phi(\XX,\Fp)+\widehat\COUPLING(\XX,\Fe,\theta)-\theta\widehat\COUPLING'_\theta(\XX,\Fe,\theta).
\end{equation}
We now introduce the \emph{enthalpy} as
$w=J^{-1}(\widehat\COUPLING^{\bm\xi}(\Fe,\theta)-\theta[\widehat\COUPLING'_\theta]^{\bm\xi}(\Fe,\theta))$.
From the identity $J^{-1}\DT f=\DT{\overline{(J^{-1}f)}}-\DT{\overline{J^{-1}}}f=\DT{\overline{(J^{-1}f)}}+J^{-1}f\,{\rm div}\vv$
we have
\begin{equation}\label{dttt}
J^{-1}\DT{(\overline{\widehat\COUPLING^{\bm\xi}(\Fe,\theta)-\theta[\widehat\COUPLING'_\theta]^{\bm\xi}(\Fe,\theta)})}=\DT w+w{\rm div}\vv.
\end{equation}
From \eqref{defe}, by making use of \eqref{dttt} we obtain
$J^{-1}\DT e=J^{-1}[\widehat\varphi'_\Fe]^{\bm\xi}(\Fe){:}\DT\Fe+J^{-1}[\widehat\phi'_{\Fp}]^{\bm\xi}(\Fp){:}\DT\Fp+\DT w+w{\rm div}\vv$.
Thus, energy balance \eqref{enBal} becomes
\begin{align}\nonumber
	\DT w+w\,{\rm div}\vv+{\rm div}\jj&=\TT{:}\nabla\vv
	+\mathfrak{H}\Vdots\nabla^2\vv+
	\Tp{:}\Lp+\mathfrak{H}_{\rm p}\Vdots\nabla\Lp
\\&\ \ \ -J^{-1}[\widehat\varphi'_{\Fe}]^{\bm\xi}(\Fe){:}\DT\Fe
-J^{-1}[\widehat\phi'_{\Fp}]^{\bm\xi}(\Fp){:}\DT\Fp\,.
\label{enBal2}\end{align}
Using the constitutive equations \eqref{const}, and recalling the decomposition \eqref{dafa-formula},
we can rewrite the right-hand side of \eqref{enBal2} as:
\begin{align}
	\TT{:}&\nabla\vv
	+\mathfrak{H}\Vdots\nabla^2\vv+
	\Tp{:}\Lp+\mathfrak{H}_{\rm p}\Vdots\nabla\Lp
-J^{-1}[\widehat\varphi']^{\bm\xi}(\Fe){:}\DT\Fe-J^{-1}[\widehat\phi']^{\bm\xi}(\Fp){:}\DT\Fp\nonumber\\
&=\big(J^{-1}[\widehat\psi'_{\Fe}]^{\bm\xi}(\Fe,\Fp,\theta)\Fe^\top+\nu_0|\ee(\vv)|^{p-2}\ee(\vv)\big){:}\nabla\vv+\nu_1|\nabla\ee(\vv)|^{p-2}\nabla\ee(\vv)\vdots\nabla^2\vv\nonumber\\
&\qquad+\big({J^{-1}[\widehat\phi']^{\bm\xi}(\Fp)}\Fp^\top-J^{-1}\Fe^\top[\widehat\psi'_\Fe]^{\bm\xi}(\Fe,\Fp,\theta)\big){:}\Lp+J^{-1}\partial_{\Lp}\!\zeta(\theta;\Lp){:}\Lp\nonumber\\
&\qquad+\nu_2|\nabla\Lp|^{q-2}\nabla\Lp\vdots\Lp-J^{-1}[\widehat\varphi']^{\bm\xi}(\Fe){:}\DT\Fe-{J^{-1}[\widehat\phi']^{\bm\xi}(\Fp)}{:}\DT\Fp\nonumber\\
&={J^{-1}[\widehat\psi'_{\Fe}]^{\bm\xi}(\Fe,\Fp,\theta)\Fe^\top{:}\DT\Fe\Fe^{\!-1}}+{J^{-1}[\widehat\psi'_\Fe]^{\bm\xi}(\Fe,\Fp,\theta)\Fe^{\top}{:}\Fe\Lp\Fe^{\!-1}}\nonumber
\\
&\qquad+\nu_0|\ee(\vv)|^p+\nu_1|\nabla\ee(\vv)|^p-{J^{-1}\Fe^\top[\widehat\psi'_\Fe]^{\bm\xi}(\Fe,\Fp,\theta){:}\Lp}\nonumber\\
&\qquad+J^{-1}\partial_{\Lp}\!\zeta(\theta;\Lp){:}\Lp+\nu_2|\nabla\Lp|^q-
{J^{-1}[\widehat\varphi']^{\bm\xi}(\Fe){:}\DT\Fe}\nonumber\\
&={\displaystyle {J^{-1}[\widehat\varphi'_{\Fe}]^{\bm\xi}(\Fe){:}\DT\Fe}+[\widehat\gamma'_{\Fe}]^{\bm\xi}(\Fe,\theta){:}\DT\Fe}+{J^{-1}[\widehat\psi'_\Fe]^{\bm\xi}(\Fe,\Fp,\theta)\Fe^{\top}{:}\Fe\Lp\Fe^{\!-1}}\nonumber
\\
&\qquad+\nu_0|\ee(\vv)|^p+\nu_1|\nabla\ee(\vv)|^p-{J^{-1}\Fe^\top[\widehat\psi'_\Fe]^{\bm\xi}(\Fe,\Fp,\theta){:}\Lp}\nonumber\\
&\qquad+J^{-1}\partial_{\Lp}\!\zeta(\theta;\Lp){:}\Lp+\nu_2|\nabla\Lp|^q-
{J^{-1}[\widehat\varphi']^{\bm\xi}(\Fe){:}\DT\Fe}\nonumber\\
&=\nu_0|\ee(\vv)|^p+\nu_1|\nabla\ee(\vv)|^p+J^{-1}\partial_{\Lp}\!\zeta(\theta;\Lp)
+\nu_2|\nabla\Lp|^q+J^{-1}[\gamma'_\Fe]^{\bm\xi}(\Fe,\theta){:}\DT\Fe.
\end{align}
Thus, \eqref{enBal2} becomes
\begin{subequations}\label{heat-eq-in-w}\begin{align}
	&\DT w+w\,{\rm div}\vv+{\rm div}\jj=\xi_{\rm d}+J^{-1}[\widehat\gamma_\Fe']^{\bm\xi}(\Fe,\theta){:}\DT\Fe\,,
\intertext{where %
the dissipation rate density is%
}
&\xi_{\rm d}=\nu_0|\ee(\vv)|^{p}+\nu_1|\nabla\ee(\vv)|^p+\nu_2|\nabla\Lp|^q+J^{-1}\partial_{\Lp}\!\zeta(\theta;\Lp){:}\Lp.	
\end{align}\end{subequations}
In the ensuing treatment we shall remove hats from the functions $\widehat\psi$, $\widehat\varphi$, $\widehat\gamma$ and $\widehat\phi$.

\subsection{The system formulated in terms of the reference mapping}

The plasticity evolution will be formulated exclusively
in terms of rate $\Lp$ from \eq{dafa-formula}. In \cite{Roub22QHLS}, 
the plastic distortion $\Fp$ was eliminated 
and, multiplying \eq{dafa-formula} by $\Fe$, an evolution rule
for $\Fe$ even without any explicit occurrence of $\Fee^{-1}$ was used, namely
\eqref{DTFe}.
Here, $\Fp$ is explicitly needed due to the hardening. The definition of $\Lp$
in \eqref{dafa-formula} allows to reconstruct $\Fp$ from $\Lp$ if an initial
condition for $\Fp$ is prescribed by means  of the evolution-and-transport
equation
\begin{align}
\DT\Fp=\Lp\Fp\,.
\label{evol-of-P}\end{align}
Then one can re-construct $\Fe$ from $\FF=\Fe\Fp$ and $\FF^{-1}=\nabla\bm{\xi}$, i.e.
\begin{align}
\Fe=(\Fp\nabla\bm{\xi})^{-1}\,.
\label{Fe-from-Fp-xi}\end{align}
We will then effectively eliminate $\Fe$ from the formulation by
use of \eqref{Fe-from-Fp-xi} just to abbreviate the expression
$(\Fp\nabla\bm{\xi})^{-1}$ occuring many times and making thus
the equations shorter.
Moreover, as now $\DT\Fe$ is not explicitly used, we
substitute \eqref{DTFe} into \eqref{heat-eq-in-w}.

Relying on \eqref{Fe-from-Fp-xi}, as $\FF$ is eliminated from the system
\eqref{system-mech} and \eqref{heat-eq-in-w}, using the
reference mapping $\bm\xi$ is not straightforward. One should thus
reformulate the system rather in terms of $(\vv,\bm\xi,\Fp,\Lp,\theta)$ as 

\begin{subequations}\label{Euler-hypoplast-xi}
\begin{align}
\nonumber
     &%
   \rhoRxi{\rm det}(\nabla\bm\xi)\DT\vv={\rm div}
     (\TT_{\rm e}{+}\DD)
     +\rho^{\bm\xi}_\text{\sc r}\det(\nabla\bm\xi)\GRAVITY%
     \,
     \\ \nonumber &\hspace*{10em}\text{with } \
    \TT_{\rm e}=
(\det\Fe)^{-1}\big([\varphi'_{\Fe}]^{\bm\xi}(\Fe){+}[\gamma'_{\Fe}]^{\bm{\xi}}(\Fe,\theta)\big)\Fetop
\ %
    \\
    &\hspace*{10em}\text{and }\ 
    \DD=\nu_0|\ee(\vv)|^{p-2}\ee(\vv)-{\rm div}(\nu_1|\nabla\EE(\vv)|^{p-2}\nabla\EE(\vv))
\,,
\label{Euler1=hypoplast-xi}
\\\label{Euler2=hypoplast-xi}
&\DT{\bm\xi}={\bm0}\,,
\\\label{Euler3=hypoplast-xi}
&\DT\Fp=\Lp\Fp\,,
\\\nonumber
&\pl_{\Lp}^{}\zeta(\theta;\Lp)-
\frac{{\rm div}(\nu_2|\Nabla\Lp|^{q-2}\Nabla\Lp)}{\det(\nabla\bm\xi)}\ni
{\rm dev}\Big(\Fetop[\varphi_{\Fe}']^{\bm\xi}(\Fe)
\\[-.5em]&\hspace{19em}+\Fetop[\gamma_{\Fe}']^{\bm\xi}(\Fe,\theta)-[\phi'_{\Fp}]^{\bm\xi}(\Fp)\Fp^\top\Big),\label{Euler4=hypoplast-xi}
\\
& \pdt\W={\rm div}(\kappa^{\bm\xi}(\Fe,\theta)\nabla\theta)-\W{\rm div}\,\vv
+\nu_0|\ee(\vv)|^{p}+\partial_{\Lp}\!\zeta(\theta;\Lp){:}\Lp\nonumber
+\nu_1|\nabla\ee(\vv)|^p
\\&\qquad\quad
+\nu_2|\nabla\Lp|^q
+\det(\nabla\bm\xi)[\COUPLING'_{\Fe}]^{\bm\xi}(\Fe,\theta){:}\big(\nabla\vv\Fe{-}\Fe\Lp\big)\quad\text{ with }\ \ \W=\OMEGA^{\bm\xi}(\Fe,\theta), 
\nonumber
\\& \hspace{3em}\text{ where }\ \ 
\omega(\XX, \Fe,\theta)=\frac{
\COUPLING(\XX,\Fe,\theta){-}\theta\COUPLING'_\theta (\XX,\Fe,\theta)}{\det\Fe}
\ \text{ and }\ \Fe=(\Fp\nabla\bm\xi)^{-1}.\label{Euler5=hypoplast-xi}
\end{align}\end{subequations}
Note that the second term on the left-hand side of \eqref{Euler4=hypoplast-xi} is deviatoric. Indeed, using commas to denote partial differentiation, we can write ${{\rm div}(\nu_2|\Nabla\Lp|^{q-2}\Nabla{\Lp})}_{ij}=\nu_2{\LL_{\rm p}}_{ij,kk}+\nu_2(|\nabla{\LL_{\rm p}}^{q-2}|)_{,k}{\LL_{\rm p}}_{ij,k}$, which is the sum of two deviatoric terms.

The above system is to be completed by suitable initial and 
boundary conditions. As to the initial conditions, we take:
\begin{equation}\label{ic}
\vv(0)=\vv_0,\ \ \ \ \bm\xi(0)=\bm\xi_0,\ \ \ \
\FF_{{\rm p}}(0)=\FF_{{\rm p}},\ \ \text{ and }\ \ \theta(0)={\theta_0}.
\end{equation}
Having in mind the so-called sticky-air approach, we consider here the
boundary conditions,
\begin{subequations}\label{Euler-hypoplast-xi-BC}
\begin{align}\label{Euler-hypoplast-xi-BC1}
&\vv
=\bm0,\ \ 
\Nabla\EE(\vv){{:}}(\nn{\otimes}\nn)={\bm0}\,,\ \ 
\nn{\cdot}\nabla\Lp={\bm0},\ \ \text{ and }\ \ \
\\
&
\nn{\cdot}\kappa(\Fe,\theta)\nabla\theta=h(\theta)\,.\label{eq:10}
\end{align}\end{subequations}
We note that the first condition fixes the boundary displacement, which allows to
conclude global invertibility of the reference map when using the results
of \cite{Ball81GISF}, although we will not use this result in the present paper.

\def\GM{M}
\def\varsigma{c_1}

\begin{example}[{\sl Neo-Hookean material}]\upshape
For illustration of the structure of the model,
let us consider the data
\begin{subequations}\begin{align}\label{neo-Hookean??}
&\varphi(\Fe)=\frac12K_\text{\sc e}^{}\big(J-1\big)^2+
\frac12G_\text{\sc e}^{}\big(J^{-2/d}{\rm tr}(\Fe\Fetop)-d\big)\ \ \text{ with }\ J=\det\Fe\,,
\\&\phi(\Fp)=\frac12H_\text{\sc e}^{}|\Fp|^2\,,
\\&\COUPLING(\Fe,\theta)=c\theta({\rm ln}\theta{-}1)-\varsigma J\theta^\alpha\ \text{ so that} \ \OMEGA(\Fe,\theta)=c\theta/J+\varsigma(\alpha{-}1)\theta^\alpha\ \text{ with }\ \alpha>1
\\&\zeta(\theta;\Lp)=\frac12\GM(\theta)|\Lp|^2\,.
                    \end{align}\end{subequations}
\begin{center}
\begin{my-picture}{35}{5.2}{fig1}
\psfrag{e}{\small $\FF$}
\psfrag{s}{\small $\sigma$}
\psfrag{f}{\small $\GRAVITY$}
\psfrag{s1}{\small $\sigma_{\rm sph}$}
\psfrag{s2}{\small $\sigma_{\rm dev}$}
\psfrag{e1}{\small $J^{-1/d}\FF$}
\psfrag{e2}{\footnotesize $J^{1/d}\bbI$}
\psfrag{e3}{\small $J^{-1/d}\Fe$}
\psfrag{C1}{\small$K_\text{\sc e}^{}$}
\psfrag{C2}{\small$G_\text{\sc e}^{}$}
\psfrag{C3}{\small$H_\text{\sc e}^{}$}
\psfrag{D1}{\small$G_\text{\sc v}^{}$}
\psfrag{D2}{\small$K_\text{\sc v}^{}$}
\psfrag{D3}{\small$\GM$}
\psfrag{q}{\small $\!\!\theta$}
\psfrag{a}{\small $\alpha$}
\psfrag{P}{\small $\Fp$}
\psfrag{spherical part}{\!\!\!\!\footnotesize\sf\begin{minipage}[t]{10em}\hspace*{0em}volumetric part
\\\hspace*{0em}(Kelvin-Voigt)
\end{minipage}}
\psfrag{deviatoric part}{\!\!\!\!\footnotesize\sf\begin{minipage}[t]{10em}\hspace*{0em}isochoric part
  \\[-.1em]\hspace*{0em}(4-parameter solid
  \\[-.1em]\hspace*{0em}or, if $H_\text{\sc e}^{}=0$,
 \\[-.1em]\hspace*{0em}Jeffreys' rheology)
\end{minipage}}
\hspace*{2em}\includegraphics[width=35em]{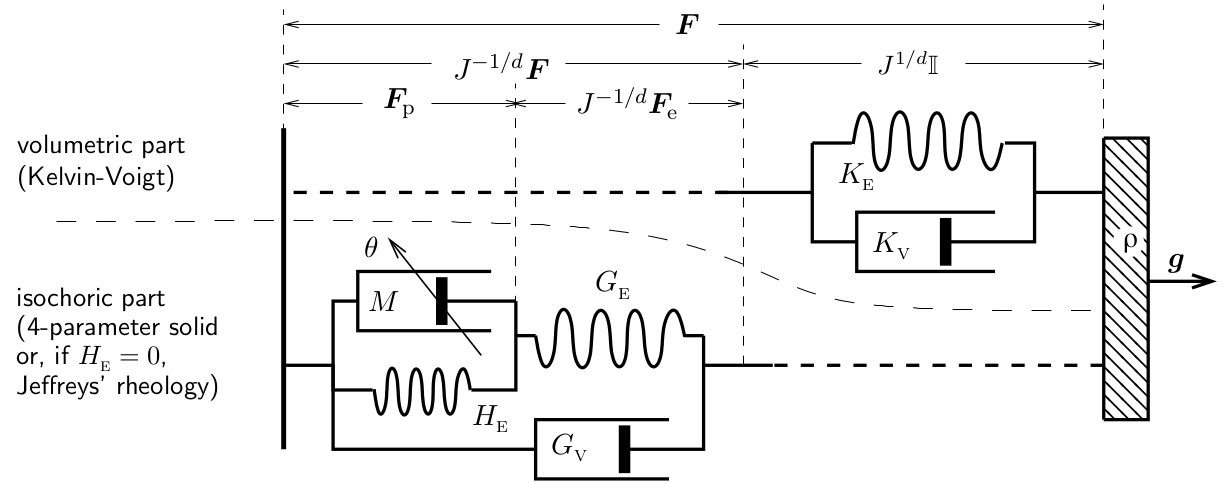}
\end{my-picture}
\nopagebreak
    {\small\sl\hspace*{-.5em}Fig.~\ref{fig1}:~\begin{minipage}[t]{34em}
A schematic 0-dimensional diagramme of the mixed rheology acting
differently on spherical (volumetric) and the deviatoric parts.
The Jeffreys rheology in the deviatoric part combines
Stokes' and Maxwell rheologies in parallel and may degenerate
to mere Stokes-type fluid if $\GM$ vanishes within melting.
\end{minipage}
} 
\end{center}
Then $\Fetop[\COUPLING_{\Fe}']^{\bm\xi}/\det\Fe=[\COUPLING_{\Fe}']^{\bm\xi}\Fetop/\det\Fe=\varsigma\theta^\alpha{\rm Cof}(\Fe)\Fetop/{\det}(\Fe)
=-\varsigma\theta^\alpha\bbI$
so that \eqref{Euler-ass-adiab+} below is satisfied.
Note that $(\det\FF)^{-2/d}{\rm tr}(\FF\FF^\top)=(\det\BB)^{-1/d}{\rm tr}\BB$ is the
first invariant of the isochoric part $(\det\BB)^{-1/d}\BB$ of the left
Cauchy-Green tensor $\BB=\FF\FF^\top$. Thus $\varphi$
is frame indifferent
because it depends on $\BB$ and the volumetric variations influences
only the $K_\text{\sc e}^{}$-term. %
Then, having the isochoric inelastic strain, the model separates the
volumetric (spherical) and the deviatoric parts so that it combines
the Kelvin-Voigt rheology in the volumetric part with the 
Jeffreys rheology (if there is no hardening) or a 4-parameter solid in the deviatoric part, cf.~Figure~\ref{fig1}.

\end{example}

\subsection{Energetics behind the system \eqref{Euler-hypoplast-xi}}
\label{sec-energetics}
The mechanical energy balance of the model can be seen when testing the
momentum equation by $\vv$ and the flow rule by $\Lp$, integrating over $\varOmega$ and summing up the results, as detailed below. Recalling that $J=\det\Fe$ since $\det\Fp=1$, using the algebra
$F^{-1}={\rm Cof}\,F^\top/\!\det\,F$ and the calculus $\det'(F)={\rm Cof}\,F$,
we can write 
\begin{align}\nonumber
	&J^{-1}\big[\varphi'_{\Fe}\big]^{\bm\xi}(\Fe)\Fetop= \frac{\big[\varphi'_{\Fe}\big]^{\bm\xi}(\Fe)}{\det\Fe}\Fetop
	=\ \frac{\big[\varphi'_{\Fe}\big]^{\bm\xi}(\Fe)-\varphi^{\bm\xi}(\Fe)\Fe^{-\top}\!\!\!\!}{\det\Fe}\Fetop\!+
	\frac{\varphi^{\bm\xi}(\Fe)}{\det\Fe}\bbI
	\\&\qquad=\bigg(\frac{\big[\varphi'_{\Fe}\big]^{\bm\xi}(\Fe)}{\det\Fe}
	-\frac{\varphi^{\bm\xi}(\Fe){\rm Cof}\Fe}{(\det\Fe)^2}\bigg)\Fetop\!\!+
	\frac{\varphi^{\bm\xi}(\Fe)}{\det\Fe}\bbI
	=\Big[\frac{\varphi^{\bm\xi}(\Fe)}{\det\Fe}\Big]'_{\Fe}\!\Fetop\!\!+\frac{\varphi^{\bm\xi}(\Fe)}{\det\Fe}\bbI\,.
	\label{referential-stress}\end{align}
A similar calculation yields
\begin{align}\label{eq:44}
	J^{-1}\Fetop\big[\varphi'_{\Fe}\big]^{\bm\xi}(\Fe) =\Fetop\Big[\frac{\varphi^{\bm\xi}(\Fe)}{\det\Fe}\Big]'_{\Fe}+\frac{\varphi^{\bm\xi}(\Fe)}{\det\Fe}\bbI\,.
	\end{align}
Using the calculation in \eq{referential-stress}, we can write%
\begin{align}\nonumber
	&\int_\varOmega{\rm div}\,\TT_{\rm e}{\cdot}\vv\,\d\xx
	=\!\int_\varGamma\underbrace{(\TT_{\rm e}\nn){\cdot}\vv}_{\displaystyle=0}\,\d S-\!\int_\varOmega\!\TT_{\rm e}{:}\ee(\vv)\,\d\xx
	\\&\quad\nonumber=
	-\!\int_\varOmega\!\Big(\frac{[\varphi']^{\bm\xi}(\Fe)\!}{\det\Fe}
	+\frac{[\COUPLING_\Fe']^{\bm\xi}(\Fe,\theta)}{\det\Fe}\Big)\Fetop{:}\ee(\vv)\,\d\xx
	\\&\quad\nonumber
	=
	-\!\int_\varOmega\!\Big(\Big[\frac{\varphi^{\bm\xi}(\Fe)}{\det\Fe}\Big]_{\Fe}'\!\!\Fetop
	+\frac{\varphi^{\bm\xi}(\Fe)}{\det\Fe}\bbI
+\frac{[\COUPLING_\Fe'](\Fe,\theta)}{\det\Fe}\Fetop\Big){:}\ee(\vv)\,\d\xx
	\\&\quad
	=
	-\!\int_\varOmega\!\Big[\frac{\varphi^{\bm\xi}(\Fe)}{\det\Fe}\Big]_{\Fe}'\!\!
	{:}(\Nabla\vv)\Fe
	+\frac{\varphi^{\bm\xi}(\Fe)}{\det\Fe\!}\,{\rm div}\,\vv
+\frac{[\COUPLING_\Fe']^{\bm\xi}(\Fe,\theta)\Fetop\!\!\!}{\det\Fe}{:}\ee(\vv)\,\d\xx.
	\label{eq:22}\end{align}
The test of the flow rule by $\Lp$ yields the following term arising from the
Mandel stress:
\begin{align}\nonumber
&\int_{\varOmega} J^{-1}{\rm dev}\big(\Fetop{[\varphi'_{\Fe}]^{\bm\xi}(\Fe)}
\big){:}\Lp\,\d\xx
=\int_{\varOmega} J^{-1}\Fetop{[\varphi'_{\Fe}]^{\bm\xi}(\Fe)}\big){:}\Lp\,\d\xx
\\&\qquad=\int_\varOmega\Big(\Fetop\Big[\frac{\varphi^{\bm\xi}(\Fe)}{\det\Fe}\Big]_{\Fe}'\!\!
+\frac{\varphi^{\bm\xi}(\Fe)}{\det\Fe}\bbI\Big){:}\Lp\,\d\xx
=\int_\varOmega\Big[\frac{\varphi^{\bm\xi}(\Fe)}{\det\Fe}\Big]_{\Fe}'\!\!
{:}\Fe\Lp\,\d\xx,
\label{eq:43}\end{align}
where \eqref{eq:44} has been used and the fact that $\Lp{:}\bbI=0$.
When the balance of momentum tested by $\vv$ is added to the flow rule
tested by $\Lp$, the terms in \eqref{eq:22} and \eqref{eq:43} sum up and
can be manipulated as follows:
\begin{align}\nonumber
	&\int_\varOmega\!\Big[\frac{\varphi^{\bm\xi}(\Fe)}{\det\Fe}\Big]'_{\Fe}
	{:}\big(\hspace*{-.7em}\lineunder{(\Nabla\vv)\Fe-\Fe\Lp}{$=\DT\Fe$ due to \eqref{DTFe}\!\!}\hspace*{-.7em}\big)
	+\frac{\varphi^{\bm\xi}(\Fe)}{\det\Fe\!}\,{\rm div}\,\vv\,\d\xx
	\nonumber\\[-.3em]
	&\nonumber\hspace*{1em}=\int_\varOmega\!\Big[\frac{\varphi^{\bm\xi}(\Fe)}{\det\Fe}\Big]'_{\Fe}
	{:}\Big(\pdt{\Fe}+(\vv{\cdot}\nabla)\Fe\Big)
+\frac{\varphi^{\bm\xi}(\Fe)}{\det\Fe\!}\,{\rm div}\,\vv
       +\frac{\big[\varphi'_{\XX}\big]^{\bm\xi}(\Fe)}{\det\Fe}{\cdot}\Big(\hspace*{-.7em}\lineunder{\pdt{\bm\xi}+(\vv{\cdot}\nabla)\bm\xi}{$=0$ due to \eqref{Euler2=hypoplast-xi}\!\!}\hspace*{-.7em}\Big)\,\d\xx
	\\[-.9em]
	&\nonumber\hspace*{1em}
	=\frac{\d}{\d t}\int_\varOmega\frac{\varphi^{\bm\xi}(\Fe)}{\det\Fe}\,\d\xx
	+\!\int_\varOmega\!\nabla\Big(\frac{\varphi^{\bm\xi}(\Fe)}{\det\Fe}\Big){\cdot}\vv+\frac{\varphi^{\bm\xi}(\Fe)}{\det\Fe\!}\,{\rm div}\,\vv\,\d\xx
	\\&\hspace*{1em}
	=\frac{\d}{\d t}\int_\varOmega\frac{\varphi^{\bm\xi}(\Fe)}{\det\Fe}\,\d\xx
	+\!\int_\varGamma\frac{\varphi^{\bm\xi}(\Fe)}{\det\Fe}(\hspace*{-.7em}\lineunder{\vv{\cdot}\nn}{$=0$}\hspace*{-.7em})\,\d S\,.
\end{align}
For the term containing the hyperstress
$\mathfrak{H}=\nu_1|\nabla\EE(\vv)|^{p-2}\nabla\EE(\vv)$, we obtain
\begin{align}
\int_{\Omega} {\rm div}^{2} \mathfrak{H}{\cdot}\vv \dx &=\int_{\varGamma} \vv{\cdot}{\rm div} \mathfrak{H}\,\nn\d S-\int_{\Omega} \operatorname{div} \mathfrak{H}{:}\nabla \vv \dx \nonumber\\
&=\int_{\Omega} \mathfrak{H} \vdots \nabla^{2} \vv \dx+\int_{\varGamma} \mathfrak{H}\Vdots (\nabla \vv{\otimes}\nn)-\vv{\cdot}\operatorname{div} \mathfrak{H} \nn\d S \nonumber \\
&=\int_{\Omega} \mathfrak{H} \vdots \nabla^{2} \vv \dx+\int_{\varGamma} \mathfrak{H}\Vdots(\partial_{\nn}\vv{\otimes}\nn {\otimes} \nn)+\mathfrak{H}\Vdots(\nabla_{\mathrm{S}} \vv{\otimes}\nn)-\vv{\cdot}\operatorname{div} \mathfrak{H} \nn\d S  \nonumber\\
&=\int_{\Omega} \mathfrak{H} \vdots \nabla^{2}\vv\dx+\int_{\varGamma} \mathfrak{H}\Vdots(\partial_\nn\vv{\otimes}\nn {\otimes} \nn)-\left(\operatorname{div}_{\mathrm{S}}(\mathfrak{H}\nn)+\operatorname{div} \mathfrak{H} \nn\right){\cdot}\vv\d S \nonumber\\
&=\int_\Omega \nu_1|\nabla \boldsymbol{e}(\vv)|^{p}\,\d\xx+\int_\varGamma\,\underbrace{\nabla(\ee(\vv)){:}(\nn{\otimes}\nn)}_{\displaystyle =0}{\cdot}\partial_\mm\vv{\rm d}S=0.\label{hyper}
\end{align}
In the last equality in \eqref{hyper}, the boundary integrals vanish because
of the first and the second of the boundary conditions
$\eqref{Euler-hypoplast-xi-BC}$. Again using that $\Lp$ is trace-free, we also
compute the contribution from the hardening energy $\phi(\Fp)$ from the flow
rule as follows
\begin{align}
&\int_\varOmega J^{-1}{\rm dev}\big(\Fp^\top[\phi'_{\Fp}]^{\bm\xi}(\Fp)\big){:}\Lp\,\d\xx=\int_\varOmega J^{-1}\Fp^\top[\phi'_{\Fp}]^{\bm\xi}(\Fp){:}\Lp\,\d\xx
\nonumber
\\
&\qquad\qquad=\int_\varOmega J^{-1}[\phi'_{\Fp}]^{\bm\xi}(\Fp){:}\DT\Fp\,\d\xx
\!\!\stackrel{\eqref{Euler2=hypoplast-xi}}=\!\!\!\int_\varOmega J^{-1}[\phi'_{\Fp}]^{\bm\xi}(\Fp){:}\DT\Fp+J^{-1}[\phi'_{\XX}]^{\bm\xi}(\Fp){\cdot}\DT{\bm\xi}\,\d\xx\nonumber
\\
&\qquad\qquad=\int_\varOmega\bigg(J^{-1}[\phi'_{\Fp}]^{\bm\xi}(\Fp){:}\pdt\Fp+J^{-1}[\phi'_{\XX}]^{\bm\xi}(\Fp){\cdot}\pdt{\bm\xi}\nonumber
\\[-.5em]&\hspace{11em}
+J^{-1}[\phi'_{\Fp}]^{\bm\xi}(\Fp){:}(\vv{\cdot}\nabla)\Fp+J^{-1}[\phi'_{\XX}]^{\bm\xi}(\Fp){\cdot}(\vv{\cdot}\nabla){\bm\xi}\bigg)\,\d\xx\nonumber
\\[-.5em]
&\qquad\qquad=\int_\varOmega J^{-1}\pdt\phi^{\bm\xi}(\Fp)+J^{-1}(\vv{\cdot}\nabla)\phi^{\bm\xi}(\Fp)\,\d\xx\nonumber\\
&\qquad\qquad=\int_\varOmega \pdt{}\big({J^{-1}\phi^{\bm\xi}(\Fp)}\big)-\pdt{J^{-1}}\phi^{\bm\xi}(\Fp)+J^{-1}(\vv{\cdot}\nabla)\phi^{\bm\xi}(\Fp)\,\d\xx
\nonumber\\
&\qquad\qquad=\frac{{\rm d}}{{\rm d}t}\int_\varOmega {J^{-1}\phi^{\bm\xi}(\Fp)}\,\d\xx-\int_\varOmega\pdt{J^{-1}\!\!}\,\phi^{\bm\xi}(\Fp)-J^{-1}(\vv{\cdot}\nabla)\phi^{\bm\xi}(\Fp)\,\d\xx
\nonumber\\
&\qquad\qquad\!\!\stackrel{\eqref{DT-det-1+}}=\!\!\frac{\rm d}{{\rm d}t}\int_\varOmega {J^{-1}\phi^{\bm\xi}(\Fp)}\,\d\xx+\int_\varOmega{\rm div}(J^{-1}\vv)\phi^{\bm\xi}(\Fp)+J^{-1}(\vv{\cdot}\nabla)\phi^{\bm\xi}(\Fp)\,\d\xx\nonumber\\
&\qquad\qquad
=\frac{\rm d}{{\rm d}t}\int_\varOmega {J^{-1}\phi^{\bm\xi}(\Fp)}\,\d\xx+\int_\varOmega {\rm div}(J^{-1}\phi^{\bm\xi}(\Fp)\vv))\,\d\xx\nonumber
\\&\qquad\qquad
=\frac{\rm d}{{\rm d}t}\int_\varOmega {J^{-1}\phi^{\bm\xi}(\Fp)}\,\d\xx+\int_\varGamma J^{-1}\phi^{\bm\xi}(\Fp)\underbrace{\vv{\cdot}\nn}_{\displaystyle=0}\,\d\xx.
\nonumber\end{align}
Finally, the test of the inertial term yields
\begin{align}
\nonumber\int_\varOmega \rhoRxi{\rm det}(\bm\xi)\DT\vv{\cdot}\vv\,\d\xx&=	\int_\varOmega\rhoRxi{\rm det}(\nabla\bm\xi)\vv{\cdot}\Big(\frac{\partial\vv}{\partial t}+\vv{\cdot}\nabla\vv\Big)\,\d\xx\\
&=\int_\varOmega\rhoRxi{\rm det}(\nabla\bm\xi)\vv{\cdot}\frac{\partial\vv}{\partial t}+\rhoRxi{\rm det}(\nabla\bm\xi)\vv{\cdot}\nabla\frac{|\vv|^2}{2}\,\d\xx\nonumber\\
&=\int_\varOmega\rhoRxi{\rm det}(\nabla\bm\xi)\vv{\cdot}\frac{\partial\vv}{\partial t}-{\rm div}(\rhoRxi{\rm det}(\nabla\bm\xi)\vv)\frac{|\vv|^2}2\,\d\xx\nonumber\\
&=\int_\varOmega\rhoRxi{\rm det}(\nabla\bm\xi)\vv{\cdot}\frac{\partial\vv}{\partial t}+\frac{\partial}{\partial t}(\rhoRxi{\rm det}(\nabla\bm\xi))\frac{|\vv|^2}{2}\,\d\xx\nonumber\\
&=\int_\varOmega\pdt{}\big(\rhoRxi{\rm det}(\nabla\bm\xi)\frac{|\vv|^2}2\big)\,\d\xx\nonumber
=\frac{{\rm d}}{{\rm d}t}\int_\varOmega\rhoRxi{\rm det}(\nabla\bm\xi)\frac{|\vv|^2}2\,\d\xx.
\end{align}

We thus obtain the {\it energy-dissipation balance}:%
\begin{align}
\frac{\rm d}{{\rm d}t}&\int_\varOmega \hspace*{-.7em}\linesunder{J^{-1}\rhoRxi\frac{|\vv|^2}2}{kinetic}{energy}\hspace*{-.7em}+
\hspace*{-.7em}\linesunder{J^{-1}\varphi^{\bm \xi}(\Fe)}{strain}{energy}\hspace*{-.7em}
+\hspace*{-.7em}\linesunder{J^{-1}\phi^{\bm\xi}(\Fp)}{hardening}{energy} \,\d\xx
\nonumber
\\\nonumber
&+\int_\varOmega%
\hspace*{-.7em}\linesunder{\nu_1|\nabla\ee(\vv)|^p+\nu_0|\ee(\vv)|^p}{dissipation rate due to}{Stokes (hyper)viscosity}\hspace*{-.7em}+\hspace*{-.7em}\linesunder{J\partial_{\Lp}\!\zeta(\theta;\Lp){:}\Lp\!+\nu_2|\nabla\Lp|^q}{dissipation rate due}{to plastification}\hspace*{-.5em}%
\,\d\xx
\\
&\hspace{9em}=\int_\varOmega\hspace*{-.7em}\linesunder{J^{-1}\varrho_\text{\sc r}^{\bm\xi}\,\bm g{\cdot}\vv}{power of}{gravity}\hspace*{-.7em}
-\hspace*{-.7em}\linesunder{J^{-1}[\gamma'_{\Fe}]^{\bm\xi}(\Fe,\theta){:}(\nabla\vv\Fe{-}\Fe\Lp)}{power of adiabatic}{effects}\hspace*{-.7em}\,\d\xx.
\label{energetics1}\end{align}
When \eqref{energetics1} is added to the heat equation
\eqref{Euler5=hypoplast-xi}, the dissipative terms and the adiabatic term
cancel, and using the identity
\begin{align}
\int_\varOmega\DT w+w\,{\rm div}\,\vv\,\d\xx&=\int_\varOmega\pdt w+\vv{\cdot}\nabla w	+w\,{\rm div}\,\vv\,\d\xx\nonumber\\
&=\frac{\rm d}{{\rm d}t}\int_\varOmega w\,\d\xx+\int_\varOmega{\rm div}(w\vv)\,\d\xx\nonumber
=\frac{\rm d}{{\rm d}t}\int_\varOmega w\,\d\xx+\int_\varGamma\underbrace{w\vv{\cdot}\nn}_{\displaystyle =0}\,\d S\,,
\end{align}
we obtain the {\it total-energy balance}:
\begin{align}
\frac{\rm d}{{\rm d}t}\int_\varOmega\hspace*{-.7em}\linesunder{J^{-1}\rhoRxi\frac{|\vv|^2}2}{kinetic}{energy}\hspace*{-.7em}+
\hspace*{-.7em}\linesunder{J^{-1}\varphi^{\bm \xi}(\Fe)}{stored}{energy}\hspace*{-.7em}+
\hspace*{-.7em}\linesunder{\omega^{\bm\xi}(\Fe,\theta)}{heat part of}{internal energy}\hspace*{-.7em}\,\d\xx=\int_\varGamma \hspace*{-.7em}\linesunder{h(\theta)}{heat}{influx}\hspace*{-.7em}\d S+\int_\varOmega \hspace*{-.7em}\linesunder{J^{-1}\varrho_\text{\sc r}^{\bm\xi}\,\bm g{\cdot}\vv}{power of}{gravity}\hspace*{-.7em}\,\d\xx.\label{tot-energy-balance}
\end{align}
We recall that in the above equations $J^{-1}={\rm det}(\nabla\bm\xi)=1/{\rm det}(\Fe)$.

\section{Analysis: existence of weak solutions}\label{sec-anal}
We will use the standard notation concerning the Lebesgue and the Sobolev
spaces, namely $L^p(\varOmega;\R^d)$ for Lebesgue measurable functions
$\varOmega\to\R^d$ whose Euclidean norm is integrable with $p$-power, and
$W^{k,p}(\varOmega;\R^d)$ for functions from $L^p(\varOmega;\R^d)$ whose
all derivative up to the order $k$ have their Euclidean norm integrable with
$p$-power. We also write briefly $H^k=W^{k,2}$. The notation
$p^*$ will denote the exponent from the embedding
$W^{1,p}(\varOmega)\subset L^{p^*}(\varOmega)$, i.e.\ $p^*=dp/(d{-}p)$
for $p<d$ while $p^*\ge1$ arbitrary for $p=d$ or $p^*=+\infty$ for $p>d$.
Moreover, for a Banach space
$X$ and for $I=[0,T]$, we will use the notation $L^p(I;X)$ for the Bochner
space of Bochner measurable functions $I\to X$ whose norm is in $L^p(I)$
while $W^{1,p}(I;X)$ denotes for functions $I\to X$ whose distributional
derivative is in $L^p(I;X)$. Also, $C(\cdot)$ and $C^1(\cdot)$
will denote spaces of continuous and continuously differentiable functions. Moreover, as usual, we will use $C$ for a generic constant which may vary
from estimate to estimate.

\begin{definition}[\sl Weak solutions to  \eqref{Euler-hypoplast-xi}]\label{def:1}
Let $p,q,r>1$. We say that a quintuple $(\vv,\bm\xi,\Lp,\Fp,\theta)$ with $\vv\in L^\infty(I;L^2(\varOmega;\R^d))\cap L^p(I;W^{2,p}(\varOmega;\R^d))$, $\bm\xi\in L^\infty(I;W^{2,r}(\varOmega;\R^d))\cap W^{1,1}(I;L^r(\varOmega;\R^d))$,  $\Lp\in L^q(I;W^{1,q}(\varOmega;\R^{d\times d}))$ with
${\rm div}(\nu_2|\Nabla\Lp|^{q-2}\Nabla\Lp)\in L^2(I{\times}\varOmega;\R^{d\times d})$, 
 $\Fp\in L^\infty(I;W^{1,r}(\varOmega;\R^{d\times d}))\cap W^{1,1}(I;L^r(\varOmega;\R^{d\times d}))$ and $\theta\in L^1(I;W^{1,1}(\varOmega))$ is a weak solution of system \eqref{Euler-hypoplast-xi} with the initial conditions \eqref{ic} and with the boundary conditions \eqref{Euler-hypoplast-xi-BC} if: 
\begin{itemize}
\item [(i)] \eqref{Euler1=hypoplast-xi} holds in the weak sense, specifically,
\begin{align}
&\int_0^T\!\!\int_\varOmega\bigg(\Big(\operatorname{det}(\nabla \bm\xi)\big(\big[\varphi_{\FF_{\mathrm{e}}}^{\prime}\big]^{\bm\xi}(\FF_{\mathrm{e}})+\big[\gamma_{\FF_{\mathrm{e}}}^{\prime}\big]^{\bm\xi}(\FF_{\mathrm{e}},\theta)\big) \FF_{\mathrm{e}}^{\top}+\nu_0|\ee(\vv)|^{p-2} \ee(\vv)
\nonumber\\
&-\det(\nabla\bm\xi)\rhoRxi\vv{\otimes}\vv\Big){:}\ee(\widetilde\vv)
+\nu_1|\nabla\ee(\vv)|^{p-2}\nabla\ee(\vv)\vdots\nabla\ee(\widetilde\vv)\nonumber
-\det(\nabla\bm\xi)\rhoRxi\vv\cdot\pdt{\widetilde\vv}\bigg)\dx\dt\\
&=\int_0^T\!\!\int_\varOmega\det(\nabla\bm\xi){\rhoRxi}\gg\cdot\widetilde\vv+\int_\varOmega\det(\nabla\bm\xi_0)\rhoRxizero\vv_0\cdot\widetilde\vv(0)\dx
\end{align}
for any smooth $\widetilde\vv$ such that $\widetilde\vv=0$ on $\varGamma$ and
$\widetilde\vv(T)=0$, with $\Fe=(\nabla\bm\xi)^{-1}\Fp$; 
\item [(ii)] (\ref{Euler-hypoplast-xi}b--d) hold a.e.\ in $I{\times}\varOmega$,
the first two initial conditions in \eqref{ic} hold a.e.\ in $\varOmega$, and
the boundary conditions \eqref{Euler-hypoplast-xi-BC1} hold a.e.\ on
$\varSigma=I{\times}\varGamma$;
\item [(iii)] the heat equation \eqref{Euler5=hypoplast-xi}, with the boundary condition \eqref{eq:10} and the initial condition for $\theta$ in \eqref{ic}, hold in the weak sense, specifically that
\begin{align}
  \nonumber &\int_0^T\!\!\int_\varOmega\bigg(\omega^{\bm\xi}(\Fe,\theta)\pdt{\widetilde\theta}+\big(\omega^{\bm\xi}(\Fe,\theta)\vv-\kappa^{\bm\xi}(\Fe,\theta)\nabla\theta\big){\cdot}\nabla\widetilde\theta+\Big(\nu_0|\ee(\vv)|^p+\nu_1|\nabla\ee(\vv)|^p\\
  \nonumber  &\quad +\partial_{\Lp}\!\zeta(\theta;\Lp) +\nu_2|\nabla\Lp|^q+\det(\nabla\bm\xi)[\gamma'_{\Fe}]^{\bm\xi}(\Fe,\theta):(\nabla\vv\Fe{-}\Fe\Lp)\Big)\widetilde\theta\bigg)\dx\dt\\
  &\quad+\int_0^T\!\!\int_\varGamma h(\theta)\widetilde\theta\d S\dt+\int_\varOmega\omega(\FF_{\rm e,0},\theta_0)\widetilde\theta(0)\dx=0
\end{align}
for all $\widetilde\theta$ smooth such that $\widetilde\theta(T)=0$, where $\FF_{\rm e,0}=\nabla\bm\xi_0^{-1}\FF^{-1}_{\rm p 0}$.
\end{itemize}
\end{definition}

In writing the weak form of \eqref{Euler1=hypoplast-xi} we have used the
identity
\begin{equation*}
\rho\DT\vv=\rho\pdt\vv+\rho(\vv\cdot\nabla)\vv=\pdt{(\rho\vv)}+{\rm div}(\rho\vv)\vv+\rho(\vv\cdot\nabla)\vv=\pdt{(\rho\vv)}+{\rm div}(\rho\vv\otimes\vv)\,,
\end{equation*}
where $\rho=\rhoRxi{\det}(\nabla\bm\xi)$. Indeed, $\rho$ obeys automatically
the local form of mass conservation, as can be verified from the following
calculation:
\begin{align} 
 \nonumber \frac{\partial}{\partial t}\rho_{ }&=\pdt{}(\rhoRxi{\rm det}(\nabla\bm\xi_{ }))=
  [[\rhoR]_\XX']^{\bm\xi_{ }}{\rm det}(\nabla\bm\xi)\pdt{\bm\xi_{ }}+\rhoRxi{\rm Cof}(\nabla\bm\xi_{ }){:}\pdt{}\nabla\bm\xi_{ }\\
  \nonumber 
  & 
=-%
[[\rhoR]_\XX']^{\bm\xi_{ }}{\rm det}(\nabla\bm\xi_{ })(\vv_{ }
{\cdot}\nabla)\bm\xi-\rhoRxi{\rm Cof}(\nabla\bm\xi_{ }){:}(\nabla\bm\xi)\nabla\vv_{ }-\rhoRxi{\rm Cof}(\nabla\bm\xi)(\vv_{ }{\cdot}\nabla)\nabla\bm\xi\\
  &=-(\vv{\cdot}\nabla)(\rhoRxi{\det}(\nabla\bm\xi))
 -\rhoRxi\det(\nabla\bm\xi){\rm div}\vv_{ }
  =-{\rm div}(\rho_{ }\vv_{ }).\label{mass-conserv}
\end{align}	
Note that to obtain \eqref{mass-conserv} we have used \eqref{dtnabla} and
the formula
${\rm Cof}(\nabla\bm\xi)={\rm det}(\nabla\bm\xi)\nabla\bm\xi^{-\top}$. 

For analytical reasons, specifically the estimate
(see \eqref{Euler-thermodynam3-test2} below) of the convective
term in the heat equation,
we adopt a rather special form of the coupling part of the free energy. Namely,
given $1<\alpha\le2$ %
and $\varsigma>0$, we assume:
\begin{align}\label{ass-split-omega}
\COUPLING(\XX,\Fe,\theta)=\widetilde\COUPLING(\XX,\Fe,\theta)-\varsigma\theta^\alpha\det\Fe
\end{align}
and we let
\begin{equation}\label{split-omega}
	\varpi(\XX,\Fe,\theta)=\frac{\widetilde\COUPLING(\XX,\Fe,\theta)-\theta\widetilde\COUPLING'_\theta(\XX,\Fe,\theta)}{\det\Fe}.
\end{equation}
As a result, we have the representation
\begin{equation}\label{representation-omega}
	\OMEGA(\XX,\Fe,\theta)=\varpi(\XX,\Fe,\theta)+\varsigma(\alpha{-}1)\theta^\alpha
\end{equation}
{with} $\varpi$ {growing} at most linearly with respect
to $\theta$. This means that the actual
heat capacity $\OMEGA_\theta'(\XX,\Fe,\theta)$
contains an $(\XX,\Fe)$-independent contribution with an
$(\alpha{-}1)$-polynomial growth in $\theta$.

We can now summarize the data
qualification. For some $\delta>0$, $\epsilon>0$, $\max(d,1+\alpha)<r<\infty$, $1<\alpha<{2}$, $d<q<\infty$ and $d<p<\infty$, we assume:
\begin{subequations}\label{Euler-ass+}\begin{align}
&\varOmega\ \text{ smooth bounded domain of $\R^d$, }\ d=2,3
\\&\label{Euler-ass-phi+}
\varphi\in C^1({\rm GL}^+(d)),\ \ \forall\Fe\,{\in}\,{\rm GL}^+(d):\ \
\varphi(\Fe)\ge\delta(1{+}\det\Fe)
\\&
\COUPLING\,{\in}\, C^2({\rm GL}^+(d){\times}\R^+),\ \forall(\Fe,\theta)\,{\in}\,{\rm GL}^+(d){\times}\R^+:\ 
\nonumber\\&\hspace{5em}\max\bigg(\Big|\frac{\COUPLING_{\Fe}'(\Fe,\theta)\Fetop}{\det\Fe}\Big|,
\Big|\frac{\Fetop\COUPLING_{\Fe}'(\Fe,\theta)}{\det\Fe}\Big|\bigg)
\le C\Big(1+\frac{\varphi(\Fe)}{\det\Fe}+\theta^\alpha\Big)\,,
\label{Euler-ass-adiab+}
\\[-.0em]&\nonumber
\exists c_{\rm m}^{}
\in C({\rm GL}^+(d))\text{ positive such that}\ \forall
(\XX,\Fe,\theta)\in\Omega\times{\rm GL}^+(d){\times}\R^+:\ \ \ 
\\[-.2em]&\qquad
0\ge\COUPLING_{\theta\theta}''(\XX,\Fe,\theta)\ge -c_{\rm m}^{}(\Fe)/\theta
\ \ \text{ and }\ \
\COUPLING(\XX,\Fe,\theta){-}\theta\COUPLING_\theta'(\XX,\Fe,\theta)\ge
0\,,
\label{Euler-ass-c-Euler-ass-W+}
\\[-.0em]&\nonumber
{\zeta\in C(\R^+{\times}\R_{\rm dev}^{d\times d}):\ \forall\theta\in\R^+\  
{\zeta(\theta;\cdot)\in C^1(\R_{\rm dev}^{d\times d}\setminus\{0\})}\ }
\text{ convex with }\zeta(\theta;0)=0\ \text{ and }
\\&\hspace{2em}
\exists\,c_\zeta>0,\ \zeta_\text{\sc m}\in C^1(\R_{\rm dev}^{d\times d})
 \ \forall(\theta,\Lp)\in\R^+{\times}\R_{\rm dev}^{d\times d}:\ \
\zeta_\text{\sc m}(\Lp)\ge\zeta(\theta;\Lp)\ge c_\zeta|\Lp|^2\,,
\label{Euler-ass-zeta}
\\[-.0em]&
\forall \Lp\in\mathbb R^{d\times d}_{\rm dev},\  t\mapsto\zeta(t\boldsymbol{L}_{\rm p})\text{ is differentiable at }t=1\,,
\\[-.0em]&
{\phi\in C(\varOmega{\times}\R_{\rm dev}^{d\times d})\ \text{ with }\
\phi_{\Fp}'\in C(\varOmega{\times}\R_{\rm dev}^{d\times d};\R_{\rm dev}^{d\times d})
\ \text{ and }\ \phi\ge0\,,}
\label{ass-hard}
\\[-.0em]&\nonumber
\forall K{\subset}\,{\rm GL}^+(d)\text{ compact }\ \exists C_K<\infty
\ \ \forall \Fe\in K\,,\XX\in\Omega:\ \ \ 
|\varpi_\theta'(\XX,\Fe,\theta)|\le  C_K\,,\ \
\\&\hspace{16.7em}
|\varpi_{(\XX,\Fe)}'(\XX,\Fe,\theta)|\le  C_K(1{+}\theta^{\alpha})\,,\nonumber
\\&\hspace{16.7em}
|\varpi_{(\XX\theta,\Fe\theta)}''(\XX,\Fe,\theta)|\le  C_K(1+\theta^{\alpha-1})\,,
\label{Euler-ass-primitive-c+}
\\&\kappa\in C(\varOmega{\times}{\rm GL}^+(d){\times}\R^+)\ \text{ bounded},\ \ \
\inf\kappa(\varOmega{\times}{\rm GL}^+(d){\times}\R^+)>0\,,\label{Euler-ass-k}
\\&\nonumber
h:I{\times}\varGamma{\times}\R^+\to\R\ \text{ Carath\'eodory function},\ \ \
\theta\,h(t,\xx,\theta)\le C(1{+}\theta)\ \ \text{ and}
\\[-.1em]&\hspace{11em}
h(t,\xx,\theta)\le h_{\max}(t,\xx)\ \ \text{ for some }\ h_{\max}\in L^1(I{\times}\varGamma)\,,
\label{Euler-ass-h+}
\\&
\GRAVITY\in L^1(I;L^\infty(\varOmega;\R^d))\,,
\label{Euler-ass-f-g}\\
&\vv_0\in L^2(\varOmega;\R^d)\label{Euler-ass-v0}
\\&\bm\xi_0\in W^{2,r}(\varOmega;\R^{d})\,,\ \ r>d{\ \text{ and }\ r\ge\alpha{+1}}\,,\ \ \ \text{ with }\ \ \
\min_{\varOmega}^{}\det(\nabla\bm\xi_0)>0\,,
\label{Euler-ass-Fe0++}
\\&\bm\FF_{\rm p,0}\in W^{1,r}(\varOmega;\R^{d})\,,\ \ r>d\,,\ \ \
\text{ with }\ \ \ \det{\FF_{\rm p,0}}=1\,,
\label{Euler-ass-Fp0}
\\&\rhoR\in C(\barvarOmega)\,,\quad\text{with }\ \min_\varOmega^{}{\rhoR}>0,%
\label{Euler-ass-rhoR}
\\&\theta_0\in L^1(\varOmega),\ \ \ \theta_0\ge0\ \text{ a.e.\ on }
\end{align}\end{subequations}
Note that \eq{ass-hard} admits a model without any hardening.
  
Since $\det(\nabla\bm\xi_0)$ is positive almost everywhere, and since $\FF_{\rm p,0}$ has unit determinant, both matrices are invertible. Using the identity $\mathbb A^{-1}={\rm Cof}(\mathbb A)^{\top}\!/\!\det(\mathbb A)$ for every invertible matrix $\mathbb A$, we obtain from \eqref{Euler-ass-Fe0++} and \eqref{Euler-ass-Fp0} that
\begin{equation}
(\nabla\bm\xi_0^{-1})\FF_{\rm p\, 0}^{-1}\in W^{1,r}(\varOmega;\R^{d\times d}).\label{Euler-ass-Fe0}
\end{equation}
                        
\begin{theorem}
Let \eqref{ass-split-omega} and \eqref{Euler-ass+} hold. Then:\\
\Item{(i)}{there exists a weak solution of \eqref{Euler-hypoplast-xi} with the boundary conditions \eqref{Euler-hypoplast-xi-BC} and the initial conditions \eqref{ic} in the sense of Definition \ref{def:1}.
Such solution satisfies the additional regularity $\pdt{}\bm\xi\in L^p(I;L^\infty(\varOmega;\R^d)$,  $\pdt{}\bm F_{\rm e}\in L^{\min(p,q)}(I;L^r(\varOmega;\R^{d\times d}))$,   $\theta\in L^\infty(I;L^\alpha(\varOmega))\cap L^\mu(I;W^{1,\mu}(\varOmega))$ with $1\le \mu<(d{+}2{\alpha})/(d{+}{\alpha})$.}
\Item{(ii)}{Moreover, the dissipation-energy balance \eqref{energetics1} and
the total-energy balance \eqref{tot-energy-balance}
integrated over a time interval $[0,t]$ with $t\in I$ hold.}
\end{theorem}

\begin{proof}
For clarity of exposition, we divide the proof in eight steps.

\medskip

\noindent \emph{Step 1: Formal a-priori estimates.}
As a preliminary step, we are going show that every solution  must satisfy
\begin{equation}\label{cutoff}
{\rm ess}\,{\rm sup}_{I\times\Omega}^{}|\Fe|\le\lambda\qquad\text{and}\qquad
{\rm ess}\,{\rm inf}_{I\times\Omega}^{}\det\Fe\ge \frac 1 \lambda,	
\end{equation}
for some $\lambda\le 1$. Relying on this fact, we will in Step 2 truncate the constitutive equations for the stress in such a way that the truncation is inactive when \eqref{cutoff} are satisfied.

First, we use the total energy balance \eqref{tot-energy-balance} which does
not see any adiabatic and dissipative-heat terms which are problematic as far
as estimation concerns. At this point we assume that $\theta \geq 0$ and that
$J^{-1}={\rm det}(\nabla\bm\xi)>0$. It will be later proved that such
inequalities are satisfied, at least for some solutions. Note that we have
also $\omega^{\bm\xi}(\Fe,\theta) \geq 0$ and thus we are only to estimate the
right-hand side in \eqref{tot-energy-balance}. 

We perform an estimate of the power of the gravity force $\varrho\GRAVITY$ when tested by the velocity $\vv$, using the H\"older/Young inequality as follows:
\begin{align}
  \int_{\Omega} \rhoRxi J^{-1} \GRAVITY{\cdot}\vv\,\d\xx&=\int_{\Omega} \sqrt{{\rhoRxi}{J^{-1}}} \sqrt{\rhoRxi J^{-1}}\vv{\cdot}\GRAVITY\,\d\xx \nonumber
\\&
  \le\big\|\sqrt{\rhoRxi J^{-1}}\big\|_{L^2(\Omega)}\big\|\sqrt{\rhoRxi J^{-1}}\vv\big\|_{L^2(\Omega ; \R^d)}\|\GRAVITY\|_{L^{\infty}\left(\Omega ; \R^d\right)}\nonumber
\\&
\le \frac{1}{2}\bigg(\big\|\sqrt{{\rhoRxi}{J^{-1}}}\big\|_{L^2(\Omega)}^2+\big\|\sqrt{\rhoRxi J^{-1}} \vv\big\|_{L^2(\Omega ; \R^d)}^2\bigg)\|\GRAVITY\|_{L^{\infty}\left(\Omega ; \R^d\right)}\nonumber\\ & =\|\GRAVITY\|_{L^{\infty}(\Omega ; \R^d)} \int_{\Omega} \frac{\rhoRxi J^{-1}}{2}+\frac{\rhoRxi J^{-1}}{2}|\vv|^2\,\,\d\xx\nonumber \\
   &\le \|\GRAVITY\|_{L^{\infty}(\Omega;\R^d)}
   \bigg(\frac{\max _{\Omega} \rhoR}{2 \delta} \int_{\Omega}J^{-1}{\varphi(\Fe)}\,\d\xx+\int_{\Omega} \frac{\rhoRxi J^{-1}}{2}|\vv|^2\,\,\d\xx\bigg).
\label{eq:70}\end{align}
 By Assumption \eqref{Euler-ass-f-g}, $t\mapsto\|\boldsymbol{g}(t)\|_{L^\infty(\varOmega)}$ is integrable, and hence using Gronwall's inequality we obtain the bounds
\begin{subequations}
  \begin{align}
    &\Big\|\sqrt{{\rhoRxi}J^{-1}}\vv\Big\|_{L^\infty(I;L^2(\varOmega;\R^d))}\!\!\!\!=\Big\|\sqrt{{\rhoRxi}{\rm det}(\nabla\bm\xi)}\vv\Big\|_{L^\infty(I;L^2(\varOmega;\R^d))}\!\!\!\!=\Big\|\sqrt{\frac{\rhoRxi}{\det\Fe}}\vv\Big\|_{L^\infty(I;L^2(\varOmega;\R^d))}\le C,\label{eq:1}\\
    &\Big\|{J^{-1}\varphi^{\bm\xi}(\Fe)}\Big\|_{L^\infty(I;L^1(\varOmega))}
    \!\!\!\!=\Big\|{{\rm det}(\nabla\bm\xi)\varphi^{\bm\xi}(\Fe)}\Big\|_{L^\infty(I;L^1(\varOmega))}
    \!\!\!\!=\Big\|\frac{\varphi^{\bm\xi}(\Fe)}{\det\Fe}\Big\|_{L^\infty(I;L^1(\varOmega))}\le C.\label{eq:3}
  \end{align}
  \end{subequations}
  Here we have used the fact that since $\Fe=(\nabla\bm\xi)^{-1}\Fp^{-1}$, and since $\det(\Fp)=1$, we have $J^{-1}=\det(\nabla\bm\xi)=\det\Fe^{\!-1}$.
  
When recalling that we consider now only solutions with $\theta\ge 0$ and realizing that $\omega(\XX,\Fe,\theta)\ge \min(1,\varsigma){\rm sign}(\theta)|\theta|^\alpha$,
 we obtain the estimate:
 \begin{equation}\label{eq:5}
   \|\theta\|_{L^\infty(I;L^\alpha(\varOmega))}\le C.
 \end{equation}
  We now derive the energy-dissipation balance \eqref{energetics1}. The issue is now the estimation of the adiabatic term, which splits into two contributions as follows:
\begin{align}\label{adiab1}
\int_\varOmega\!\frac{[\gamma'_{\Fe}]^{\bm\xi}(\Fe,\theta)}{\det\Fe}{:}(\nabla\vv\Fe{-}\Fe\Lp)\,\d\xx&=\int_\varOmega \frac{[\gamma'_{\Fe}]^{\bm\xi}(\Fe,\theta)\Fe^{\!\!\top}\!\!\!}{\det\Fe}{:}\nabla\vv%
+%
\frac{\Fe^{\!\!\top}[\gamma'_{\Fe}]^{\bm\xi}(\Fe,\theta)\!}{\det\Fe}{:}\Lp\,\d\xx.
\end{align}
We observe that by frame indifference, $[\gamma'_{\Fe}]^{\bm\xi}(\Fe,\theta)\Fe^\top$ is a symmetric tensor, this in the first %
term on the right-hand side of \eqref{adiab1} we can replace $\nabla\vv$ with $\ee(\vv)$ and, by making use of assumption \eqref{Euler-ass-adiab+}, perform the estimate%
\begin{align}
	&\int_\varOmega \Big|\frac{[\gamma'_{\Fe}]^{\bm\xi}(\Fe,\theta)\Fetop\!\!}{\det\Fe}{:}\ee(\vv)\Big|\,\d\xx\le \bigg\|\frac{\big[\gamma'_{\Fe}\big]^{\bm\xi}(\Fe,\theta)\Fetop }{\det\Fe}\bigg\|_{L^{1}(\varOmega;\R^{d\times d})}\|\ee(\vv)\|_{L^\infty(\varOmega;\R^{d\times d})}\nonumber\\
&\qquad\qquad\qquad\le
  {C} \Big\|1+\frac{\varphi^{\bm\xi}(\Fe)}{\det\Fe}+\theta^\alpha\Big\|_{L^{1}(\varOmega;\R^{d\times d})}\|\ee(\vv)\|_{L^\infty(\varOmega;\R^{d\times d\times d})},\nonumber\\
&\qquad\qquad\qquad\le \frac {C^{p'}K_1^{p'}} {\delta^{1/(p-1)}}\Big\|1+\frac{\varphi^{\bm\xi}(\Fe)}{\det\Fe}+\theta^\alpha\Big\|^{p'}_{L^{1}(\varOmega;\R^{d\times d})}\!\!+\delta \|\nabla\ee(\vv)\|^p_{L^p(\varOmega;\R^{d\times d\times d})},
\label{est-rhs2}
\end{align} 
where we have used the inequality
$\|\nabla \vv\|_{L^{\infty}\left(\Omega ; \R^{d \times d}\right)}^p \leq K_1\left\|\nabla\ee(\vv)\right\|_{L^p(\varOmega ; \R^{d \times d \times d})}^p$,
which holds because $p>d$. The second term on the right-hand side of \eqref{adiab1} can be estimated in a similar fashion
\begin{align}
	&\int_\varOmega\Big|\frac{\Fetop[\gamma'_{\Fe}]^{\bm\xi}(\Fe,\theta)}{\det\Fe}{:}\Lp\Big|\,\d\xx\le \bigg\|\frac{\Fetop\big[\gamma'_{\Fe}\big]^{\bm\xi}(\Fe,\theta)}{\det\Fe}\bigg\|_{L^{1}(\varOmega;\R^{d\times d})}\|\Lp\|_{L^\infty(\varOmega;\R^{d\times d})}\nonumber\\
	&\stackrel{\eqref{Euler-ass-adiab+}}\le C
	\Big\|1+\frac{\varphi^{\bm\xi}(\Fe)}{\det\Fe}+\theta^\alpha\Big\|_{L^1(\varOmega;\R^{d\times d})}\|\Lp\|_{L^\infty(\varOmega;\R^{d\times d})}\nonumber
	\\
	&\ \le \frac{C^{2}K_2}{\delta}\Big\|1+\frac{\varphi^{\bm\xi}(\Fe)}{\det\Fe}+\theta^\alpha\Big\|_{L^1(\varOmega;\R^{d\times d})}^{2}\!\!+\delta\|\Lp\|^2_{L^{2}(\varOmega;\R^{d\times d})}\!+\delta\|\nabla\Lp\|^q_{L^q(\varOmega;\R^{d\times d\times d})},
        \label{eq:4}
\end{align}
when using
$\|\Lp\|^2_{L^\infty(\varOmega;\R^{d\times d})}\le K_2(\|\Lp\|^2_{L^2(\varOmega;\R^{d\times d})}+\|\nabla\Lp\|^2_{L^q(\varOmega;\R^{d\times d\times d})})$, thanks to the assumption $q>d\ge 2$. Then for $\delta$ sufficient small we can absorb on the left-hand side the last term of \eqref{est-rhs2} and the last two terms in \eqref{eq:4}. The remaining terms on the right-hand side of \eq{eq:4} have already been estimated through \eqref{eq:3} and \eqref{eq:5}. Thus, we can use the balance \eqref{energetics1} to obtain the estimates
of dissipation rates:
\begin{subequations}\label{bounds000}
\begin{align}
&&&\|\ee(\vv)\|_{L^p(I\times\varOmega;\mathbb R^{d\times d})}^{}\le C,&&\|\nabla\ee(\vv)\|_{L^p(I\times\varOmega;\mathbb R^{d\times d\times d})}^{}\le C,&&&&\label{bounds0aa}\\
&&&\|\Lp\|_{L^2(I\times\varOmega;\mathbb R^{d\times d})}^{}\le C,\ \ \ \text{ and }\!\!\!
&&\|\nabla\Lp\|_{L^q(I\times\varOmega;\mathbb R^{d\times d\times d})}^{}\le C.
\label{bounds0bb}\end{align}
\end{subequations}
In particular, from \eqref{bounds0aa} we obtain
\begin{align}\label{bound0c+}
\|\vv\|_{L^p(I;W^{2,p}(\varOmega;\R^d))}^{}\le C.	
\end{align}
Note that indeed $\nabla\nabla\vv=\nabla\ee(\vv)+\nabla\WW(\vv)$ where $\WW(\vv)$ is the skew-symmetric part of $\nabla\vv$. Let $\ww(\vv)$ be the axial vector of $\WW(\vv)$. Then it is well known that $\nabla\ww(\vv)={\rm curl}\ee(\vv)$. This implies that $\nabla\ww(\vv)$, and hence $\nabla\WW(\vv)$, is controlled by $\nabla\ee(\vv)$. Thus, in order to control the full second gradient of $\vv$ it suffices to control the gradient of $\ee(\vv)$.

Having the regularity of the velocity field $\vv$, we can exploit the assumed regularity
(\ref{Euler-ass+}j,k) of the initial conditions for $\bm\xi$ and $\Fp$, along with the
results from \cite{Roub22QHLS} and \cite{Roub22TVSE} concerning existence, regularity,
and continuous dependence on the data of the solution to transport equations of the
form \eqref{Euler2=hypoplast-xi} and \eqref{Euler3=hypoplast-xi} to deduce the bounds
\begin{align}
&\|\bm\xi\|_{L^\infty(I;W^{2,r}(\varOmega;\R^d))}\le C\qquad\text{and}\qquad
\|\Fp\|_{L^\infty(I;W^{1,r}(\varOmega;\R^{d\times d}))}\le C.\label{bounds1++}
\end{align}
Since $r>d$, we have $\|\nabla\bm\xi\|_{L^\infty(I\times\varOmega;\R^{d\times d})}\le C$.
Thus $\|\det(\nabla\bm\xi)\|_{L^\infty(I\times\varOmega)}\le C$ and
$\|{\rm Cof}(\nabla\bm\xi)\|_{L^\infty(I\times\varOmega;\R^{d\times d})}\le C$.
Accordingly, we have
$\|\nabla(\det(\nabla\bm\xi))\|_{L^\infty(I;L^r(\varOmega;\R^d))}
=\|{\rm Cof}(\nabla\bm\xi){:}\nabla^2\bm\xi\|_{L^\infty(I;L^r(\varOmega;\R^d))}
\le \|{\rm Cof}(\nabla\bm\xi)\|_{L^\infty(I\times\varOmega;\R^{d\times d})}
\|\nabla^2\bm\xi\|_{L^\infty(I;L^r(\varOmega;\R^{d\times d\times d}))}$.
Thus, we conclude\COMMENT{I was not sure whether the former equality in \eq{cutoff3} is really not needed -- so I merged it but kept. OK?}
\begin{equation}%
\|\det(\nabla\bm\xi)\|_{L^\infty(I;W^{1,r}(\varOmega))}\le C\ \ \ \text{ and }\ \ \
\Big\|\frac 1 {\det\Fe}\Big\|_{L^\infty(I;W^{1,r}(\varOmega))}\le C\,;
\label{cutoff3}\end{equation}
recall that $\det(\nabla\bm\xi)=\det(\Fp^{-1}\Fe^{\!-1})=1/\det\Fe$ since
$\det\Fp=1$.
By %
the embedding exploting $r>d$, we can see that $1/{\det\Fe}$ is bounded in
$L^\infty(I{\times}\varOmega)$. 

Also, using the transport equation that is obeyed by $\Fe$, namely, \eqref{DT-det-1+}, here rewritten as $\pdt{}{\Fe}=(\nabla\vv)\Fe{-}\Fe\Lp{-}(\vv{\cdot}\nabla)\Fe$, observing that  $\Fe(0)=(\nabla\bm\xi_0)^{-1}\FF_{\rm p\, 0}^{-1}\in W^{1,r}(\varOmega;\R^{d\times d})$ by \eqref{Euler-ass-Fe0}, we deduce also the bound
\begin{align}\label{cutoff2}
	\|\Fe\|_{L^\infty(I;W^{1,r}(\varOmega;\R^{d\times d}))}\le C.
\end{align}
Thus, the convected quantities $\Fe$ and $1/\det\Fe$ remain bounded. Specifically, for $\lambda$ a sufficiently large constant that depends on the data \eqref{Euler-ass+}, we obtain \eqref{cutoff}.

Owing to \eqref{cutoff2}, using \eqref{eq:1} and \eqref{cutoff2} we obtain
\begin{equation}
  \|\vv\|_{L^\infty(I;L^2(\varOmega;\R^d))}\le \Big\|\sqrt{\frac {\rhoRxi}{\det\Fe\!}}\,\vv\Big\|_{L^\infty(I;L^2(\varOmega;\R^d))}
\Big\|{\sqrt{\frac{\det\Fe}{\rhoRxi}}}\Big\|_{L^\infty(I\times\varOmega)}\le C.
\end{equation}

By comparison in \eq{Euler4=hypoplast-xi} and estimation by
testing it by $\Lp$, realizing the that Mandel stress on the right-hand side of
\eq{Euler4=hypoplast-xi} is bounded in $L^\infty(I{\times}\varOmega;\R_{\rm dev}^{d\times d})$ due to \eq{cutoff2} with $r>d$, we can improve \eq{bounds0bb} to
\begin{align}\label{cutoff2+}
	\|\Lp\|_{L^\infty(I;W^{1,q}(\varOmega;\R^{d\times d}))}\le C.
\end{align}
In particular $\Fp$ is bounded in
$L^\infty(I{\times}\varOmega;\R_{\rm dev}^{d\times d})$, so that the growth
of $\zeta(\theta,\cdot)$ is irrelevant.

\medskip

\noindent \emph{Step 2: regularization.} Motivated by the a-priori estimates \eqref{cutoff}, we regularize the system \eqref{Euler-hypoplast-xi} by means of the following cutoff function
\begin{align}&\pi_\LAM(\Fe)
=\begin{cases}
\qquad\qquad1&\hspace{-8em}
\text{for $\det \Fe\ge\LAM$ and $|\Fe|\le1/\LAM$,}
\\
\qquad\qquad0&\hspace{-8em}\text{for $\det \Fe\le\LAM/2$ or $|\Fe|\ge2/\LAM$,}
\\
\displaystyle{\Big(\frac{3}{\LAM^2}\big(2\det\Fe-\LAM\big)^2
-\frac{2}{\LAM^3}\big(2\det\Fe-\LAM\big)^3\Big)\,\times}\!\!&
\\[.2em]
\qquad\qquad\displaystyle{\times\,\big(3(\LAM|\Fe|-2)^2
+2(\LAM|\Fe|-2)^3\big)}\!\!&\text{otherwise}.
\end{cases}
\label{cut-off-general}
\end{align}
Note indeed that the function $f(x)=3(x{-}1)^2-2(x{-}1)^3$ we have $f(1)=0$, $f(2)=1$, $f'(1)=f'(2)=0$. Similarly, the function $g(x)=f((2-x)+1)=3(x-2)^2+2(x-2)^3$ satisfies $g(1)=1$, $g(2)=0$, and $g'(1)=g'(2)=0$.
Furthermore, we also regularize the singular nonlinearity $1/\!\det(\cdot)$
which is still employed in the right-hand-side force in the linear momentum
equation. To this aim, 
we introduce the short-hand notation 
\begin{subequations}
\begin{align}\label{cut-off-det}
&{\det}_\LAM(\Fe):=\pi_\lambda(\Fe)\det\Fe+1-\pi_\lambda(\Fe)
\ \ \text{ and }
\\&
\label{cut-off-T}
\TT_{\LAM,\EPS}^{\bm\xi}(\Fe,\theta)=
\Big(\big[(\pi_\LAM\varphi)_{\Fe}'\big]^{\bm\xi}(\Fe)+\pi_\LAM(\Fe)\frac{\big[\COUPLING_\Fe'\big]^{\bm\xi}(\Fe,\theta)}{1{+}\varepsilon|\theta|^\alpha}\Big)\frac{\Fe^\top}{\det(\Fe)}\,.
\end{align}
\end{subequations}
Note that thanks to the selected regularization and thanks to assumption \eqref{Euler-ass-adiab+}, the adiabatic stress is bounded as $\theta\to\infty$ and hence, at the level of mathematical estimates, the mechanical balance equation is decoupled from the heat equation. {Note also that no regularization for ${\rm det}$ is needed, since the cutoff is already achieved by the function $\pi_\lambda$.} 

Altogether, we consider the following system of partial differential equations
for $(\vv_\EPS,\bm\xi_\EPS,\FF_{\rm p\,\EPS},\LL_{\rm p\,\EPS},\theta_\EPS)$:
\begin{subequations}\label{Euler-hypoplast-xi-eps}
\begin{align}
\nonumber
     &%
     \rhoRxieps{\det}(\nabla\bm\xi_\EPS)\DT\vv_\EPS={\rm div}
     (\TT^{\bm\xi_\EPS}_{\lambda,\EPS}(\FF_{\rm e\,\EPS},\theta_\EPS){+}\DD_\EPS)
     +\rhoRxieps{\det}_\lambda(\nabla\bm\xi_\EPS)\GRAVITY%
     \,
     \ \\\nonumber
     &\hspace*{8em}\text{ with }\ \TT^{\bm\xi_\EPS}_{\rm e\,\lambda,\EPS}(\FF_{\rm e\,\EPS},\theta_\EPS)\text{ from }\eqref{cut-off-T}\\
     &\hspace*{8em}\text{ and }\ \DD_\EPS=\nu_0|\ee(\vv_\EPS)|^{p-2}\ee(\vv_\EPS)-{\rm div}\big(
\nu_1|\nabla\EE(\vv_\EPS)|^{p-2}\nabla\EE(\vv_\EPS)\big)\,,
\label{Euler1=hypoplast-xi-eps}
\\\label{Euler2=hypoplast-xi-eps}
&\DT{\bm\xi}_{\EPS}={\bm0}\,,
\\\label{Euler3=hypoplast-xi-eps}
&\DT\FF_{\rm p\,\EPS}=\LL_{\rm p\,\EPS}\FF_{\rm p\,\EPS}\,,
\\\nonumber
&\pl_{\Lp}^{}\zeta(\theta_\EPS;\LL_{\rm p\,\EPS})-
\frac{{\rm div}(\nu_2|\Nabla\LL_{\rm p\,\EPS}|^{q-2}\Nabla\LL_{\rm p\,\EPS})}{{\det}(\nabla\bm\xi_\EPS)}\ni
{\rm dev}\Big(\FF^{\top}_{\rm e\,\EPS}[(\pi_\lambda\varphi)_{\Fe}']^{\bm\xi_\EPS}(\FF_{\rm e\,\EPS})
\\&\hspace{12em}+\pi_\lambda(\FF_{\rm e\,\EPS})\FF_{\rm e\,\EPS}^{\top}\frac{[\gamma_{\Fe}']^{\bm\xi_\EPS}(\FF_{\rm e\,\EPS},\theta_\EPS)}{1+\EPS|\theta_\EPS|^\alpha
}-[\phi'_{\Fp}]^{\bm\xi_\EPS}(\FF_{\rm p\,\EPS})\FF_{\rm p\,\EPS}^\top\Big),\label{Euler4=hypoplast-xi-eps}
\\
&\DT\W_\EPS={\rm div}(\kappa^{\bm\xi_\EPS}(\FF_{\rm e\,\EPS},\theta_\EPS)\nabla\theta_\EPS)-\W_\EPS{\rm div}\,\vv_\EPS
\nonumber\\[0.3em]
&
\qquad\quad+\frac{\nu_0|\ee(\vv_\EPS)|^{p}+\partial_{\Lp}\!\zeta(\theta_\EPS;\LL_{\rm p\,\EPS}){:}\LL_{\rm p\,\EPS}+\nu_1|\nabla\ee(\vv_\EPS)|^p+\nu_2|\nabla\LL_{\rm p\,\EPS}|^q}{1+\EPS|\ee(\vv_\EPS)|^{p}+\EPS\partial_{\Lp}\!\zeta(\theta_\EPS;\LL_{\rm p\,\EPS}){:}\LL_{\rm p\,\EPS}+\EPS|\nabla\ee(\vv_\EPS)|^p+\EPS|\nabla\LL_{\rm p\,\EPS}|^q}
\nonumber
\\&\nonumber\qquad\quad
+\det(\nabla\bm\xi_\EPS){\pi_\lambda(\FF_{\rm e\,\EPS})}\frac{[\COUPLING'_{\Fe}]^{\bm\xi_\EPS}(\FF_{\rm e\,\EPS},\theta_\EPS)}{1+\EPS|\theta_\EPS^\beta|
}{:}\big(\nabla\vv\FF_{\rm e\,\EPS}{-}\FF_{\rm e\,\EPS}\LL_{\rm p\,\EPS}\big)
\\&\hspace{7em}\text{ with }\ \ \W_\EPS=\OMEGA^{\bm\xi_\EPS}(\nabla\bm\xi_\EPS,\FF_{\rm e\,\EPS},\theta_\EPS)\ \ \text{ and }\ \
\FF_{\rm e\,\EPS}=(\FF_{\rm p\,\EPS}\nabla\bm\xi_\EPS)^{-1}\,.
\label{Euler5=hypoplast-xi-eps}
\end{align}\end{subequations}
We complete the system with the regularized boundary conditions
\begin{align}\label{Euler-hypoplast-xi-BC-reg}
\vv_\EPS
=\bm0\,,\ \ \ 
\Nabla\EE(\vv_\EPS){{:}}(\nn{\otimes}\nn)={\bm0}\,,\ \
(\nn{\cdot}\Nabla)\LL_{\rm p\,\EPS}={\bm0},\ \text{ and }\ 
\nn{\cdot}\kappa(\FF_{\rm e\,\EPS},\theta_\EPS)\nabla\theta_\EPS=\frac{h(\theta_\EPS)}{\!1{+}\EPS|h(\theta_\EPS)|\!}
\end{align}
and with the initial conditions%
\begin{equation}\label{ic-eps}
  \vv_\EPS(0)=\vv_0,\ \ \ \ \bm\xi_\EPS(0)=\bm\xi_0,\ \ \ \
\FF_{{\rm p}\,\EPS}(0)=\FF_{{\rm p},0},\ \ \text{ and }\ \ \theta_\EPS(0)=\frac{\theta_0}{1{+}\EPS|\theta_0|}.
\end{equation}

\medskip

\noindent\emph{Step 3: semi-discretization of \eq{Euler-hypoplast-xi-eps}--\eq{ic-eps}.} 
For $\varepsilon>0$ fixed, we use a spatial semi-discretization of \eqref{Euler-hypoplast-xi-eps}, keeping the
transport equations
\eq{Euler2=hypoplast-xi-eps} and \eq{Euler3=hypoplast-xi-eps} continuous (i.e.\
non-discretised) to exploit available results concerning the regularity of their solutions,
\emph{cf.} \cite{Roub22TVSE,Roub22QHLS}.
More specifically, we make a conformal Galerkin approximation of
\eq{Euler1=hypoplast-xi-eps} by using  nested finite-dimensional
subspaces $\{V_k\}_{k\in\N}$ whose union is dense in
$W^{2,p}(\varOmega;\R^d)$. Separately, we make a Galerkin approximation of \eqref{Euler4=hypoplast-xi-eps} by using other nested finite-dimensional subspaces $\left\{W_{l}\right\}_{l \in \mathbb{N}}$ whose union is dense in $W^{1,q}(\Omega ; \R^{d \times d}_{\rm dev})$. We also make a conformal Galerkin approximation of
\eq{Euler5=hypoplast-xi-eps} by using  nested finite-dimensional
subspaces $\{Z_k\}_{k\in\N}$ whose union is dense in $H^1(\varOmega)$.
Without loss of generality, we assume $\vv_0\in V_1$ and 
$\theta_{0,\varepsilon}\in Z_1$.

The approximate solution of the regularized system will be denoted by $(\vv_{\EPS k},\bm\xi_{\EPS k},\FF_{{\rm p}\,\EPS k},\LL_{{\rm p}\,\EPS k},\theta_{\EPS k}):
I\to V_k\times W^{2,r}(\varOmega;\R^{d})\times W^{1,r}(\varOmega;\R^{d\times d})\times W_k\times Z_k$.
Specifically, such a quintuple should satisfy
\begin{align}\label{transport-equations+}
\pdt{{\bm\xi}_{\EPS k}}=-(\vv_{\EPS k}{\cdot}\nabla){\bm\xi}_{\EPS k}\ \ \ \text{ and }
\ \ \ \pdt{{\FF}_{{\rm p}\,\EPS k}}={\LL}_{{\rm p}\,\EPS k}{\FF}_{{\rm p}\,\EPS k}
-(\vv_{\EPS k}{\cdot}\nabla){\FF}_{{\rm p}\,\EPS k},
\end{align}
respectively in the $L^1(I{\times}\varOmega;\R^d)$ and $L^1(I{\times}\varOmega;\R^{d\times d})$ sense, together with the following integral identities
\begin{subequations}\label{Euler-weak-Galerkin+}\begin{align}
&\nonumber
\int_0^T\!\!\!\int_\varOmega\Big(\big(\TT^{\bm\xi_{\EPS k}}_{\LAM,\EPS}(\Fe_{\EPS k},\theta_{\EPS k})
+\nu_0|\ee(\vv)|^{p-2}\ee(\vv)-{{\det}}(\nabla\bm\xi_{\varepsilon k})\rhoRxiepsk\vv_{\varepsilon k}{\otimes}\vv_{\varepsilon k}\big){:}\ee(\widetilde\vv)
\\[-.2em]\nonumber&\hspace{3em}
+\nu_1|\nabla\ee(\vv_{\EPS k})|^{p-2}\nabla\ee(\vv_{\EPS k})\Vdots
\Nabla\ee(\widetilde\vv)-{\operatorname{det}}(\nabla \bm\xi_{\EPS k})
\rhoRxi \vv_{\varepsilon k}{\cdot}\frac{\partial \widetilde{\vv}}{\partial t}\Big)
\,\d\xx\d t
\\[-.2em]\nonumber&\hspace{7em}
=\!\int_0^T\!\!\!\int_\varOmega
{\rm det}_\lambda(\nabla\bm\xi_{\EPS k}){\rhoRxiepsk\gg}{\cdot}\widetilde\vv\,\d\xx\d t+\int_\varOmega{{\det}}(\nabla\bm\xi_0)\rhoRxizero\vv_0\cdot\widetilde\vv(0)\dx\quad\\[-0.2em]
&\hspace{3em}\text{ with }\ \ \
\FF_{{\rm e}\,\EPS k}=(\FF_{{\rm p}\,\EPS k}\nabla\bm\xi_{\EPS k})^{-1}
\label{Euler1-weak-Galerkin+}
\intertext{for any $\widetilde\vv\in C^\infty(I;V_k)$, such that
$\widetilde\vv(T)=0$ and $\vv=0$ on $I{\times}\varGamma$, and}
&\!\int_0^T\!\!\!\int_\varOmega{{\det}}(\nabla\bm\xi_{\EPS k})\zeta(\theta_{\EPS k},\widetilde\Lp)+\nu_2|\nabla\widetilde\Lp|^q\,\d\xx\d t\ge\int_0^T\!\!\!\int_\varOmega\bigg({{\det}}(\nabla\bm\xi_{\EPS k})\zeta(\theta_{\EPS k},\LL_{{\rm p}\,\EPS k})
\nonumber\\[-.2em]&\hspace{3em}
+{{\det}}(\nabla\bm\xi_{\EPS k})\Big({\FF}_{{\rm e}\,\EPS k}^\top[(\pi_\lambda\varphi)'_{\Fe}]^{\bm\xi_{\EPS k}}(\FF_{{\rm e}\,\EPS k})+\pi_\lambda(\FF_{{\rm e}\,\EPS k})\FF^\top_{{\rm e}\,\EPS k}\frac{[\gamma_{\Fe}']^{\bm\xi_{\EPS k}}(\FF_{{\rm e}\,\EPS k},\theta_{\EPS k})}{1+\EPS|\theta_{\EPS k}|^{\REPLACE{\beta}{\alpha}%
}}
\nonumber
\\[-.3em]\label{Euler2-weak-Galerkin+}
&\hspace{3em}-[\phi'_{\Fp}]^{\bm\xi_{\EPS k}}(\FF_{{\rm p}\,\EPS k})\FF_{{\rm p}\,\EPS k}^\top\Big){:}(\widetilde\Lp-\LL_{{\rm p}\,\EPS k})
+\nu_2|\nabla\LL_{{\rm p}\,\EPS k}|^q\bigg)\,\d\xx\d t
\intertext{for any $\widetilde\Lp\in L^\infty(I;W_k)$, and}\nonumber
&\!\int_0^T\!\!\!\int_\varOmega\bigg(\W_{\EPS k}\pdt{\widetilde\theta}
+\big(\W_{\EPS k}\vv_{\EPS k}
{-}\kappa^{\bm\xi_{\EPS k}}(\FF_{{\rm e}\,\EPS k},\theta_{\EPS k})\nabla\theta_{\EPS k}\big)
{\cdot}\nabla\widetilde\theta
\\[-.2em]&\hspace{3em}\nonumber
+\Big(\frac{\nu_0|\ee(\vv)|^{p}+\nu_1|\nabla\EE(\vv_{\EPS k})|^p+\partial_{\Lp}\!\zeta(\theta_{\EPS k};\LL_{{\rm p}\,\EPS k}){:}\LL_{{\rm p}\,\EPS k}+\nu_2|\nabla\LL_{{\rm p}\,\EPS k}|^q\!\!}{1{+}\EPS|\EE(\vv_{\EPS k})|^p{+}\EPS|\nabla\EE(\vv)|^p+\EPS\partial_{\Lp}\!\zeta(\theta_{\EPS k};\LL_{{\rm p}\,\EPS k}){:}\LL_{{\rm p}\,\EPS k}+\EPS|\nabla\LL_{{\rm p}\,\EPS k}|^q}
\\
\nonumber&\hspace{3em}+\frac{\pi_\LAM(\Fe_{\EPS k})\big[\COUPLING'_{\Fe}\big]^{\bm\xi_{\EPS k}}(\Fe_{\EPS k},\theta_{\EPS k})\Fe_{\EPS k}^\top
}{\det_\LAM(\Fe_{\EPS k})(1{+}\EPS|\theta_{\EPS k}|^{\REPLACE{\beta}{\alpha}%
})}{:}\big(\nabla\vv\FF_{{\rm e}\,\EPS k}{-}\FF_{{\rm e}\,\EPS k}\LL_{{\rm p}\,\EPS k}\big)
\Big)\widetilde\theta\,\bigg)\,\d\xx\d t
\\[-.1em]&\hspace{0em}
+\!\int_\varOmega\!\OMEGA^{\bm\xi_0}(\FF_0,\theta_{0,\varepsilon})\widetilde\theta(0)\,\d\xx
+\!\int_0^T\!\!\!\int_\varGamma\!\frac{h(\theta_{\EPS k})}{1{+}\EPS|h(\theta_{\EPS k})|}
\widetilde\theta\,\d S\d t=0
\ \text{ with }\ \W_{\EPS k}=\OMEGA^{\bm\xi_\EPS}(\FF_{{\rm e}\,\EPS k},\theta_{\EPS k})
\label{Euler3-weak-Galerkin+}
\end{align}
\end{subequations}
holds for any $\widetilde\theta\in C^1(I;Z_k)$ with $\widetilde\theta(T)=0$.

{It should be realized that} the {approximate temperatures} $\theta_{\EPS k}$ are not necessarily positive. Accordingly, {for Steps 3--5}, we {have to} extend the functions
$\psi$, $\kappa$, and $h$ to negative temperatures by defining%
\begin{align}\nonumber
&\psi(\XX,\Fe,\theta):=\varphi(\XX,\Fe)+\COUPLING(\XX,\Fe,0)-|\theta|^{\alpha-1}\theta\det\Fe\,,
\\&\kappa(\XX,\Fe,\theta):=\kappa(\XX,\Fe,-\theta)\,,
\ \ \ \text{ and }\ \ \ h(t,\xx,\theta):=h(t,\xx,-\theta)
\ \ \ \text{ for }\ \ \theta<0,
\label{extension-negative+}\end{align}
with $\varphi$ and $\COUPLING$ from the split \eq{psi-split}. This definition
makes $\psi:\varOmega\times{\rm GL}^+(d)\times\R\to\R$ continuous and
implies that, for $\theta$ negative, $\OMEGA(\XX,\Fe,\theta)=(\alpha{-}1)|\theta|^{\alpha-1}\theta$ { so that the primitive of $\widehat\omega$ used below
 will be $|\theta|^{1+\alpha}$ for $\theta<0$} and 
so that $\OMEGA_\XX'(\XX,\Fe,\theta)=0$ and $\OMEGA_\Fe'(\XX,\Fe,\theta)=0$.

The existence of solutions to \eqref{transport-equations+}--\eqref{Euler-weak-Galerkin+}
is guaranteed by the standard theory of systems of set-valued algebraic-differential
inclusions in finite-dimensional spaces combined with evolution-and-transport
equation as in \cite{Roub22QHLS,Roub22TVSE}
first locally in time and then by successive prolongation on the whole time interval based
on the $L^\infty$ estimates below.

We remark that, since the transport equations \eqref{transport-equations+} hold a.e. in $I{\times}\varOmega$, and thanks to the smoothness of their solutions, it is legal to perform the calculations in \eqref{DT-det-1+} to obtain the following transport equations:
\begin{subequations}\label{transport-eqs}
\begin{align}\label{transport-equation-det}	
&\pdt{}{\det(\nabla\bm\xi_{\EPS k})}=-(\vv_{\EPS k}{\cdot}\nabla)\det(\nabla\bm\xi_{\EPS k})-(\det(\nabla\bm\xi_{\EPS k})){\rm div}\bm v_{\EPS k}\,,
  \\
  \label{transport-equation-detinv}	
&\pdt{}\frac 1 {\det(\nabla\bm\xi_{\EPS k})}=-(\vv_{\EPS k}{\cdot}\nabla)\frac 1 {\det(\nabla\bm\xi_{\EPS k})}+\frac 1 {\det(\nabla\bm\xi_{\EPS k})}{\rm div}\bm v_{\EPS k}\,,
\\
\label{transport-equation-Fe}	
&\pdt{}{\FF_{{\rm e}\,\EPS k}}=-(\vv_{\EPS k}{\cdot}\nabla){\FF_{{\rm e}\,\EPS k}}+\nabla\vv_{\EPS\, k}{\FF_{{\rm e}\,\EPS k}}-{\FF_{{\rm e}\,\EPS k}}\LL_{{\rm p}\,\EPS k}\,.
\end{align}
\end{subequations}
In addition, on setting%
\begin{equation}\label{def-rho}
\rho_{\EPS k}={\rhoRxiepsk}
{\rm det}(\nabla\bm\xi_{\EPS k})\,,
\end{equation}
 we have the calculus
\begin{align}  
 \nonumber\pdt{\rho_{\EPS k}}&=\pdt{}(\rhoRxiepsk{\rm det}(\nabla\bm\xi_{\EPS k}))=
  [[\rhoR]_\XX']^{\bm\xi_{\EPS k}}\det(\nabla\bm\xiek)
  \pdt{\bm\xi_{\EPS k}}+\rhoRxiepsk{\rm Cof}(\nabla\bm\xi_{\EPS k})
  {:}\pdt{}\nabla\bm\xi_{\EPS k}
          \\\nonumber  &
=-[[\rhoR]_\XX']^{\bm\xi_{\EPS k}}{\rm det}(\nabla\bm\xi_{\EPS k})(\vv_{\EPS k}
{\cdot}\nabla)\xiek
\\&\qquad\nonumber 
-\rhoRxiepsk{\rm Cof}(\nabla\bm\xi_{\EPS k}){:}\nabla\xiek\nabla\vv_{\EPS k}-\rhoRxiepsk{\rm Cof}(\nabla\xiek)(\vv_{\EPS k}{\cdot}\nabla)\nabla\xiek\\
  \nonumber&=-(\vv_{\EPS k}\cdot\nabla)(\rhoRxiepsk{\det}(\nabla\xiek)
 -\rhoRxiepsk\det(\nabla\xiek){\rm div}\vv_{\EPS k}
  =-{\rm div}(\rho_{\EPS k}\vv_{\EPS k}),
\end{align}
where we have used \eqref{dtnabla}.
This shows that the density $\rho_{\EPS k}$ satisfies the %
continuity equation
\begin{equation}\label{transport-rho}
	\pdt{\rho_{\EPS k}}+{\rm div}(\rho_{\EPS k}\vv_{\EPS k})=0
\end{equation}
in the $L^1(I{\times}\varOmega)$-sense with the initial condition $\rho_{\EPS k}(0)=\rhoRxizero{\rm det}(\nabla\bm\xi_0)$. This initial condition satisfies
\begin{equation}
\nabla\varrho_{\EPS k}(0)={\rm det}(\nabla\bm\xi_0)(\nabla{\bm\xi_0})^\top
[[\rhoR]_\XX']^{\bm\xi_0}+\rhoRxizero{\rm Cof}(\nabla\bm\xi_0){:}\nabla^2{\bm\xi}_0\in L^r(\varOmega;\R^d).
\end{equation}
	\medskip

        \noindent \emph{Step 4: first a priori estimates.}
It follows from \eqref{Euler2-weak-Galerkin+} that there exists a measurable selection of $\partial_{\Lp}^{}\!\zeta(\theta_{\EPS k};\LL_{{\rm p}\,\EPS k})$, which we denote still by $\partial_{\Lp}^{}\!\zeta(\theta_{\EPS k};\LL_{{\rm p}\,\EPS k})$, such that
	\begin{align}\label{eq:dissip}
	&\!\int_0^T\!\!\!\int_\varOmega\!{\det}(\nabla\bm\xi_{\EPS k})\partial_{\Lp}^{}\!\zeta(\theta_{\EPS k},\LL_{{\rm p}\,\EPS k}){:}\LL_{{\rm p}\,\EPS k}+\nu_2|\nabla\LL_{{\rm p}\,\EPS k}|^q\,\d\xx\d t\nonumber\\[-.2em]&\hspace{0em}
	=\int_0^T\!\!\!\int_\varOmega{\det}(\nabla\bm\xi_{\EPS k})\Big({\FF}_{{\rm e}\,\EPS k}^\top[(\pi_\lambda\varphi)'_{\Fe}]^{\bm\xi_{\EPS k}}(\FF_{{\rm e}\,\EPS k})+\pi_\lambda(\FF_{{\rm e}\,\EPS k})\FF^\top_{{\rm e}\,\EPS k}\frac{\![\gamma_{\Fe}']^{\bm\xi_{\EPS k}}(\FF_{{\rm e}\,\EPS k},\theta_{\EPS k})}{1+\EPS|\theta_{\EPS k}|^{\beta/p'}}\Big){:}\LL_{{\rm p}{\,\EPS k}}\,
\d\xx\d t.
\end{align}
We remark that since $\zeta(\theta;\cdot)$ is smooth away from $0$, the product $\partial_{\Lp}^{}\!\zeta(\theta_{\EPS k};{\Lp}_{\EPS k}){:}{\Lp}_{\EPS k}$ is indeed singled valued.

We choose as tests $\widetilde\vv=\vv_{\EPS k}$ for \eq{Euler1-weak-Galerkin+}.
Here we take advantage from the transport equations \eqref{transport-equations+} holding pointwise. This allows us to replicate the formal estimates performed in Sec.~\ref{sec-energetics}. In particular, using also \eqref{eq:dissip}, we obtain
    the {inequality} (\emph{cf.} \eqref{energetics1}):
    \begin{align}\nonumber
  &\hspace*{0em}\frac{\d}{\d t}
  \int_\varOmega\!
{\frac{\rho_{\EPS k}^2}{2}|\vv_{\EPS k}|^2}+
  \frac{\pi_\LAM(\Fe_{\EPS k})\varphi^{\bm\xi_{\EPS k}}(\Fe_{\EPS k})\!}{\det(\Fe_{\EPS k})}
  \,\d\xx
\\[-.1em]\nonumber&\hspace{2em}
+\!\int_\varOmega\!{\nu_0|\EE(\vv_{\EPS k}))|^p
+\nu_1|\Nabla\EE(\vv_{\EPS k})|^p+{{\rm det}(\nabla\bm\xi_{\EPS k})}\partial_{\Lp}\!\zeta(\theta;{\LL}_{{\rm p}\,\EPS k}){:}{\LL}_{{\rm p}\,\EPS k}+\nu_2|\nabla\LL_{{\rm p}\,\EPS k}|^q}\,\d\xx
\\[-.1em]&\hspace{3em}
\le\int_\varOmega\frac{\varrho^{\boldsymbol\xi_{\EPS k}}_{\textsc r}\GRAVITY{\cdot}\vv_{\EPS k}}{\det_\LAM(\FF_{\EPS k})\!\!}
-\frac{\pi_\LAM(\Fe_{\EPS k})\COUPLING_{\Fe}'(\Fe_{\EPS k},\theta_{\EPS k})}{(1{+}\EPS|\theta_{\EPS k}|^{\alpha%
}
)\det(\Fe_{\EPS k})}{:}(\nabla\vv_{\EPS k}\FF_{{\rm e}\,\EPS k}{-}\FF_{{\rm e}\,\EPS k}\LL_{{\rm p}\,\EPS k})\,\d\xx.
\label{thermodynamic-Euler-mech-disc+}
\end{align}
The first term on the right-hand side, representing the power of the bulk force can be estimated as in \eqref{eq:70}. The estimate of the second term on the right-hand side, representing the power of the adiabatic stress, is even simpler than \eqref{est-rhs2}--\eqref{eq:4}, because the adiabatic stress in the regularized system is bounded and vanishes for $|\FF_{{\rm e}\,\EPS k}|\ge 1/\lambda$. In particular,
\begin{align}
&\int_\varOmega\frac{\pi_\LAM(\Fe_{\EPS k})\COUPLING_{\Fe}'(\Fe_{\EPS k},\theta_{\EPS k})}{(1{+}\EPS|\theta_{\EPS k}|^{\alpha%
}
)\det(\Fe_{\EPS k})}{:}(\nabla\vv_{\EPS k}\FF_{{\rm e}\,\EPS k} {-}\FF_{{\rm e}\,\EPS k}\LL_{{\rm p}\,\EPS k})\,\d\xx
\nonumber\\[-.3em]
&\hspace{13em}\le C_a+a\|\ee(\vv_{\EPS k})\|^p_{L^p(\varOmega;\R^{d\times d}_{\rm sym})}+a\|\LL_{{\rm p}\,\EPS k}\|^q_{L^q(\varOmega;\R^{d\times d}_{\rm dev})}.	
\end{align}
For $a$ sufficiently small the last two terms can be absorbed on the left-hand side of \eqref{thermodynamic-Euler-mech-disc+}. Then we can apply the Gronwall inequality and we obtain the following bounds independent on $k$:
\begin{subequations}\label{bounds0ab}
\begin{align}
&\Big\|\sqrt{\rho_{\EPS k}}\vv_{\EPS k}\Big\|_{L^\infty(I;L^2(\varOmega;\R^d))}\le C_\EPS,
&&\|\ee(\vv_{\EPS k})\|_%
{L^p(I;W^{1,p}(\varOmega;\mathbb R^{d\times d}))}^{}
\le C_\EPS,
\label{bounds0a}\\
&\|{\LL}_{{\rm p}\,\EPS k}\|_{L^2(I\times\varOmega;\mathbb R^{d\times d})}\le C_\EPS,\qquad\quad\text{ and}
&&\|\nabla{\LL}_{{\rm p}\,\EPS k}\|_{L^q(I\times\varOmega;\mathbb R^{d\times d\times d})}\le C_\EPS.\label{bounds0b}
\end{align}
\end{subequations}

We can now argue as in \eqref{bound0c+} to obtain 
\begin{align}\label{bound0c}
\|\vv_{\EPS k}\|_{L^p(I;W^{2,p}(\varOmega;\R^d))}\le C_{\EPS}.
\end{align}
Then we use the regularity results concerning the solutions of the transport equations \eqref{transport-equations+} and (\ref{transport-eqs}a-b) to obtain
\begin{subequations}
\begin{align}
  \label{bounds0}&\|\bm\xi_{\EPS k}\|_{L^\infty(I;W^{2,r}(\varOmega;\R^d))}\le C_\EPS,\\
&\|\FF_{{\rm p}\,\EPS k}\|_{L^\infty(I;W^{1,r}(\varOmega;\R^d))}\le C_\EPS.
\\
                 &\|\det(\nabla\bm\xi_{\EPS k})\|_{L^\infty(I;W^{1,r}(\varOmega))}\le C_\EPS,\label{bounds2}\\
    &\Big\|\frac 1 {\det(\nabla\bm\xi_{\EPS k})}\Big\|_{L^\infty(I;W^{1,r}(\varOmega))}\le C_\EPS,\label{bounds2b}
    \end{align}
Also, using the transport equation \eqref{transport-equation-Fe}, by the assumed regularity of the initial conditions, and arguing as in \eqref{cutoff2}
{and as in \eq{cutoff2+}}, we deduce
\begin{align}\label{cutoff2-eps-k}
\|\FF_{{\rm e}\,\EPS k}\|_{L^\infty(I;W^{1,r}(\varOmega;\R^{d\times d}))}\le C_\EPS\ \ \
{\text{ and }\ \ \ \|\LL_{{\rm p}\,\EPS k}\|_{L^\infty(I;W^{1,q}(\varOmega;\R^{d\times d}))}\le C_\EPS}.
\end{align}

From the transport equation \eqref{transport-rho} for $\rho_{\EPS k}$ we also obtain the estimate
\begin{equation}\label{eq:11}
    \|\rho_{\EPS k}\|_{L^\infty(I;W^{1,r}(\varOmega))}\le C_\EPS.
  \end{equation}
  Furthermore, from $\nabla(1/{\rho_{\EPS k}})=-(1/{\rho_{\EPS k}^2})\nabla\rho_{\EPS k}$, using \eqref{eq:11},
 we obtain the estimate
\begin{equation}\label{eq:11b}
    \Big\|\frac 1 {\rho_{\EPS k}}\Big\|_{L^\infty(I;W^{1,r}(\varOmega))}\le C_\EPS.
  \end{equation}
  This allows us to estimate the velocity by
  \begin{equation}\label{pippo:1}
    \|\vv_{\EPS k}\|_{L^\infty(I;L^2(\varOmega;\R^d))}\le\|\sqrt{\rho_{\EPS k}}\vv_{\EPS k}\|_{L^\infty(I;L^2(\varOmega;\R^d))}\Big\|\frac 1 {\sqrt{\rho_{\EPS k}}}\Big\|_{L^\infty(I\times\varOmega)}\le C_\EPS.
  \end{equation}
 By comparison in the first equation in \eqref{transport-equations+} and in \eqref{transport-equation-Fe},  using the estimates \eqref{bound0c}, \eqref{bounds0}, and \eqref{cutoff2-eps-k}, we have 
\begin{align}\label{DTFxi}
\Big\|\frac{\partial\FF_{{\rm e}\,\EPS k}}{\partial t}\Big\|_{L^{\min(p,q)}(I;L^r(\varOmega;\R^{d\times d}))}\le C_\EPS\qquad\text{and}\qquad \Big\|\frac{\partial\bm\xi_{\EPS k}}{\partial t}\Big\|_{L^p(I;W^{1,r}(\varOmega;\R^d))}\le C_\EPS.
\end{align}
Similarly, by comparison in \eqref{transport-rho}, we obtain
\begin{equation}\label{eq:12}
 \Big\|\pdt{\rho_{\EPS k}}\Big\|_{L^p(I;L^r(\varOmega))}\le C_\varepsilon.
\end{equation}
\end{subequations}
We now apply the $L^2$ theory to the heat equation, testing it by $\theta_{\EPS k}$. Note that this test is legal because the heat source terms and the adiabatic terms on the right-hand side have been regularized. The test of the discretized heat equation \eqref{Euler3-weak-Galerkin+} by $\theta_{\EPS k}$ gives
\begin{align}\nonumber
&\int_\varOmega\theta_{\EPS k}\pdt{}\OMEGA(\FF_{{\rm e}\EPS k},\theta_{\EPS k})
+\kappa^{{\bm\xi}_{\EPS k}}(\FF_{\EPS k},\theta_{\EPS k})|\nabla\theta_{\EPS k}|^2\,\d\xx
\\[-.2em]&\hspace{1em}\nonumber
=\int_\varOmega\!\bigg(
\Big(\frac{\nu_0|\ee(\vv_{\EPS k})|^{p}+\nu_1|\nabla\EE(\vv_{\EPS k})|^p+\partial_{\Lp}\!\zeta(\theta_{\EPS k};\LL_{{\rm p}\,\EPS k}){:}\LL_{{\rm p}\,\EPS k}+\nu_2|\nabla\LL_{{\rm p}\,\EPS k}|^q\!\!}{1{+}\EPS|\EE(\vv_{\EPS k})|^p{+}\EPS|\nabla\EE(\vv)|^p+\EPS\partial_{\Lp}\!\zeta(\theta_{\EPS k};\LL_{{\rm p}\,\EPS k}){:}\LL_{{\rm p}\,\EPS k}+\EPS|\nabla\LL_{{\rm p}\,\EPS k}|^q}
\\\nonumber
&\hspace{4em}+\frac{\pi_\LAM(\Fe_{\EPS k})[\COUPLING'_{\Fe}]^{\bm\xi_{\EPS k}}(\Fe_{\EPS k},\theta_{\EPS k})}{\det_\LAM(\Fe_{\EPS k})
(1{+}\EPS|\theta_{\EPS k}|^{\alpha%
}
)}{:}(\nabla\vv_{\EPS\,k}\FF_{{\rm e}\,\EPS k}{-}\FF_{{\rm e}\,\EPS k}\LL_{{\rm p}\,\EPS k})\Big)\theta_{\EPS k}
\\[-.1em]&\hspace{6em}
+\OMEGA^{\bm\xi_{\EPS k}}(\FF_{\EPS k},\theta_{\EPS k})\vv_{\EPS k}{\cdot}
\nabla\theta_{\EPS k}\bigg)\,\d\xx
+\int_\varGamma
\frac{h(\theta_{\EPS k})}{1{+}\EPS |h(\theta_{\EPS k})|}\theta_{\EPS k}\,\d S\,.
\label{Euler3-Galerkin-est+++}\end{align}
We will use the calculus%
\begin{align}\nonumber
&\int_\varOmega\theta_{\EPS k}\pdt{\W_{\EPS k}}\,\d\xx
=\int_\varOmega\Big(\theta_{\EPS k}[\OMEGA_\theta']^{\bm\xi_{\EPS k}}({\FF_{{\rm e}\,\EPS k}},\theta_{\EPS k})\pdt{\theta_{\EPS k}}
+\theta_{\EPS k}[\OMEGA_\Fe']^{\bm\xi_{\EPS k}}({\FF_{{\rm e}\,\EPS k}},\theta_{\EPS k}){:}\pdt{{\FF_{{\rm e}\,\EPS k}}}\\
&\hspace{22.3em}\nonumber+\theta_{\EPS k}[\omega'_\XX]^{\bm\xi_{\EPS k}}({\FF_{{\rm e}\,\EPS k}},\theta_{\EPS k})\cdot\pdt{\bm\xi_{\EPS k}}\Big)
\,\d\xx 
\\&\hspace{0em}\nonumber
=\frac{\d}{\d t}\int_\varOmega\widehat\OMEGA^{\bm\xi_{\EPS k}}({\FF_{{\rm e}\,\EPS k}},\theta_{\EPS k})\,\d\xx
-\!\int_\varOmega\!\Big([\widehat\OMEGA_\Fe']^{\bm\xi_{\EPS k}}({\FF_{{\rm e}\,\EPS k}},\theta_{\EPS k})
{-}\theta_{\EPS k}[\OMEGA_\Fe']^{\bm\xi_{\EPS k}}({\FF_{{\rm e}\,\EPS k}},\theta_{\EPS k})\Big){:}\pdt{{\FF_{{\rm e}\,\EPS k}}}\,\d\xx
\\&\hspace{10em}-\int_\varOmega\Big([\widehat\OMEGA_\XX']^{\bm\xi_{\EPS k}}({\FF_{{\rm e}\,\EPS k}},\theta_{\EPS k})
{-}\theta_{\EPS k}[\OMEGA_\XX']^{\bm\xi_{\EPS k}}({\FF_{{\rm e}\,\EPS k}},\theta_{\EPS k})\Big){:}\pdt{\bm\xi_{\EPS k}}\,\d\xx,
\label{Euler-thermodynam3-test+++}
\end{align}
where $\widehat\OMEGA(\XX,\FF,\theta)$ is a primitive function of
$\theta\mapsto\theta\OMEGA_\theta'(\XX,\Fe,\theta)$ depending smoothly
on $\Fe$ and $\XX$, specifically, recalling \eqref{split-omega}, 
\begin{align}
\widehat\OMEGA(\XX,\Fe,\theta)=\!\int_0^1\!\!r\theta^2\OMEGA_\theta'(\XX,\Fe,r\theta)\,\d r=\!\int_0^1\!\!r\theta^2\varpi_\theta'(\XX,\Fe,r\theta)\,\d r+\varsigma\alpha(\alpha{-}1)|\theta|^{1+\alpha},
\label{primitive}\end{align}
so that
\begin{align}\label{eq:omegahat1}
\widehat\OMEGA_\Fe'(\XX,\Fe,\theta)=\int_0^1\!\!r\theta^2\varpi_{\Fe\theta}''(\XX,\Fe,r\theta)\,\d r\ \text{ and }\
\widehat\OMEGA_\XX'(\XX,\Fe,\theta)=\int_0^1\!\!r\theta^2\varpi_{\XX\theta}''(\XX,\Fe,r\theta)\,\d r.
\end{align}
By \eqref{cutoff2-eps-k}, the elastic strains are valued in some bounded set
(depending possibly on $\EPS$).
Therefore, thanks to assumption
\eq{Euler-ass-primitive-c+}, by the first equation in \eqref{eq:omegahat1}, we have $|[\widehat\OMEGA_\Fe']^{\bm\xi_{\EPS k}}({\FF_{{\rm e}\,\EPS k}},\theta_{\EPS k})|\le C_\EPS(1{+}|\theta_{\EPS k}|^{{(1+\alpha)}/r'})$ and $|\theta_{\EPS k}[\omega'_{\Fe}]^{\bm\xi_{\EPS k}}({\FF_{{\rm e}\,\EPS k}},\theta_{\EPS k})|=|\theta_{\EPS k}[\varpi'_{\Fe}]^{\bm\xi_{\EPS k}}({\FF_{{\rm e}\,\EPS k}},\theta_{\EPS k})|\le C_\EPS(1{+}|\theta_{\EPS k}|^{(1+\alpha)/r'})$. 
Thus, $|[\widehat\OMEGA_\Fe']^{\bm\xi_{\EPS k}}({\FF_{{\rm e}\,\EPS k}},\theta_{\EPS k})-[\theta_{\EPS k}\omega'_{\Fe}]^{\bm\xi_{\EPS k}}({\FF_{{\rm e}\,\EPS k}},\theta_{\EPS k})|^{r'}\le C_\EPS(1{+}|\theta_{\EPS k}|^{1+\alpha})$. Hence, the penultimate integral on the right-hand side of 
\eq{Euler-thermodynam3-test+++} can be estimated as
\begin{align}\nonumber
&\int_\varOmega\!\Big([\widehat\OMEGA_\Fe']^{\bm\xi_{\EPS k}}(\FF_{{\rm e}\EPS k},\theta_{\EPS k})
{-}\theta_{\EPS k}[\OMEGA_\Fe']^{\bm\xi_{\EPS k}}(\FF_{{\rm e}\EPS k},\theta_{\EPS k})\Big){:}\pdt{\FF_{{\rm e}\EPS k}}\,\d\xx
\\\nonumber
&\le
\,\Big\|[\widehat\OMEGA_\Fe']^{\bm\xi_{\EPS k}}(\FF_{{\rm e}\EPS k},\theta_{\EPS k})
{-}\theta_{\EPS k}[\OMEGA_\Fe']^{\bm\xi_{\EPS k}}(\FF_{{\rm e}\EPS k},\theta_{\EPS k})\Big\|_{L^{r'}(\varOmega)}\Big\|\pdt{\FF_{{\rm e}\EPS k}}\Big\|_{L^r(\varOmega;\R^{d\times d})}\\
&= \Big\|\pdt{\FF_{{\rm e}\EPS k}}\Big\|_{L^r(\varOmega;\R^{d\times d})}\bigg(\int_\varOmega\big([\widehat\OMEGA_\Fe']^{\bm\xi_{\EPS k}}(\FF_{{\rm e}\EPS k},\theta_{\EPS k})
{-}\theta_{\EPS k}[\OMEGA_\Fe']^{\bm\xi_{\EPS k}}(\FF_{{\rm e}\EPS k},\theta_{\EPS k})\big)^{r'}\,\d\xx\bigg)^{1/r'}\nonumber
\\
                &\le C_\EPS\bigg\|\pdt{\FF_{{\rm e}\EPS k}}\Big\|_{L^r(\varOmega;\R^{d\times d})}\Big(1{+}\!\int_\varOmega\!|\theta_{\EPS k}|^{1+\alpha}\,\d\xx\bigg)^{1/r'}\!\!\!\!
                \nonumber
                  \le \frac{C_\EPS}{r}\Big\|\pdt{\FF_{{\rm e}\EPS k}}\Big\|^r_{L^r(\varOmega;\R^{d\times d})}\!\!+\frac 1 {r'}\bigg(1{+}\!\int_\varOmega\!|\theta_{\EPS k}|^{1+\alpha}\,\d\xx\bigg).%
\end{align}%
In a similar fashion, the growth assumption \eqref{Euler-ass-primitive-c+} and
eqref{eq:omegahat1} imply that $|[\widehat\OMEGA_\XX']^{\bm\xi_{\EPS k}}(\FF_{{\rm e}\,\EPS k},\theta_{\EPS k})|\le C_\EPS(1+|\theta_{\EPS k}|^{1+\alpha})$ and $|\theta_{\EPS k}[\OMEGA_\XX']^{\bm\xi_\EPS k}(\FF_{{\rm e}\EPS k},\theta_{\EPS k})|\le C_\EPS(1+|\theta_{\EPS k}|^{1+\alpha})$, therefore
\begin{align}
	\nonumber
\int_\varOmega\!\Big([\widehat\OMEGA_\XX']^{\bm\xi_{\EPS k}}&({\FF_{{\rm e}\,\EPS k}},\theta_{\EPS k})
{-}\theta_{\EPS k}[\OMEGA_\XX']^{\bm\xi_{\EPS k}}({\FF_{{\rm e}\,\EPS k}},\theta_{\EPS k})\Big){:}\pdt{\bm\xi_{\EPS k}}\,\d\xx
\\
&\nonumber\le
\,\Big\|[\widehat\OMEGA_\XX']^{\bm\xi_{\EPS k}}({\FF_{{\rm e}\,\EPS k}},\theta_{\EPS k})
{-}\theta_{\EPS k}[\OMEGA_\XX']^{\bm\xi_{\EPS k}}({\FF_{{\rm e}\,\EPS k}},\theta_{\EPS k})\Big\|_{L^{1}(\varOmega)}\Big\|\pdt{\bm\xi_{\EPS k}}\Big\|_{L^\infty(\varOmega;\R^d)}
\nonumber\\
&=
\,\Big\|\pdt{\bm\xi_{\EPS k}}\Big\|_{L^\infty(\varOmega;\R^d)}\int_\Omega 
\big|[\widehat\OMEGA_\XX']^{\bm\xi_{\EPS k}}({\FF_{{\rm e}\,\EPS k}},\theta_{\EPS k})
{-}\theta_{\EPS k}[\OMEGA_\XX']^{\bm\xi_{\EPS k}}({\FF_{{\rm e}\,\EPS k}},\theta_{\EPS k})\big|\,\d\xx
\nonumber
\\&\nonumber
\le
C_\EPS\Big\|\pdt{\bm\xi_{\EPS k}}\Big\|_{L^\infty(\varOmega;\R^d)}\bigg(1+\int_\varOmega|\theta_{\EPS k}|^{1+\alpha}\,\d\xx\bigg).
\end{align}
The regularized dissipative and adiabatic heat contributions in the right-hand side of \eq{Euler3-Galerkin-est+++}
can be bounded from above by
$C_{\EPS}\|\FF_{\EPS k}\|_{L^\infty(\varOmega;\R^{d\times d})}(\|\nabla\vv_{\EPS k}\|_{L^\infty(\varOmega;\R^{d\times d})}+\|\LL_{{\rm p}\,\EPS k}\|_{L^\infty(\varOmega;\R^{d\times d}_{\rm dev})})(1+\|\theta_{\EPS k}\|_{L^{1+\alpha}(\varOmega)}^{1+\alpha})$ while the boundary term can be estimated by \eq{Euler-ass-h+}, taking also
into the account the extension \eq{extension-negative+}, as 
\begin{align}\nonumber
\int_\varGamma\frac {h(\theta_{\EPS k})}{1{+}\EPS|h(\theta_{\EPS k})|}\,\theta_{\EPS k}\,\d S
&\le C_{\EPS,a}+a\|\theta_{\EPS k}\|_{L^2(\varGamma)}^2
\\&\le C_{\EPS,a}+a N^2\big(\|\theta_{\EPS k}\|_{L^2(\varOmega)}^2
+\|\nabla\theta_{\EPS k}\|_{L^2(\varOmega;\R^d)}^2\big)\,\nonumber
\\&\le C_{\EPS,a}+a N^2(1+\|\theta_{\EPS k}\|_{L^{1+\alpha}(\varOmega)}^{1+\alpha})
+a N^2\|\nabla\theta_{\EPS k}\|_{L^2(\varOmega;\R^d)}^2\,,
\label{boundary-heat-est++++}
\end{align}
where $C_{\EPS,a}$ depends also on $C$ from \eq{Euler-ass-h+}, 
$N$ is the norm of the trace operator $H^1(\varOmega)\to L^2(\varGamma)$.
For $a>0$ in \eq{boundary-heat-est++++} sufficiently small, the last term
can be absorbed in the left-hand side of \eq{Euler3-Galerkin-est+++}, cf.\
\eq{boundary-heat-est++} below.

The convective term $\OMEGA(\FF_{\EPS k},\theta_{\EPS k})\vv_{\EPS k}{\cdot}\nabla\theta_{\EPS k}$
in \eq{Euler3-Galerkin-est+++} can be estimated using the linear growth of $\varpi$ with respect to $\theta$ as%
\begin{align}
&\!\!\int_\varOmega\OMEGA(\FF_{{\rm e}\,\EPS k}\theta_{\EPS k})\vv_{\EPS k}{\cdot}\nabla\theta_{\EPS k}\,\d\xx\nonumber
=
\int_\varOmega\varpi(\FF_{{\rm e}\,\EPS k},\theta_{\EPS k})\vv_{\EPS k}{\cdot}\nabla\theta_{\EPS k}+\varsigma(\alpha{-}1)|\theta_{\EPS k}|^{\alpha-1}\theta_{\EPS k}\vv_{\EPS k}{\cdot}\nabla\theta_{\EPS k}\,\d\xx\ \ \ \ \ \
\\
\!&=\int_\varOmega\varpi(\FF_{{\rm e}\,\EPS k},\theta_{\EPS k})\vv_{\EPS k}{\cdot}\nabla\theta_{\EPS k}+\varsigma\frac{\alpha{-}1}{\alpha{+}1}\vv_{\EPS k}{\cdot}\nabla|\theta_{\EPS k}|^{1+\alpha}\dx
\nonumber
\\
&\!=\int_\varOmega\varpi(\FF_{{\rm e}\,\EPS k},\theta_{\EPS k})\vv_{\EPS k}{\cdot}\nabla\theta_{\EPS k}-\varsigma\frac{\alpha{-}1}{\alpha{+}1}{\rm div}\vv_{\EPS k}|\theta_{\EPS k}|^{1+\alpha}\dx+\varsigma\frac{\alpha{-}1}{\alpha{+}1}\int_\varGamma |\theta_{\EPS k}|^{1+\alpha}\underbrace{\vv_{\EPS k}{\cdot}\nn}_{=0}{\rm d}S
\nonumber
\\
&\!\le C_a\|\vv_{\EPS k}\|^2_{L^\infty(\varOmega;\R^d)}\|\varpi(\FF_{\EPS k},\theta_{\EPS k})\|_{L^{\ONEALPH}(\varOmega)}^{\ONEALPH}+a
\|\nabla\theta_{\EPS k}\|^2_{L^2(\varOmega;\R^d)}+C\|{\rm div}\,\vv_{\EPS k}\|_{L^\infty(\varOmega)}\|\theta_{\EPS k}\|^{1+\alpha}_{L^{1+\alpha}(\varOmega))}
\nonumber
\\[.1em]
&\!\le C_a\|\vv_{\EPS k}\|^2_{L^\infty(\varOmega;\R^d)}\|\theta_{\EPS k}\|_{L^{\ONEALPH}(\varOmega)}^{\ONEALPH}+a
\|\nabla\theta_{\EPS k}\|^2_{L^2(\varOmega;\R^d)}+C\|{\rm div}\,\vv_{\EPS k}\|_{L^\infty(\varOmega)}\|\theta_{\EPS k}\|^{1+\alpha}_{L^{1+\alpha}(\varOmega))}
\nonumber
\\[.1em]
&\!\le C_a\|\vv_{\EPS k}\|^2_{L^\infty(\varOmega;\R^d)}(1{+}\|\theta_{\EPS k}\|_{L^{1+\alpha}(\varOmega)}^{1+\alpha})+a
\|\nabla\theta_{\EPS k}\|^2_{L^2(\varOmega;\R^d)}\!+C\|{\rm div}\,\vv_{\EPS k}\|_{L^\infty(\varOmega)}\|\theta_{\EPS k}\|^{1+\alpha}_{L^{1+\alpha}(\varOmega)}\!
\label{Euler-thermodynam3-test2}\end{align}
with some constant $C_a$ depending on $a$ and on $C_K$ from (\ref{Euler-ass+}f). Then the term $a\|\nabla\theta_{\EPS k}\|^2_{L^2(\varOmega;\R^d)}$ can be absorbed on the right--hand side of \eqref{Euler3-Galerkin-est+++}.

From \eq{Euler3-Galerkin-est+++}, we thus obtain the estimate%
\begin{align}\nonumber
&\frac{\d}{\d t}\int_\varOmega\!\widehat\OMEGA^{\bm\xi}({\FF_{{\rm e}\,\EPS k}})\,\d\xx
+\Big(\inf_{\XX,\FF,\theta}\kappa(\XX,\FF,\theta){-}a(1{+}N^2)\Big)\!\int_\varOmega\!|\nabla\theta_{\EPS k}|^2\,\d\xx
\\
&\nonumber\quad\le \Big(\Big\|\pdt{\FF_{{\rm e}\,\EPS k}}\Big\|_{L^r(\varOmega;\R^{d\times d})}\!+\Big\|\pdt{\bm\xi_{\EPS k}}\Big\|_{L^\infty(\varOmega;\R^d)}\Big)
(1+\|\theta_{\EPS k}\|_{L^{1+\alpha}(\varOmega)}^{1+\alpha})
\\
\nonumber&\qquad+C_\varepsilon\|\FF_{{\rm e}\,\EPS k}\|_{L^\infty(\varOmega;\R^{d\times d})}(\|\nabla\vv_{\EPS k}\|_{L^\infty(\varOmega;\R^{d\times d})}{+}\|\LL_{{\rm p}\,\EPS k}\|_{L^\infty(\varOmega;\R^{d\times d})})(1{+}\|\theta_{\EPS k}\|_{L^{1+\alpha}(\varOmega)}^{1+\alpha})\\
&\qquad+(C_a\|\vv_{\EPS k}\|_{L^\infty(\varOmega;\R^d)}+C\|{\rm div}\vv_{\EPS k}\|_{L^\infty(\varOmega;\R^d)})(1+\|\theta_{\EPS k}\|_{L^{1+\alpha}(\varOmega)}^{1+\alpha})\,.
\label{boundary-heat-est++}\end{align}
It follows from the definitions of $\widehat\omega$ and $\omega$ that $|\theta_{\EPS k}|^{1+\alpha}\le C\widehat\omega^{\bm\xi_{\EPS k}}(\FF_{{\rm e}\,\EPS k},\theta_{\EPS k})$
Thus, by the Gronwall inequality, exploiting \eqref{bound0c} and \eqref{DTFxi} we obtain
\begin{subequations}\label{Euler-quasistatic-est2}
\begin{align}\label{Euler-quasistatic-est2-theta+}
&\|\theta_{\EPS k}\|_{L^\infty(I;L^{1+\alpha}(\varOmega))\,\cap\,L^2(I;H^1(\varOmega))}^{}\le C\,.
\end{align}
Our next goal is to prove the strong convergence of the temperature field $\theta_{\EPS k}$. To this aim, we would like to apply some generalized version of the Aubin-Lions lemma. Note however that the heat equation contains the time derivative of $w_{\EPS k}=\omega^{\boldsymbol\xi_{\EPS k}}(\FF_{\EPS k},\theta_{\EPS k})$, rather than $\theta_{\EPS k}$. We therefore resort to first showing the strong convergence of $w_{\EPS k}$ and then relying on the invertibility and continuity of the map $\theta\mapsto\omega^{\boldsymbol{\xi}_{\EPS k}}(\FF_{\EPS k},\theta)$.

We begin by estimating $w_{\EPS k}=\omega^{\bm\xi_{\EPS k}}(\Fe,\theta)$. We recall the split \eqref{representation-omega} of $\omega(\XX,\Fe,\theta)$ into the sum of $c_1(\alpha-1)\theta^\alpha$ with $1<\alpha<2$ and a term $\varpi(\XX,\Fe,\theta)$ satisfying assumptions \eqref{Euler-ass-primitive-c+}. Accordingly, we have the split $\omega^{\bm\xi_{\EPS k}}(\Fe_{\EPS k},\theta_{\EPS k})=c_1(\alpha{-}1)\theta_{\EPS k}^\alpha+\varpi^{\bm\xi_{\EPS k}}(\Fe_{\EPS k},\theta_{\EPS k})$. Next, we notice that
    by \eqref{cutoff2-eps-k}, thanks to the hypothesis that $r>d$, there exists a compact set $K\subset{\rm GL}^+(d)$ such that $\Fe_{\EPS k}(t)\in K$ for a.e. $t\in I$. Thus, by assumption \eqref{Euler-ass-primitive-c+}, this implies that $\varpi^{\bm\xi_{\EPS k}}(\FF_{{\rm e}\EPS k},\theta_{\EPS k})\le C_K(1{+}\theta_{\EPS k})$ for some constant $C_K$, and that overall $w_{\EPS k}=\omega^{\bm\xi_{\EPS k}}(\FF_{{\rm e}\EPS k},\theta_{\EPS k})\le C(1+\theta_{\EPS k}^\alpha)$ for some constant $C$. Thus, by \eqref{Euler-quasistatic-est2-theta+}, we conclude that
    \begin{equation}\label{eq:w}
    	\|\W_{\EPS k}\|_{L^\infty(I;L^{1+1/\alpha}(\varOmega))}\le C_\varepsilon.
    \end{equation}
We next compute
\begin{equation}\label{eq:grad}
\nabla\W_{\EPS k}=[\OMEGA_\theta']^{\bm\xi_{\EPS k}}(\FF_{{\rm e}\,\EPS k},\theta_{\EPS k})\nabla\theta_{\EPS k}
+[\OMEGA_\Fe']^{\bm\xi_{\EPS k}}(\FF_{{\rm e}\,\EPS k},\theta_{\EPS k})\nabla\FF_{{\rm e}\,\EPS k}
+[\OMEGA_\XX']^{\bm\xi_{\EPS k}}(\FF_{{\rm e}\,\EPS k},\theta_{\EPS k})\nabla\bm\xi_{\EPS k}	.
\end{equation}
Using again the fact that $\Fe_{\EPS k}(t)$ is {valued} in some compact set $K$ for a.e. $t\in I$, we can exploit assumptions \eqref{Euler-ass-primitive-c+} to deduce that  $[\varpi'_\theta]^{\bm\xi_{\EPS k}}(\Fe_{\EPS k},\theta_{\EPS k})\le C_K$ a.e. in $I{\times}\varOmega$ and that by \eqref{Euler-quasistatic-est2-theta+}, 
 we see that $[\omega'_\theta]^{\bm\xi_{\EPS k}}(\Fe_{\EPS k},\theta_{\EPS k})$ is bounded by $C(1+\theta_{\EPS k}^{\alpha-1})$ for some constant $C$. Thus, by invoking the estimate \eqref{Euler-quasistatic-est2-theta+}, we realize that 
 $[\OMEGA_\theta']^{\bm\xi_{\EPS k}}(\FF_{{\rm e}\,\EPS k},\theta_{\EPS k})$ is bounded in
$L^\infty(I;L^{(1+\alpha)/(\alpha-1)}(\varOmega))$. Therefore, by \eqref{Euler-quasistatic-est2-theta+},
\begin{equation}\label{eq:111}
	\big\|[\OMEGA_\theta']^{\bm\xi_{\EPS k}}(\FF_{{\rm e}\,\EPS k},\theta_{\EPS k})\nabla\theta_{\EPS k}\big\|_{L^2(I;L^{2(1+\alpha)/(3\alpha-1)})}\le C_\varepsilon.
\end{equation}
Likewise, using again \eqref{representation-omega} together with \eqref{Euler-ass-primitive-c+}, we see that $[\OMEGA_\Fe']^{\bm\xi_{\EPS k}}(\FF_{{\rm e}\,\EPS k},\theta_{\EPS k})$ grows at most as $\theta^\alpha$; thus, by \eqref{Euler-quasistatic-est2-theta+} we have that
$[\OMEGA_\Fe']^{\bm\xi_{\EPS k}}(\FF_{{\rm e}\,\EPS k},\theta_{\EPS k})$
is bounded in $L^\infty(I;L^{%
{1+1/\alpha}}(\varOmega;\R^{d\times d}))$. As already observed, $\nabla\FF_{\EPS k}$ is bounded in $L^\infty(I;L^r(\varOmega;\R^{d\times d\times d}))$, hence
\begin{equation}\label{eq:112}	
\big\|[\OMEGA_\Fe']^{\bm\xi_{\EPS k}}(\FF_{{\rm e}\,\EPS k},\theta_{\EPS k})\nabla\FF_{{\rm e}\,\EPS k}\big\|_{L^\infty(I;L^{r(\alpha+1)/((r+1)\alpha+1)}(\Omega))}\le C.
\end{equation}
Finally, 
$[\OMEGA_\XX']^{\bm\xi_{\EPS k}}(\FF_{{\rm e}\,\EPS k},\theta_{\EPS k})$ is bounded in
$L^\infty(I;L^{1+1/\alpha}(\varOmega;\R^d))$ due to
\eqref{Euler-quasistatic-est2-theta+} and \eq{Euler-ass-primitive-c+} whereas $\nabla\bm\xi_{\EPS k}$
is bounded in 
$L^\infty(I{\times}\varOmega;\R^{d\times d})$, as previously shown in 
\eqref{bounds0}. Accordingly, 
\begin{equation}\label{eq:113}
\big\|[\OMEGA_\XX']^{\bm\xi_{\EPS k}}(\FF_{{\rm e}\,\EPS k},\theta_{\EPS k})\nabla\bm\xi_{\EPS k}	\big\|_{L^\infty(I;L^{1+1/\alpha}(\Omega))}\le C.	
\end{equation}
It follows from \eqref{eq:grad}, \eqref{eq:111}, \eqref{eq:112} and \eqref{eq:113}, that $\|\nabla w\|_{L^2(I;L^{
\sigma}(\varOmega))}\le C$ with $\sigma=\min\big(\frac{2(\alpha+1)}{3\alpha-1},\frac{r(\alpha+1)}{(r+1)\alpha+1},1{+}\frac1\alpha\big)$. Hence, using also
\eqref{eq:w}, we arrive at
\begin{align}
\|\W_{\EPS k}\|_{L^\infty(I;L^{1+1/\alpha}(\varOmega))\,\cap\,L^2(I;W^{1,
\sigma}(\varOmega))}^{}\le C\,.%
\label{Euler-weak-sln-w}
\end{align}
\end{subequations}
Note that $\sigma>1$ if $r>\alpha+1$. %

\medskip\noindent{\it Step 5: Limit passage for $k\to\infty$}.
Using the Banach selection principle, we can extract some subsequence of
$(\vv_{\EPS k},\bm\xi_{\EPS k},\FF_{{\rm p}\,\EPS k},\LL_{{\rm p}\,\EPS k},\theta_{\EPS k},w_{\EPS k},\rho_{\EPS k})$ and its limit
$(\vv_\EPS,\bm\xi_\EPS,\FF_{\rm p\,\EPS},\LL_{\rm p\,\EPS},\theta_\EPS,w_\EPS,\rho_\EPS):
I\to W^{2,p}(\varOmega;\R^d)\times W^{2,r}(\varOmega;\R^{d})\times W^{1,r}(\varOmega;\R^{d\times d}_{\rm dev})\times W^{1,q}(\varOmega;\R^{d\times d}_{\rm dev})\times H^1(\varOmega)^2\times W^{1,r}(\varOmega)$
such that %
\begin{subequations}\label{Euler-weak-sln}
\begin{align}
&\!\!\vv_{\EPS k}\to\vv_\EPS&&\text{weakly* in $L^\infty(I;L^2(\varOmega;\R^d))\cap\
L^p(I;W^{2,p}(\varOmega))$}\,,
\\\label{Euler-weak-sln-v}
&\!\!\bm\xi_{\EPS k}\to\bm\xi_\EPS&&\text{weakly* in $\
L^\infty(I;W^{2,r}(\varOmega;\R^d))\cap
W^{1,p}(I;W^{1,r}(\varOmega;\R^d))$,}\!\!&&
\\\label{Euler-weak-sln-v2}
&\!\!\FF_{{\rm p}\,\EPS k}\to\FF_{\rm p\,\EPS}
\!\!\!&&\text{weakly* in $\ L^\infty(I;W^{1,r}(\varOmega;\R^{d\times d}))\,\cap\,
W^{1,\min(p,q)}(I;L^p(\varOmega;\R^{d\times d}))$,}\!\!
\\\label{Euler-weak-sln-v3}
&\!\!\LL_{{\rm p}\,\EPS k}\to\LL_{{\rm p}\,\EPS}&&\text{weakly* in $\
L^{\infty}(I;W^{1,q}(\varOmega;\R^{d\times d}_{\rm dev}))$,}\!\!&&
\\
&\!\!\W_{\EPS k}\to\W_\EPS&&
\text{weakly* in $\ L^\infty(I;L^{1+1/\alpha}(\varOmega))\,\cap\,L^2(I;W^{1,\sigma
}(\varOmega))$}\,,
\\
&\!\!\theta_{\EPS k}\to\ \theta_\EPS&&
\text{weakly* in $\ L^\infty(I;L^{1+\alpha}(\varOmega))\,\cap\,L^2(I;H^1(\varOmega))$}\,,
\\
  &\!\!\rho_{\EPS k}\to\ \rho_\EPS&&
\text{weakly* in $\ L^\infty(I;W^{1,r}(\varOmega))$}.
\end{align}
\end{subequations}
Relying on the assumption $r>d$, and on the estimates \eqref{DTFxi} and \eqref{eq:12}, by the Aubin-Lions lemma we also have that
\begin{subequations}\label{rho-conv+}
\begin{align}\label{rho-conv1}
&{\bm\xi}_{\EPS k}\to{\bm\xi}_\EPS\ &&\text{ strongly in }\ C(I{\times}\barvarOmega;\R^{d}),&&&&&&\\
&{\nabla\bm\xi}_{\EPS k}\to\nabla{\bm\xi}_\EPS\!\!\!
&&\text{ strongly in }\ C(I{\times}\barvarOmega;\R^{d\times d}),\label{rho-conv2}&&\\ 
&\FF_{{\rm p}\,\EPS k}\to\FF_{{\rm p}\,\EPS}\!\!\!&&\text{ strongly in
             $C(I{\times}\barvarOmega;\R^{d\times d})$,}&&\label{rho-conv3}
                                                                                             \\ 
&\rho_{\EPS k}\to\rho_{\EPS}&&\text{ strongly in
$C(I{\times}\barvarOmega)$.}&&\label{rho-conv4}
\end{align}
\end{subequations}
It is shown (see \cite{Roub22QHLS,Roub22TVSE}) that the evolution-and-transport equations
\eq{transport-equations+} and \eqref{transport-rho} with the corresponding initial
conditions define (weak,weak${}^*$)-continuous mappings $\vv_{\EPS k}\mapsto \bm\xiek$,
$\vv_{\EPS k}\mapsto \FF_{{\rm p}\,\EPS k}$, $\vv_{\EPS k}\mapsto \rho_{\EPS k}$ in the corresponding
functions spaces. Thus the weak convergences $(\ref{rho-conv+}{\rm b,\!c})$ already allows
for the limit passage in the evolution equations \eq{transport-equations+} and
\eqref{transport-rho}
We also have the strong convergence of the density
\begin{equation}\label{conv-rho}
	\rhoRxiepsk{\det}(\nabla\bm\xi_{\EPS k})\to 
	\rhoRxieps{\det}(\nabla\bm\xi_{\EPS})\quad\text{strongly in }C(I{\times}\barvarOmega)\ \ \text{ with }\ \rho_\EPS=\rhoRxieps{\det}(\nabla\bm\xi_{\EPS})\,.
      \end{equation}
We now derive bounds on $\pdt{}w_{\EPS k}$, which willl allow us to derive the strong convergence of $w_{\EPS k}$ through the (generalized) Aubin-Lions lemma, and hence strong convergence of $\theta_{\EPS k}$.
    
      By comparison in \eq{Euler3-weak-Galerkin+} 
we obtain a bound on $\pdt{}\W_{\EPS k}$ in seminorms $|\cdot|_l$ on $L^2(I;H^1(\varOmega)^*)$
arising from this Galerkin approximation:
$$
\big|f\big|_l^{}:=\sup\limits_{\stackrel{{\scriptstyle{\widetilde\theta(t)\in Z_l\ \text{for }t\in I}}}{{\scriptstyle{\|\widetilde\theta\|_{L^2(I;H^1(\varOmega))}^{}\le1}}}}\int_0^T\!\!\!\int_\varOmega f\widetilde\theta\,\d\xx\d t\,.
$$
More specifically, for any $k\ge l$, we can estimate
\begin{align}\nonumber
&\Big|\pdt{\W_{\EPS k}}\Big|_l^{}=
\sup\limits_{\stackrel{{\scriptstyle{\widetilde\theta(t)\in Z_l\ \text{for }t\in I}}}{{\scriptstyle{\|\widetilde\theta\|_{L^2(I;H^1(\varOmega))}^{}\le1}}}}
\!\int_0^T\!\!\!\int_\varGamma\!
\frac {h(\theta_{\EPS k})}{1+\EPS |h(\theta_{\EPS k})|}
\widetilde\theta\,\d S\d t
-\!\int_0^T\!\!\!\int_\varOmega\bigg(
\kappa^{\bm\xi_{\EPS k}}(\FF_{{\rm e}\,\EPS k},\theta_{\EPS k})\nabla\theta_{\EPS k}\big)
{\cdot}\nabla\widetilde\theta
\\[-.2em]&\hspace{7em}\nonumber
-\Big(\frac{\nu_0|\ee(\vv_{\EPS k})|^p+\nu_1|\nabla\EE(\vv_{\EPS k})|^p+\partial_{\Lp}\!\zeta(\theta_{\EPS k},\LL_{{\rm p}\,\EPS k}){:}\LL_{{\rm p}\,\EPS k}+\nu_2|\nabla\LL_{{\rm p}\,\EPS k}|^q\!\!}{1{+}\EPS|\EE(\vv_{\EPS k})|^q{+}\EPS|\nabla\EE(\vv)|^p+\EPS\partial_{\Lp}\!\zeta(\theta_{\EPS k},\FF_{{\rm p}\,\EPS k}){:}\LL_{{\rm p}\,\EPS k}+\EPS|\nabla\LL_{{\rm p}\,\EPS k}|^q}
\\&\hspace{7em}%
+\frac{\pi_\LAM(\Fe_{\EPS k})[\COUPLING'_{\Fe}]^{\bm\xi_{\EPS k}}(\Fe_{\EPS k},\theta_{\EPS k})\FF_{{\rm e},\EPS k}^\top
}{\det(\FF_{{\rm e}\,\EPS k})(1{+}\EPS|\theta_{\EPS k}|^{\alpha%
})}{:}\ee(\vv_{\EPS k})
\Big)\widetilde\theta\,\bigg)\,\d\xx\d t\le C_\varepsilon
\label{Euler3-weak-Galerkin++}
\end{align}
with some $C_\varepsilon$ that depends on the regularization parameter $\EPS$ but is independent on $l\in N$. 
Thus, by a generalized Aubin-Lions theorem
\cite[Ch.8]{Roub13NPDE}, we obtain %
\begin{subequations}\label{w-conv}
\begin{align}
&\label{w-conv+}
\W_{\EPS k}\to \W_\EPS\hspace*{-0em}&&
\hspace*{-3em}\text{strongly in $L^s(I{\times}\varOmega)$ for $1\le s<2+4/d$}.
\intertext{Since $\OMEGA(\XX,\Fe,\cdot)$ is increasing for fixed $\XX$ and $\Fe$, we can
write $\theta_{\EPS k}=[\OMEGA^{\bm\xi_{\EPS k}}(\Fe_{\EPS k},\cdot)]^{-1}(\W_{\EPS k})$. 
Thanks to the continuity of
$(\XX,\FF,\W)\mapsto[\OMEGA(\XX,\FF,\cdot)]^{-1}(\W):\Omega\times\R^{d\times d}\times\R\to\R$
and the at most linear growth with respect to $\W$ uniformly
 with respect to $\XX$ from $\Omega$ and $\FF$ from any compact $K\subset{\rm GL}^+(d)$, we have also
 }
&\label{z-conv+}
\theta_{\EPS k}\to \theta_\EPS=[\OMEGA^{\bm\xi_{\EPS k}}(\Fe_\EPS,\cdot)]^{-1}(\W_\EPS)
\hspace*{-0em}&&\hspace*{-1em}\text{strongly in $L^s(I{\times}\varOmega)$
for $1\le s<2+4/d$}\,;
\intertext{%
actually, the bound $2{+}4/d$ in \eq{w-conv+} and \eq{z-conv+}
valid for $\alpha\ge1$ comes
from an embedding and interpolation of \eq{Euler-quasistatic-est2-theta+} and
is actually not optimal if $\alpha\ne1$ but sufficient for our purposes.
Note that we do not have any direct information about
$\pdt{}\theta_{\EPS k}$
so that we could not use the Aubin-Lions arguments straight for
$\{ \theta_{\EPS k}\}_{k\in\N}$. Thus, by the continuity of the corresponding
Nemytski\u{\i} (or here simply superposition) mappings, also the 
conservative part of the regularized Cauchy stress
as well as the heat part of the internal energy,  namely
}
&
\TT^{\bm\xi_{\EPS k}}_{\LAM,\EPS}(\Fe_{\EPS k},\theta_{\EPS k})\to \TT^{\bm\xi_{\EPS}}_{\LAM,\EPS}(\Fe_\EPS,\theta_\EPS)
\hspace*{-0em}&&\hspace*{0em}\text{strongly in
$L^c(I{\times}\varOmega;\R_{\rm sym}^{d\times d}),\ \ 1\le c<\infty$,}
\label{Euler-T-strong-conv+}
\\&\nonumber
\frac{\pi_\LAM(\Fe_{\EPS k})[\COUPLING'_{\Fe}]^{\bm\xi_{\EPS k}}(\Fe_{\EPS k},\theta_{\EPS k})\Fe_{\EPS k}^\top
}{\det(\Fe_{\EPS k})(1{+}\EPS|\theta_{\EPS k}|^{\beta%
})}
\\&\hspace{2em}
\to \frac{\pi_\LAM(\Fe_\EPS)[\COUPLING'_{\Fe}]^{\bm\xi_{\EPS}}(\Fe_\EPS,\theta_\EPS)\Fe_\EPS^\top
}{\det(\Fe_\EPS)(1{+}\EPS|\theta_\EPS|^{\beta%
})}
\hspace*{0em}&&\hspace*{0em}\text{strongly in
$L^c(I{\times}\varOmega;\R_{\rm sym}^{d\times d}),\ \ 1\le c<\infty$},
\\[-.0em]&\OMEGA^{\bm\xi_{\EPS k}}(\Fe_{\EPS k},\theta_{\EPS k})\to
\OMEGA^{\bm\xi_{\EPS}}(\Fe_\EPS,\theta_\EPS)\hspace*{-0em}&&\hspace*{0em}\text{strongly in $L^c(I{\times}\varOmega),\ \ 1\le c<2{+}4/d$,}\label{Euler-weak+w+}
\end{align}
\end{subequations}
where $\Fe_\EPS=(\nabla\bm\xi_{\EPS})^{-1}\FF_{{\rm p}\,\EPS}^{-1}$.

We now focus our attention to the density of linear momentum $\rho_{\EPS k}\vv_{\EPS k}$.
We notice that
$\nabla(\rho_{\EPS k}\vv_{\EPS k})=\nabla\rho_{\EPS k}{\otimes}\vv_{\EPS k}
+\rho_{\EPS k}\nabla\vv_{\EPS k}$
      is bounded in $L^p(I;L^r(\varOmega;\R^{d}))$ due to the already obtained bounds  \eqref{bound0c} and \eqref{eq:11}. Moreover, by \eqref{rho-conv4} $\rho_{\EPS k}$ converges strongly to $\rho_\EPS$, while by (\ref{Euler-weak-sln}a) $\vv_{\EPS k}$ converges weakly. Hence we have
$\rho_{\EPS k}\vv_{\EPS k}\to\rho_\EPS\vv_\EPS%
\text{ weakly in }L^p(I;W^{1,r}(\varOmega;\R^{d}))$.
We can actually prove that the above convergence is strong {but in a weaker topology}.
Indeed, by comparison in the Galerkin approximation of the balance of linear momentum
we also obtain information on $\pdt{}(\rho_{\EPS k}\vv_{\EPS k})$. More precisely, by
(\ref{Euler-weak-Galerkin+}a),
      \begin{align}
        \nonumber
        &\int_0^T\!\!\int_{\varOmega} \frac{\partial}{\partial t}\left(\varrho_{\varepsilon k} \vv_{\varepsilon k}\right) {\cdot}\widetilde{\vv}\,\d\xx\d t
  =\int_0^T\!\!\int_\varOmega\Big(\rho_{\EPS k}\gg{\cdot}\widetilde\vv
\\\nonumber&\qquad
        +\big(\rho_{\EPS k}\vv_{\EPS k}{\otimes}\vv_{\EPS k}-\TT_{\lambda,\EPS}^\xiek(\FF_{{\rm e}\,\EPS k},\theta_{\EPS k})-\nu_0|\ee(\vv_{\EPS k})|^{p-2}\ee(\vv_{\EPS k})\big){:}\ee(\widetilde\vv)
          \\
 &\qquad\qquad-\nu_1|\nabla\ee(\vv_{\EPS k})|^{p-2}\nabla\ee(\vv_{\EPS k}){\vdots}\nabla\ee(\widetilde\vv)\Big)\dx\dt\le C_\EPS\|\widetilde\vv\|_{L^p(I;W^{2,p}(\varOmega;\R^d))}^{}\,.
        \label{ddt}
      \end{align}
Thus, $\pdt{}(\rho_{\EPS k}\vv_{\EPS k})$ is bounded in the dual space
$L^{p'}(I;W^{2,p}(\varOmega;\R^d)^*)$ with respect to the family of seminorms induced by the
Galerkin discretization by $\{W_l\}_{%
{l\le k}}$ {independently of $k$}.
 By a generalization of the Aubin-Lions compactness theorem \cite[Ch.8]{Roub13NPDE},
 we have%
      \begin{equation}\label{strong-momentum}
        \rho_{\EPS k}\vv_{\EPS k}\to\rho_{\EPS}\vv_{\EPS}\quad\text{strongly in }L^p(I;L^s(\varOmega;\R^d))\quad\text{for all }2\le s<\infty.
      \end{equation}
Since $\vv_{\EPS k}=\rho_{\EPS k}\vv_{\EPS k}/\rho_{\EPS k}$, from the strong convergence of
$\rho_{\EPS k}$ in \eqref{rho-conv4} we obtain
\begin{equation}\label{strong-v}
      	\vv_{\EPS k}\to\vv_{\EPS}\quad\text{strongly in }L^p(I;L^s(\varOmega;\R^d))\quad\text{for all }2\le s<\infty.
      \end{equation}
We also notice that  $\varrho_{\EPS k}(T)$ and $\vv_{\EPS k}(T)$ are well defined, since $\rho_{\EPS k}$ and $\vv_{\EPS k}$ are weakly continuous. Moreover, they are both bounded
by \eqref{eq:11} and \eqref{pippo:1} in $W^{1,r}(\varOmega)$ and $L^2(\varOmega;\R^d)$,
respectively. Thus, their product is bounded in the space
$L^2(\varOmega;\R^d)$ (recall that $r>d$) and therefore it converges to a limit. 
To identify the limit, take $\widetilde\vv\in V_l$ for any $l\in\mathbb N$. Then, using the bound \eqref{ddt}, we have also that the time derivative $\pdt{}(\rho_{\EPS k}\vv_{\EPS k})$ converges to some limit, which by \eqref{strong-momentum} coincides with the distributional time derivative of $\rho_\EPS\vv_\EPS$. Moreover, denoting $\rho_0=\rhoRxizero$, it holds
      \begin{align}
        \int_\varOmega \rho_{\EPS k}(T)\vv_{\EPS k}(T)\widetilde\vv\dx&=\int_\varOmega\rho_0\vv_0\widetilde\vv\dx+\int_0^T\!\! \int_\varOmega\pdt{(\rho_{\EPS k}\vv_{\EPS k})}\widetilde\vv\dx\dt
        \nonumber
        \\
        &\!\!\!\!\stackrel{k\to\infty}\to\!\!\int_\varOmega\rho_0\vv_0\widetilde\vv\dx+\int_0^T\!\! \Big\langle\pdt{(\rho_{\EPS}\vv_{\EPS})},\widetilde\vv\Big\rangle\dt=\int_\varOmega\rho_\EPS(T)\vv_{\EPS}(T)\widetilde\vv\dx.
        \end{align}
        Thus we conclude that
        $\rho_{\EPS k}(T)\vv_{\EPS k}(T)\to \rho_\EPS(T)\vv_\EPS(T)$ %
        {  weakly in }$L^2(\varOmega;\R^d)$.
Such convergence allows us to perform the following calculation
\begin{align}\nonumber
\int_0^T & \int_{\varOmega}\left(\frac{\partial}{\partial t}\left(\rho_{\varepsilon k} \vv_{\varepsilon k}\right)+\operatorname{div}\left(\rho_{\varepsilon k} \vv_{\varepsilon k} {\otimes} \vv_{\varepsilon k}\right)\right){\cdot}\widetilde{\vv}\,\d\xx\d t \\
\quad &\nonumber
=\int_{\varOmega} \rho_{\varepsilon k}(T) \vv_{\varepsilon k}(T){\cdot}\widetilde{\vv}(T)-\rho_0 \vv_0{\cdot}\widetilde{\vv}(0)\,\,\d\xx
-\int_0^T\!\!\!\int_{\varOmega} \rho_{\varepsilon k} \vv_{\varepsilon k}{\cdot}\frac{\partial \widetilde{\vv}}{\partial t}+\left(\rho_{\varepsilon k} \vv_{\varepsilon k} {\otimes} \vv_{\varepsilon k}\right){:}\nabla \widetilde{\vv}\,\d\xx\d t
\\\nonumber&
\!\!\stackrel{k \rightarrow \infty}{\longrightarrow} \int_{\varOmega} \rho_{\varepsilon}(T) \vv_{\varepsilon}(T){\cdot}\widetilde{\vv}(T)-\rho_0 \vv_0{\cdot}\widetilde{\vv}(0)\,\d\xx
-\int_0^T\!\!\!\int_{\varOmega} \rho_{\varepsilon} \vv_{\varepsilon}{\cdot}\frac{\partial \widetilde{\vv}}{\partial t}+\left(\rho_{\varepsilon} \vv_{\varepsilon} {\otimes} \vv_{\varepsilon}\right){:}\nabla \widetilde{\vv}\,\d\xx\d t \\
&=\int_0^T \int_{\varOmega}\left(\frac{\partial}{\partial t}\left(\rho_{\varepsilon} \vv_{\varepsilon}\right)+\operatorname{div}\left(\rho_{\varepsilon} \vv_{\varepsilon} {\otimes} \vv_{\varepsilon}\right)\right){\cdot}\widetilde{\vv}\,\,\d\xx\d t\,.
\label{conv2}
\end{align}

We next note that
\begin{align}
\nonumber 
&\int_\varOmega \Big(\pdt{}(\rho_{\EPS k}\vv_{\EPS k})+{\rm div}(\rho_{\EPS k}\vv_{\EPS k}\otimes\vv_{\EPS k})\Big){\cdot}\vv_{\EPS k}\dx\\
&\quad =\int_\varOmega \pdt{\rho_{\EPS k}}\vv_{\EPS k}{\cdot}\vv_{\EPS k}+\vv_{\EPS k}{\cdot}\vv_{\EPS k}{\rm div}(\rho_{\EPS k}\vv_{\EPS k})+\rho_{\EPS k}(\vv_{\EPS k}{\cdot}\nabla\vv_{\EPS k})\vv_{\EPS k}\dx\nonumber\\
&\quad =\!\int_\varOmega\!\rho_{\EPS k}\DT\vv_{\EPS k}{\cdot}\vv_{\EPS k}\dx
=\!\int_\varOmega\!\rho_{\EPS k}\Big({{\frac {|\vv_{\EPS k}|^2}2}}\Big)^{\!\!\!\DT{}}\dx\nonumber
=\!\int_\varOmega\!\rho_{\EPS k}\pdt{}\Big(\frac{|\vv_{\EPS k}|^2}{2}\Big)+\rho_{\EPS k}\vv_{\EPS k}{\cdot}\nabla\Big(\frac{|\vv_{\EPS k}|^2}{2}\Big)\dx\nonumber
\\&\quad
=\int_\varOmega \rho_{\EPS k}\pdt{}\Big(\frac{|\vv_{\EPS k}|^2}{2}\Big)-\underbrace{{\rm div}(\rho_{\EPS k}\vv_{\EPS k})}_{=-\pl\rho_{\EPS k}/\pl t}\frac{|\vv_{\EPS k}|^2}{2}\dx+\int_\varGamma\frac{|\vv_{\EPS k}|^2}{2}\underbrace{(\rho_{\EPS k}\vv_{\EPS k}{\cdot}\nn)}_{=0}{\rm d}S\nonumber\\
&\quad=\int_\varOmega \pdt{}\Big(\rho_{\EPS k}\frac{|\vv_{\EPS k}|^2}{2}\Big)\dx=\frac{\rm d}{{\rm d}t}\int_\varOmega\rho_{\EPS k}\frac{|\vv_{\EPS k}|^2}{2}\dx.\label{eq:energy1}
\end{align}
This implies 
\begin{align}
\int_{\varOmega}& \frac{\rho_{\varepsilon k}(T)}{2}\left|\vv_{\varepsilon k}(T){-}\vv_{\varepsilon}(T)\right|^2\,\d\xx =\int_0^T \int_{\varOmega}\left(\frac{\partial}{\partial t}\left(\rho_{\varepsilon k} \vv_{\varepsilon k}\right)+\operatorname{div}\left(\rho_{\varepsilon k} \vv_{\varepsilon k} {\otimes} \vv_{\varepsilon k}\right)\right){\cdot}\vv_{\varepsilon k}\,\d\xx\d t
\nonumber\\[-.4em]
&\qquad\qquad\qquad\qquad\qquad
+\int_{\varOmega} \frac{\rho_0}{2}\left|\vv_0\right|^2\!-\rho_{\varepsilon k}(T) \vv_{\varepsilon k}(T){\cdot}\vv_{\varepsilon}(T)+\frac{\rho_{\varepsilon k}(T)}{2}\left|\vv_{\varepsilon}(T)\right|^2\,\d\xx.
\label{identity}
\end{align}    
We now use the Galerkin approximation of the regularized momentum equation
\eq{Euler1-weak-Galerkin+} tested by $\widetilde\vv=\vv_{\EPS k}-\widetilde\vv_k$
with $\widetilde\vv_k:I\to V_k$ an approximation of $\vv_\EPS$ in the sense
that $\vv_{\EPS k}\to\vv_\EPS$ strongly in $L^\infty(I;L^2(\varOmega;\R^d))$, 
$\EE(\widetilde\vv_k)\to\EE(\vv_\EPS)$ strongly in
$L^p(I{\times}\varOmega;\R^{d\times d})$, and
$\Nabla\EE(\widetilde\vv_k)\to\Nabla\EE(\vv_\EPS)$ 
strongly in $L^p(I{\times}\varOmega;\R^{d\times d\times d})$ for $k\to\infty$.
We can estimate%
\begin{align}\nonumber
 &\frac1{2C_\EPS}\|\vv_{\EPS k}(T){-}\vv_\EPS(T)\|^2_{L^2(\varOmega;\R^d)}\!\!
 +\nu_0c_p\|\EE(\vv_{\EPS k}{-}\vv_\EPS)\|_{L^p(I{\times}\varOmega;\R^{d\times d})}^p\!\!
+\nu_1 c_{p}\|\nabla\EE(\vv_{\EPS k}{-}\vv_\EPS)\|_{L^p(I{\times}\varOmega;\R^{d\times d\times d})}^p
\\&\nonumber
\le\int_\varOmega\frac{\rho_{\EPS k}}2|\vv_{\EPS k}(T){-}\vv_\EPS(T)|^2\dx
+\int_0^T\!\!\!\int_\varOmega\!\bigg(
\nu_0\big(|\ee(\vv_{\EPS k})|^{p-2}\ee(\vv_{\EPS k}){-}|\EE(\vv_\EPS)|^{p-2}\EE(\vv_\EPS))\big)
{:}\EE(\vv_{\EPS k}{-}\vv_\EPS)
 \\[-.6em]&\hspace*{4em}\nonumber
 +\nu_1\big(|\nabla\EE(\vv_{\EPS k})|^{p-2}\nabla\EE(\vv_{\EPS k})
{-}|\nabla\EE(\vv_\EPS)|^{p-2}\nabla\EE(\vv_\EPS)\big)\Vdots
  \nabla\EE(\vv_{\EPS k}{-}\vv_\EPS)\bigg)\,\d\xx\d t
 \\[-.5em]&\nonumber
=\int_\varOmega \frac{\rho_{\EPS k}}2|\vv_{\EPS k}(T)-\vv_\EPS(T)|^2\dx
\\
\nonumber
&\quad+\int_0^T\!\!\!\int_\varOmega\!\bigg(
\nu_1|\nabla\EE(\vv_{\EPS k})|^{p-2}\nabla\EE(\vv_{\EPS k})\Vdots
  \nabla\EE(\vv_{\EPS k}{-}\widetilde\vv_k)
-\nu_1|\nabla\EE(\vv_\EPS)|^{p-2}\nabla\EE(\vv_\EPS)\big)\Vdots
  \nabla\EE(\vv_{\EPS k}{-}\vv_\EPS)
  \\[-.6em]&\hspace*{3em}\nonumber
 +\nu_0|\ee(\vv_{\EPS k})|^{p-2}\ee(\vv_{\EPS k}){:}\ee(\vv_{\EPS k}-\widetilde\vv_{k})
 -\nu_0|\EE(\vv_\EPS)|^{p-2}\EE(\vv_\EPS){:}\EE(\vv_{\EPS k}{-}\vv_\EPS)
\bigg)\,\d\xx\d t+\mathscr O_k
 \\[-.5em]&=\nonumber
 \int_\varOmega\frac{\rho_{\EPS k}}2|\vv_{\EPS k}(T)-\vv_\EPS(T)|^2\dx
 +\int_0^T\!\!\!\int_\varOmega\bigg(\frac{\rhoRxiepsk\GRAVITY}{\det_{\LAM}(\FF_{\EPS k})\!\!}{\cdot}(\vv_{\EPS k}{-}\widetilde\vv_k)
 -\TT^{\bm\xi_{\EPS k}}_{\LAM,\EPS}(\FF_{{\rm e}\,\EPS k},\theta_{\EPS k}){:}\ee(\vv_{\EPS k}{-}\widetilde\vv_k)
 \\&\hspace{2em}
 -\Big(\pdt{(\rho_{\EPS k}\vv_{\EPS k})}+{\rm div}(\rho_{\EPS k}(\vv_{\EPS k}{\otimes}\vv_{\EPS k})\Big)\cdot(\vv_{\EPS k}-\widetilde\vv_k)
 \nonumber
 \\[-.4em]&\nonumber\hspace{2em}
 -\nu_0|\EE(\vv_\EPS)|^{p-2}\EE(\vv_\EPS)
{:}\EE(\vv_{\EPS k}{-}\vv_\EPS)
-\nu_1\big(|\nabla\EE(\vv_\EPS)|^{p-2}\nabla\EE(\vv_\EPS)\big)\Vdots
 \nabla\EE(\vv_{\EPS k}{-}\vv_\EPS)\bigg)
 \,\d\xx\d t+\mathscr{O}_k
\\
\nonumber &
\!\!\stackrel{\eqref{identity}}=
\int_0^T\!\!\!\int_\varOmega\Big(\frac{\rhoRxiepsk\GRAVITY}{\det_{\LAM}(\FF_{\EPS k})\!\!}{\cdot}(\vv_{\EPS k}{-}\widetilde\vv_k)
 -\TT^{\bm\xi_{\EPS k}}_{\LAM,\EPS}(\FF_{{\rm e}\,\EPS k},\theta_{\EPS k}){:}\ee(\vv_{\EPS k}{-}\widetilde\vv_k)
 \\&\hspace{4em}+\Big(\pdt{(\rho_{\EPS k}\vv_{\EPS k})}+{\rm div}(\rho_{\EPS k}(\vv_{\EPS k}{\otimes}\vv_{\EPS k})\Big)\cdot\widetilde\vv_k
 \nonumber
 &\quad 
 \\[.1em]&\nonumber\hspace{4em}
 -\nu_0|\EE(\vv_\EPS)|^{p-2}\EE(\vv_\EPS)
{:}\EE(\vv_{\EPS k}{-}\vv_\EPS)
-\nu_1\big(|\nabla\EE(\vv_\EPS)|^{p-2}\nabla\EE(\vv_\EPS)\big)\Vdots
 \nabla\EE(\vv_{\EPS k}{-}\vv_\EPS)\Big)
 \,\d\xx\d t\nonumber
 \\
 &\hspace{2em}+\int_\varOmega\frac{\rho_0}2|\vv_0|^2-\rho_{\EPS k}(T)\vv_{\EPS k}(T)\cdot\vv_\EPS(T)+\frac{\rho_{\EPS k}(T)}2|\vv_\EPS(T)|^2\dx
 +\mathscr{O}_k\nonumber
 \\
\nonumber &
=
\int_0^T\!\!\!\int_\varOmega\bigg(\frac{\rhoRxiepsk\GRAVITY}{\det_{\LAM}(\FF_{\EPS k})\!\!}{\cdot}(\vv_{\EPS k}{-}\widetilde\vv_k)
       -\TT^{\bm\xi_{\EPS k}}_{\LAM,\EPS}(\FF_{{\rm e}\,\EPS k},\theta_{\EPS k}){:}\ee(\vv_{\EPS k}{-}\widetilde\vv_k)
 \\&\hspace{4em}+\Big(\pdt{(\rho_{\EPS k}\vv_{\EPS k})}
 +{\rm div}(\rho_{\EPS k}(\vv_{\EPS k}{\otimes}\vv_{\EPS k})\Big){\cdot}\widetilde\vv_k
 \nonumber
 \\[-.4em]&\nonumber\hspace{4em}
 -\nu_0|\EE(\vv_\EPS)|^{p-2}\EE(\vv_\EPS)
{:}\EE(\vv_{\EPS k}{-}\vv_\EPS)
-\nu_1\big(|\nabla\EE(\vv_\EPS)|^{p-2}\nabla\EE(\vv_\EPS)\big)\Vdots
 \nabla\EE(\vv_{\EPS k}{-}\vv_\EPS)\bigg)
 \,\d\xx\d t
   \\[-.4em]
 &\hspace{2em}+\int_\varOmega\frac{\rho_0}2|\vv_0|^2-\rho_{\EPS k}(T)\vv_{\EPS k}(T)\cdot\widetilde\vv_k(T)+\frac{\rho_{\EPS k}(T)}2|\widetilde\vv_k(T)|^2\dx
   +\mathscr{O}_k+\mathscr{Q}_k
   \label{strong-hyper+}\end{align}
with $C_\EPS>0$ from \eq{eq:11b} %
and with some $c_{p}>0$ related to the inequality
$c_{p}|G-\widetilde G|^p\le(|G|^{p-2}G-|\widetilde G|^{p-2}\widetilde G)\Vdots(G-\widetilde G)$ holding for $p\ge2$. The remaining terms $\mathscr Q_k$ and 
$\mathscr{O}_k$  in \eq{strong-hyper+} are, respectively,
\begin{align}\nonumber
&\mathscr Q_k=\int_\varOmega\rho_{\EPS k}(T)\vv_{\EPS k}(T)\cdot(\widetilde\vv_k(T)-\vv_{\EPS}(T))+\frac{\rho_{\EPS k}}{2}(|\vv_\EPS(T)|^2-|\widetilde\vv_k(T)|^2)\ \ \text{ and}
\\\nonumber
&\mathscr{O}_k=%
\int_0^T\!\!\!\int_\varOmega\!
\nu_0|\EE(\vv_{\EPS k})|^{p-2}\EE(\vv_{\EPS k}){:}\EE(\widetilde\vv_k{-}\vv_\EPS)
 +\nu_1|\nabla\EE(\vv_{\EPS k})|^{p-2}\nabla\EE(\vv_{\EPS k})\Vdots
  \nabla\EE(\widetilde\vv_k{-}\vv_\EPS)\,\d\xx\d t\,;
\end{align}
they both converge to zero due to the strong approximation properties of the
approximation $\widetilde\vv_k$ of $\vv_\EPS$. 
Moreover, thanks to \eqref{conv2}, and using the same calculus as \eqref{eq:energy1}, we have
\begin{align}
&\int_0^T\!\! \int_{\varOmega}\Big(\frac{\partial}{\partial t}\left(\varrho_{\varepsilon k} \vv_{\varepsilon k}\right)+\operatorname{div}(\varrho_{\varepsilon k} \vv_{\varepsilon k} {\otimes} \vv_{\varepsilon k})\Big){\cdot}\widetilde{\vv}_k\,\d\xx\d t	
\nonumber\\&
\stackrel{k\to\infty}\longrightarrow
\int_0^T\!\! \int_{\varOmega}\Big(\frac{\partial}{\partial t}\left(\varrho_{\varepsilon} \vv_{\varepsilon}\right)+\operatorname{div}(\varrho_{\varepsilon} \vv_{\varepsilon} {\otimes} \vv_{\varepsilon})\Big){\cdot}{\vv}_\EPS\,\d\xx\d t
\nonumber
=\int_\varOmega \frac{\rho_\EPS}2|\vv_\EPS(T)|^2-\frac {\rho_0}2|\vv_0|^2\dx\,.
\end{align}
Furthermore
\begin{align}
&\int_{\varOmega}\!\frac{\rho_0}{2}\left|\vv_0\right|^2\!-\rho_{\varepsilon k}(T) \vv_{\varepsilon k}(T){\cdot}\widetilde{\vv}_k(T)+\frac{\rho_{\varepsilon k}(T)}{2}\left|\widetilde{\vv}_k(T)\right|^2\,\d\xx
\nonumber
\stackrel{k\to\infty}\longrightarrow\!\int_\varOmega\!\frac{\rho_0}2|\vv_0|^2\!
-\frac{\rho_\EPS(T)}2|\vv_\EPS(T)|^2\dx. 
\end{align}
Thus, the last term in the chain of inequalities \eqref{strong-hyper+} converges to 0 as $k\to\infty$. Thus we obtain the desired strong convergence
\begin{align}\label{strong-conv+}
&\ee(\vv_{\EPS k})\to\ee(\vv_{\EPS})\ \text{ strongly in 
$\,L^p(I;W^{1,p}(\varOmega;\R_{\rm sym}^{d\times d}))\,$}.
\end{align}
Overall, thanks to the already proved strong convergence \eqref{strong-v}, we
arrive at
$\vv_{\EPS k}\to\vv_\EPS$ 
{ strongly in }$L^p(I;W^{1,p}(\varOmega;\R^d))$.	

We also need to prove the convergence of the plastic dissipation.
{As $\LL_{\rm p \EPS}$ is not a legal test for the Galerkin approximation
\eq{Euler2-weak-Galerkin+} in general, }
we consider $\widetilde\LL_{{\rm p}k}{:I\to W_k}$ so that the sequence
$\{\widetilde\LL_{{\rm p}k}\}_{k\in\N}$ %
converges strongly in $L^\infty(I;W^{1,q}(\varOmega;\R^{d\times d}_{\rm dev}))$ to
$\LL_{\rm p \EPS}$. Using the assumptions \eqref{Euler-ass-zeta} on $\zeta$, 
we can write 
\begin{align}
&{\lambda c_2}\|\LL_{{\rm p}\,\EPS k}-\LL_{{\rm p}\,\EPS}\|^2_{L^2(I\times\varOmega;\R^{d\times d})}+c_q\|\nabla\LL_{{\rm p}\,\EPS k}-\nabla\LL_{{\rm p}\,\EPS}\|^q_{L^q(I\times\varOmega;\R^{d\times d\times d})}\nonumber
\\
&\le
\int_0^T\!\!\!\int_\varOmega\Big(\det(\nabla\bm\xi_{\EPS k})
\big(\partial_{\Lp}\!\zeta(\theta_{\EPS k};\LL_{{\rm p}\,\EPS k})-\partial_{\Lp}\!\zeta(\theta_{\EPS k};\LL_{{\rm p}\,\EPS})\big)
{:}(\LL_{{\rm p}\,\EPS k}-\LL_{{\rm p}\,\EPS})\nonumber\\
&\hskip5.0em+\nu_2(|\nabla\LL_{{\rm p}\,\EPS k}|^{q-2}\nabla\LL_{{\rm p}\,\EPS k}-|\nabla\LL_{{\rm p}\,\EPS}|^{q-2}\nabla\LL_{{\rm p}\,\EPS})\Vdots\nabla(\LL_{{\rm p}\,\EPS k}{-}\LL_{{\rm p}\,\EPS})\Big)\,\d\xx\d t\nonumber\\
&\le
\int_0^T\!\!\!\int_\varOmega \Big(\det(\nabla\bm\xi_{\EPS k})\zeta(\theta_{\EPS k};\LL_{{\rm p}\,\EPS k})-\det(\nabla\bm\xi_{\EPS k})\zeta(\theta_{\EPS k};{\LL}_{{\rm p}\,\EPS})\nonumber
\\
&\hskip5.0em -\det(\nabla\bm\xi_{\EPS k})\partial_{\Lp}\!\zeta(\theta_{\EPS k};\LL_{{\rm p}\,\EPS}){:}(\LL_{{\rm p}\,\EPS k}-\LL_{{\rm p}\,\EPS})\nonumber\\
&\hskip5.0em+\nu_2|\nabla\LL_{{\rm p}\,\EPS k}|^q-\nu_2|\nabla\LL_{{\rm p}\,\EPS}|^q-\nu_2|\nabla\LL_{{\rm p}\,\EPS}|^{q-2}\nabla\LL_{{\rm p}\,\EPS}\Vdots\nabla(\LL_{{\rm p}\,\EPS k}{-}\LL_{{\rm p}\,\EPS})\Big)\,\d\xx\d t \nonumber\\
&=\int_0^T\!\!\!\int_\varOmega \Big(\det(\nabla\bm\xi_{\EPS k})\zeta(\theta_{\EPS k};\LL_{{\rm p}\,\EPS k})-\det(\nabla\bm\xi_{\EPS k})\zeta(\theta_{\EPS k};{\widetilde\LL}_{{\rm p}\,k})+\nu_2|\nabla\LL_{{\rm p}\,\EPS k}|^q-\nu_2|\nabla\widetilde\LL_{{\rm p}\,k}|^q\nonumber\\
&\hskip1em-\det(\nabla\bm\xi_{\EPS k})\partial_{\Lp}\!\zeta(\theta_{\EPS k};\LL_{{\rm p}\,\EPS}){:}(\LL_{{\rm p}\,\EPS k}{-}\LL_{{\rm p}\,\EPS}) 
-\nu_2|\nabla\LL_{{\rm p}\,\EPS}|^{q-2}\nabla\LL_{{\rm p}\,\EPS}\Vdots\nabla(\LL_{{\rm p}\,\EPS k}{-}\LL_{{\rm p}\,\EPS})\Big)\,\d\xx\d t+\mathscr Q_k\nonumber\\
&\le \int_0^T\!\!\!\int_\varOmega\bigg(
\det(\nabla\bm\xi_{\EPS k})\Big({\FF}_{{\rm e}\,\EPS k}^\top[(\pi_\lambda\varphi)'_{\Fe}]^{\bm\xi_{\EPS k}}(\FF_{{\rm e}\,\EPS k})+\pi_\lambda(\FF_{{\rm e}\,\EPS k})\FF^\top_{{\rm e}\,\EPS k}\frac{[\gamma_{\Fe}']^{\bm\xi_{\EPS k}}(\FF_{{\rm e}\,\EPS k},\theta_{\EPS k})}{1+\EPS|\theta_{\EPS k}|^{\beta/q'}}\nonumber\\
&\hskip5em-[\phi'_{\Fp}]^{\bm\xi_{\EPS k}}(\FF_{{\rm p}\,\EPS k})\FF_{{\rm p}\,\EPS k}^\top\Big)
{:}(\LL_{{\rm p}\,\EPS k}-\widetilde\LL_{{\rm p}k})
-\det(\bm\xi_{\EPS k})\partial_{\Lp}\!\zeta(\theta_{\EPS k};\LL_{{\rm p}\,\EPS}){:}(\LL_{{\rm p}\,\EPS k}-\LL_{{\rm p}\,\EPS})\nonumber\\\nonumber
&\hskip5em-\nu_2|\nabla\LL_{{\rm p}\,\EPS}|^{q-2}\nabla\LL_{{\rm p}\,\EPS}\Vdots\nabla(\LL_{{\rm p}\,\EPS k}{-}\LL_{{\rm p}\,\EPS})\bigg)\,\d\xx\d t+\mathscr Q_k\stackrel{k\to\infty}\longrightarrow0\,,
\end{align}
where the last inequality follows from the weak form \eqref{Euler3-weak-Galerkin+} with test function $\LL_{{\rm p}\,\EPS k}-\widetilde\LL_{{\rm p}k}$, and where we have set
$\mathscr Q_k=\int_0^T\!\!\!\int_\varOmega \zeta(\theta_{\EPS k};\widetilde\LL_{{\rm p}\,k})-	\zeta(\theta_{\EPS};\LL_{{\rm p}\,\EPS})+\nu_2|\nabla\widetilde\LL_{{\rm p}\,k}|^q-\nu_2|\nabla\LL_{{\rm p}\,\EPS}|^q\,\d\xx\d t\stackrel{k\to\infty}\longrightarrow0$.
Thus, we obtain 
\begin{subequations}\label{strong-conv-Lp}
\begin{align}\label{strong-conv-Lp+}
&&&\LL_{{\rm p}\,\EPS k}\to\LL_{{\rm p}\,\EPS}&&\text{strongly in }L^2(I{\times}\varOmega;\R^{d\times d}_{\rm dev})\ \ \text{ and}&&&&
\\
&&&\nabla\LL_{{\rm p}\,\EPS k}\to\nabla\LL_{{\rm p}\,\EPS}\!\!\!&&\text{strongly in }	L^q(I{\times}\varOmega;\R^{d\times d\times d})\,.
\end{align}
\end{subequations}
Using {also}
the strong convergence of $\FF_{{\rm e}\,\EPS k}$ as well as the strong convergence of $\nabla\bm\xi_{\EPS k}$, we can pass to the limit in the weak forms (\ref{Euler-weak-Galerkin+}a,b,c) of the Galerkin formulation. We therefore obtain a weak solution of the $\EPS$--regularized system \eqref{Euler-hypoplast-xi-eps}.
{Let us remark that, in view of \eq{Euler-weak-sln-v3},
the convergence \eq{strong-conv-Lp+} is also weak* in
$L^\infty(I{\times}\varOmega;\R^{d\times d}_{\rm dev})$ and the (possibly unlimited)
growth of $\zeta(\theta,\cdot)$ is irrelevant.}

Before going to the next step, we observe that the mechanical energy-dissipation balance \eqref{energetics1} is inherited by the limit $(\vv_\EPS,\xi_\EPS,\FF_{\rm p\,\EPS},\LL_{\rm p\,\EPS},\theta_\EPS)$, namely,
 \begin{align}\nonumber
  &\hspace*{0em}\frac{\d}{\d t}
  \int_\varOmega\frac{\rhoRxieps|\vv_\EPS |^2}{2\,{\det}(\Fe_\EPS)}+
  \frac{\pi_\LAM(\Fe_\EPS)\varphi^{\bm\xi_\EPS}(\Fe_\EPS)\!}{{\det}(\Fe_\EPS)}
  \,\d\xx
\\[-.1em]\nonumber&\hspace{2em}
+\!\int_\varOmega\!{\nu_0|\EE(\vv_\EPS))|^p
+\nu_1|\Nabla\EE(\vv_\EPS)|^p+J\partial_{\Lp}\!\zeta(\theta;\LL_{\rm p\,\EPS}){:}\LL_{\rm p\,\EPS}+\nu_2|\nabla\LL_{\rm p\,\EPS}|^q}\,\d\xx
\\[-.1em]&\hspace{3em}
=\int_\varOmega\frac{\rhoR\GRAVITY{\cdot}\vv_\EPS}{\det_\LAM(\FF_\EPS)\!\!}
-\frac{\pi_\LAM(\Fe_\EPS)\COUPLING_{\Fe}'(\Fe_\EPS,\theta_\EPS)}{(1{+}\EPS|\theta_\EPS|^{\alpha%
}
){\det}(\Fe_\EPS)}{:}(\nabla\vv_\EPS\FF_{{\rm e}\,\EPS}{-}\FF_{{\rm e}\,\EPS}\LL_{{\rm p}\,\EPS k})\,\d\xx.
\label{thermodynamic-Euler-mech-disc++}
\end{align}
We recall that \eqref{energetics1} was obtained by testing the balance of linear momentum by the velocity and the flow rule by the plastic-strain rate. Thus, to justify \eqref{thermodynamic-Euler-mech-disc++}, we must check that these tests are legal. Indeed, we first rewrite the balance of linear momentum \eqref{Euler1=hypoplast-xi-eps} as
\begin{align}
\nonumber%
     &\pdt{}(\rho_\EPS\vv_\EPS)
     =\rho_\EPS\GRAVITY\\
     &\quad+{\rm div}
 \big(\TT^{\bm\xi_\EPS}_{\lambda,\EPS}(\FF_{\rm e\,\EPS},\theta_\EPS)
 +\nu_0|\ee(\vv_\EPS)|^{p-2}\ee(\vv_\EPS)-{\rm div}(\nu_1|\nabla\ee(\vv_\EPS)|^{p-2}\nabla\ee(\vv))
 -\rho_\EPS\vv_\EPS{\otimes}\vv_\EPS\big)\,.%
\nonumber	
\end{align}
In view of the inequality \eqref{ddt} considered in the limit as $k\to\infty$,
we notice that $\pdt{}(\rho_\EPS\vv_\EPS)\in L^{p'}(I;W^{2,p}(\varOmega;\R^d)^*)$
is in duality with $\vv_\EPS\in L^p(I;W^{2,p}(\varOmega;\R^d))$.
Since $\rho_\varepsilon\in L^\infty(I{\times}\Omega)$
due to \eq{eq:11} with $r>d$
and since $\GRAVITY\in L^1(I;L^\infty(\Omega;\R^d))$ by assumption, we have that
$\rho_\EPS\GRAVITY$
is in duality with $\vv_\EPS\in L^\infty(I;L^2(\Omega;\R^d))$, cf.\ \eq{pippo:1}.
Furthermore, by \eqref{Euler-T-strong-conv+}, $\TT^{\bm\xi_\EPS}_{\lambda,\EPS}$ is in duality with $\nabla\vv_\EPS\in L^p(I;W^{1,p}(\varOmega;\R^{d\times d}))$. Also, $|\ee(\vv_\EPS)|^{p-2}\ee(\vv_\EPS)\in L^{p'}(I;L^{p'}(\varOmega;\R_{\rm sym}^{d\times d}))$ is in duality with $\nabla\vv_\EPS$ and $|\nabla\ee(\vv_\EPS)|^{p-2}\nabla\ee(\vv_\EPS)\in L^{p'}(I{\times}\varOmega;\R^{d\times d\times d})$ is in duality with
$\nabla^2\vv_\EPS\in L^{p}(I{\times}\varOmega;\R^{d\times d\times d})$. Since $\partial_{\Lp}^{}\!\zeta(\theta_\EPS;\LL_{\rm p\, \EPS})$ is single valued for $\LL_{\rm p \EPS}\neq \bm0$, the integral $\int_0^T\!\int_\varOmega \partial_{\Lp}^{}\!\zeta(\theta_\EPS,{\Lp}_\EPS){:}\LL_{\rm p \EPS}\,\d\xx\d t$ is, in fact, single-valued\INSERT{; in particular, the intergal does exist due to \eq{Euler-ass-zeta}.}
\COMMENT{OK ??}
Furthermore, thanks to \eqref{Euler-weak-sln-v3}, $\LL_{\rm p \EPS}\in L^\infty(I{\times}\varOmega{;\R^{d\times d}_{\rm dev}})$. Thus, $\partial_{\Lp}^{}\!\zeta(\theta_\EPS,{\Lp}_\EPS)\in L^\infty(I{\times}\varOmega)$ and hence it is {surely} in duality with $\LL_{\rm p \EPS}$.

We next check that $\LL_{\rm p \EPS}$ is a legal test in the flow rule
\eqref{Euler4=hypoplast-xi-eps}. In fact, we recall that
$\det(\nabla\bm\xi_\EPS)$ is in $L^\infty(I{\times}\varOmega)$ by
\eqref{Euler-weak-sln-v}. Thus, by the already established duality between
$\partial_{\Lp}^{}\!\zeta(\theta_\EPS,{\Lp}_\EPS)$ and $\LL_{\rm p \EPS}$, we
have also that $\LL_{\rm p \EPS}$ is in duality with the first term on the
right-hand side of \eqref{Euler4=hypoplast-xi-eps}. Also, we recall
from \eqref{Euler-weak-sln-v2} and \eqref{Euler-weak-sln-v3} that
${\Fe}_\EPS\in L^\infty(I{\times}\varOmega;\R^{d\times d}_{\rm dev})$ and
${\Fp}_\EPS\in L^\infty(I{\times}\varOmega;\R^{d\times d}_{\rm dev})$, thus
the left-hand side of \eqref{Euler4=hypoplast-xi-eps} is in duality
with $\LL_{\rm p \EPS}$. Thus, by comparison also the remaining term
${\rm div}(\nu_2|\nabla\LL_{\rm p \EPS}|^{q-2}\nabla\LL_{\rm p \EPS})$ is
in duality with $\LL_{\rm p \EPS}$.
\medskip

\medskip\noindent{\it Step 6 -- non-negativity of temperature}:%
We can now perform various nonlinear tests of the regularized but
non-discretized heat equation. The first test can be by the negative part
of temperature $\theta_\EPS^-:=\min(0,\theta_\EPS)$. Let us recall the
extension \eq{extension-negative+}, which in particular gives
$\OMEGA^{\bm\xi}(\FF,\theta^-)=|\theta|^{\alpha-1}\theta^-$ and 
$[\OMEGA_\FF']^{\bm\xi}(\FF,\theta^-)=\bm0$. Note also that $\theta_\EPS^-\in L^2(I;H^1(\varOmega))$,
so that it is indeed a legal test for the regularized heat equation.
Here we rely on the data qualification $\nu_0,\nu_1,\nu_2\ge0$,
$\kappa=\kappa^{\bm\xi}(\FF,\theta)\ge0$,
$\theta_0\ge0$, and $h(\theta)\ge0$ for $\theta\le0$, cf.\
(\ref{Euler-ass+}f--h,l).
Realizing that $\nabla\theta^-\!=0$ wherever $\theta>0$ so that
$\nabla\theta{\cdot}\nabla\theta^-=|\nabla\theta^-|^2$ and that
$\COUPLING'_{\Fe}(\Fe,\theta)\theta^-=\COUPLING'_{\Fe}(\Fe,\theta^-)\theta^-=0$
and $h(\theta)\theta^-=h(\theta^-)\theta^-=\bm0$, this test gives%
\begin{align}\nonumber
&\frac1{\alpha{+}1}\frac{\d}{\d t}\|\theta_\EPS^-\|_{L^{1+\alpha}(\varOmega)}^{1+\alpha}
+\kappa\|\nabla\theta_\EPS^-\|^2_{L^2(\varOmega;\R^d)}\le
\int_\varOmega\theta_\EPS^-\pdt{\W_\EPS}+
\kappa^{\bm\xi_{\EPS}}(\Fe_\EPS,\theta_\EPS)\nabla\theta_\EPS{\cdot}\nabla\theta_\EPS^-\,\d\xx\\
\nonumber &=
\int_\varGamma\frac{h(\theta_\EPS)}{1{+}\EPS h(\theta_\EPS)}\theta_\EPS^-\,\d S
+\!\int_\varOmega\!\Big(
\frac{\nu_0|\ee(\vv_{\EPS})|^{p}+\nu_1|\nabla\EE(\vv_{\EPS})|^p
+\partial_{\Lp}\!\zeta(\theta_{\EPS};\LL_{{\rm p}\,\EPS}){:}\LL_{{\rm p}\,\EPS}
+\nu_2|\nabla\LL_{{\rm p}\,\EPS}|^q\!\!}{1{+}\EPS|\EE(\vv_{\EPS})|^p
{+}\EPS|\nabla\EE(\vv_\EPS)|^p
{+}\EPS\partial_{\Lp}\!\zeta(\theta_{\EPS};\LL_{{\rm p}\,\EPS}){:}\LL_{{\rm p}\,\EPS}
{+}\EPS|\nabla\LL_{{\rm p}\,\EPS}|^q}\theta_\EPS^-
\!
\nonumber\\
&\hspace{10em}+\frac{\COUPLING'_{\Fe}(\Fe_\EPS,\theta_\EPS)\!}{\det\Fe_\EPS(1{+}\EPS|\theta_\EPS|^{\alpha})}{:}(\nabla\vv_{\EPS}\FF_{\rm e\,\EPS}-\FF_{{\rm e}\,\EPS}\LL_{\rm p\,\EPS})\theta_\EPS^-+\W_\EPS\vv_\EPS{\cdot}\nabla\theta_\EPS^-\Big)\,\d\xx\nonumber
\\[-.1em]\nonumber
&\le\!\int_\varOmega\!\W_\EPS\vv_\EPS{\cdot}\nabla\theta_\EPS^-\,\d\xx
=\!\int_\varOmega\!|\theta_\EPS|^{\alpha-1}\theta_\EPS^-\vv_\EPS{\cdot}\nabla\theta_\EPS^-\,\d\xx
=\!\int_\varOmega\!(1{-}\alpha)|\theta_\EPS|^{\alpha-1}\theta_\EPS^-\vv_\EPS{\cdot}\nabla\theta_\EPS^-
\!-|\nabla\theta_\EPS^-|^{\alpha+1}{\rm div}\,\vv_\EPS\,\d\xx
\\[-.1em]
&=-\frac1{1{+}\alpha}\int_\varOmega|\nabla\theta_\EPS^-|^{1+\alpha}{\rm div}\,\vv_\EPS\,\d\xx
\le\|\theta_\EPS^-\|_{L^{1+\alpha}(\varOmega)}^{1+\alpha}\|{\rm div}\,\vv_\EPS\|_{L^\infty(\varOmega)}^{}\,.
\label{Euler-thermo-test-nonnegative}
\end{align}
Recalling the assumption $\theta_0\ge0$ so that $\theta_{0,\EPS}^-=0$ and exploiting the
integrability of $\|\vv_\EPS\|_{L^\infty(\varOmega;\R^d)}$ inherited
from \eq{bound0c}, by the Gronwall inequality we obtain
$\|\theta_\EPS^-\|_{L^\infty(I;L^{1+\alpha}(\varOmega))}=0$, so that $\theta_\EPS\ge0$
a.e.\ on $I{\times}\varOmega$. 

Having proved non-negativity of temperature, we can repeat the arguments leading to \eqref{bounds000}-\eqref{cutoff}. By doing so, we obtain the bounds
\begin{align}\label{bounds0+}
\|\vv_{\EPS}\|_{L^p(I;W^{2,p}(\varOmega;\R^d))}\le C\,,\ \ 
\|{\LL}_{{\rm p}\,\EPS}\|_{L^2(I\times\varOmega;\mathbb R^{d\times d})}\le C\,,\ \text{ and }\
\|\nabla{\LL}_{{\rm p}\,\EPS}\|_{L^q(I\times\varOmega;\mathbb R^{d\times d\times d})}\le C\,,
\end{align}
which do not depend on $\EPS$.
Again, exploiting the regularity of the velocity field, using the regularity
of the solutions of the transport equations \eqref{Euler2=hypoplast-xi-eps}
and \eqref{Euler3=hypoplast-xi-eps}, we obtain
\begin{subequations}
\begin{align}
&\|\bm\xi_\EPS\|_{L^\infty(I;W^{2,r}(\varOmega;\R^d))}\le C\,,\ 
\ \|\det(\nabla\bm\xi_\EPS)\|_{L^\infty(I;W^{1,r}(\varOmega))}\le C\,,\ \text{ and }\
\|\FF_{{\rm p}\,\EPS k}\|_{L^\infty(I;W^{1,r}(\varOmega;\R^d))}\le C\,.
\end{align}
Similarly, using the transport equation for $\Fe$, we obtain
\begin{align}
	\|\FF_{\rm e\,\EPS}\|_{L^\infty(I;W^{1,r}(\varOmega;\R^{d\times d}))}\le C.
	\label{bounds2+}
\end{align}
\end{subequations}
By comparison, we also have
\begin{align}
&\label{est+dF/dt+}
\Big\|\pdt{\FF_{\rm e\,\EPS}}\Big\|_{L^{\min(p,q)}(I;L^r(\varOmega;\R^{d\times d}))}^{}\!\le C
\ \ \ \text{ and }\ \ \ \Big\|\frac{\partial\bm\xi_{\EPS}}{\partial t}\Big\|_{L^p(I;L^\infty(\varOmega))}\le C.
\end{align}
The estimates from the heat equation yield
\begin{align}\label{basic-est-of-theta-eps+}
\|w_\EPS\|_{L^\infty(I;L^1(\varOmega))}^{}\le C\ \ \ \text{ and }\ \ \
\|\theta_\EPS\|_{L^\infty(I;L^\alpha(\varOmega))}^{}\le C\,.	
\end{align}
Finally, observing again that the density $\rho_\EPS=\rhoRxieps\det(\nabla\bm\xi_\EPS)$
is the solution of the mass balance equation with initial condition in
$W^{1,r}(\varOmega)$, we also obtain the inequalities
\begin{align}
\|\rho_\EPS\|_{L^\infty(I;W^{1,r}(\varOmega))}\le C\ \ \ \text{ and }\ \ \
\Big\|\frac 1 {\rho_\EPS}\Big\|_{L^\infty(I;W^{1,r}(\varOmega))}\le C.
\end{align}
 
\medskip\noindent{\it Step 7 -- further a-priori estimates}:
We are to prove an estimate of $\nabla\theta_\EPS$ based on the test of the $\EPS$--regularized heat
equation by $\chi_\zeta(\theta_\EPS)$ with an increasing
nonlinear function $\chi_\zeta:[0,+\infty)\to[0,1]$ defined as
\begin{align}\label{test-chi+}
\chi_\zeta(\theta):=1-\frac1{(1{+}\theta)^\zeta}\,,\ \ \ \ \zeta>0\,,
\end{align}
simplifying the original idea of L.\,Boccardo and T.\,Gallou\"et 
\cite{BDGO97NPDE,BocGal89NEPE} in the spirit of \cite{FeiMal06NSET},
expanding the estimation strategy in \cite[Sect.\,8.2]{KruRou19MMCM}.
In particular, instead of the estimation of $\int_\varOmega|\nabla\theta_\EPS|^2\,\d\xx$ as in the $L^2$-theory of the heat equation, we shall estimate the integral $I_\zeta^{(2)}(\theta_\EPS)$ defined below in %
\eq{8-**-+}. Once that integral is estimated, one can make use of the Gagliardo-Niremberg interpolation inequality to obtain a bound on $\int_\varOmega|\nabla\theta_\EPS|^\mu\,\d\xx$ for all $1\le\mu<(d+2\ALPH)/(d+\ALPH)$
(see \eqref{est-W-eps+} below).
Importantly, here we have $\chi_\zeta(\theta_\EPS(t,\cdot))\in H^1(\varOmega)$,
hence it is a legal test function, because 
$0\le\theta_\varepsilon(t,\cdot)\in H^1(\varOmega)$ has already been proved
and because $\chi_\zeta$ is uniformly Lipschitz continuous on $[0,+\infty)$. 

We consider $1\le \EXP<2$ and estimate the $L^\EXP$-norm of $\nabla\theta_\varepsilon$
by H\"older's inequality as 
\begin{align}\nonumber
&\!\!\int_0^T\!\!\!\int_\varOmega|\nabla\theta_\varepsilon|^\EXP\,\d\xx\d t
=C\bigg(\underbrace{\int_0^T\!\!\!
\big\|1{+}\theta_\varepsilon(t,\cdot)\big\|^{(1+\zeta)\EXP/(\TWO-\EXP)}
_{L^{(1+\zeta)\EXP/(\TWO-\EXP)}(\varOmega)}\,\d\xx\,\d t}_{\displaystyle\ \ \ =:I_{\EXP,\zeta}^{(1)}(\theta_\varepsilon)}\bigg)^{1-\EXP/\TWO}
\bigg(\underbrace{\int_0^T\!\!\!\int_\varOmega
\chi_\zeta'(\theta_\EPS)|\nabla\theta_\EPS|^\TWO\,\d\xx\,\d t}
_{\displaystyle\ \ \ =:I_{\zeta}^{(2)}(\theta_\varepsilon)}\bigg)^{\EXP/\TWO}\!
\\[-2em]
\label{8-**-+}
\end{align}
with $\chi_\zeta$ from \eq{test-chi+}.
Then we interpolate the Lebesgue space $L^{(1+\zeta)\EXP/(\TWO-\EXP)}(\varOmega)$ between 
$W^{1,\EXP}(\varOmega)$ and $L^\ALPH(\varOmega)$ in order to exploit the already 
obtained $L^\infty(I;L^\ALPH(\varOmega))$-estimate  in \eq{basic-est-of-theta-eps+}.
More specifically, by the Gagliardo-Nirenberg inequality, using the already derived bounds on $\theta_\EPS$, we obtain
\begin{align}\nonumber
\big\|1{+}\theta_\varepsilon(t,\cdot)\big\|_{L^{(1+\zeta)\EXP/(\TWO-\EXP)}(\varOmega)}^{(1{+}\zeta)\EXP/(\TWO{-}\EXP)}\!\le C
\Big(1+\big\|\nabla\theta_\varepsilon(t,\cdot)\big\|_{L^\EXP(\varOmega;\R^d)}\Big)^{{\sigma}(1{+}\zeta)\EXP/(\TWO{-}\EXP)}
\\\qquad\text{ with }\ 0<{\sigma}\le\frac{\EXP d}{\EXP d{+}\EXP{-}d}\Big(1-\frac{\TWO-\EXP}{(1{+}\zeta)\EXP}\Big)\,,
\label{8-cond+}\end{align}
We choose $0<{\sigma}\le1$ in such a way to obtain the desired exponent
${\sigma}(1{+}\zeta)\EXP/(\TWO{-}\EXP)=\EXP$, i.e.\ ${\sigma}:=(\TWO{-}\EXP)/(1{+}\zeta)$, so that 
\begin{align}
\!\!I_{\EXP,\zeta}^{(1)}(\theta_\varepsilon)&
\le C\bigg(1+\int_0^T\!\!\!\int_\varOmega\big|\nabla \theta_\varepsilon\big|^\EXP\,\d\xx\d t\bigg)\,.\label{I1}
\end{align}
By substituing \eqref{I1} into \eqref{8-**-+} we see that the $L^\mu$--norm of $\nabla\theta_\EPS$ is estimated once we estimate 
$I_{\zeta}^{(2)}(\theta_\varepsilon)$ in \eq{8-**-+}. To this aim, let us denote 
by ${\cal X}_\zeta(\XX,\Fe,\cdot)$ a primitive function to
$\theta\mapsto\chi_\zeta(\theta)[\OMEGA_\theta'](\XX,\Fe,\theta)$ depending smoothly
on $\Fe$ and $\XX$, specifically
\begin{align}
{\cal X}_\zeta(\XX,\Fe,\theta)
=\int_0^1\!\!\theta\chi_\zeta(r\theta)\OMEGA_\theta'(\XX,\Fe,r\theta)\,\d r\,.
\label{primitive++}\end{align}
Thus it holds $[({\cal X}_\zeta)_\Fe']^{\bm\xi}(\Fe,\theta)
=\int_0^1\theta\chi_\zeta(r\theta)\OMEGA_{\Fe\theta}''(\Fe,r\theta)\,\d r$ and $[({\cal X}_\zeta)_\XX']^{\bm\xi}(\Fe,\theta)
=\int_0^1\theta\chi_\zeta(r\theta)\OMEGA_{\XX\theta}''(\Fe,r\theta)\,\d r$.
Similarly as \eq{Euler-thermodynam3-test+++}, we have now the calculus
\begin{align}\nonumber
&\int_\varOmega\chi_\zeta(\theta)\pdt\W\,\d\xx
=\int_\varOmega\chi_\zeta(\theta)\Big([\OMEGA_\theta']^{\bm\xi}(\Fe,\theta)\pdt{\theta}
+%
[\OMEGA_\Fe']^{\bm\xi}(\Fe,\theta){:}\pdt{\Fe}
+
[\OMEGA_\XX']^{\bm\xi}(\Fe,\theta){\cdot}\pdt{\bm\xi}\Big)
\,\d\xx
\\&\qquad \nonumber
=\frac{\d}{\d t}\int_\varOmega{\cal X}^{\bm\xi}_\zeta(\Fe,\theta)\,\d\xx
-\int_\varOmega
[(\mathscr{X}_\zeta)_\Fe']^{\bm\xi}(\Fe,\theta){:}\pdt{\Fe}
+[(\mathscr{X}_\zeta)_\XX']^{\bm\xi}(\Fe,\theta){\cdot}\pdt{\bm\xi}\,\d\xx
\nonumber\\
&\hskip 4em
\text{when abbreviating }\ \mathscr{X}_\zeta(\XX,\Fe,\theta):={\cal X}_\zeta(\XX,\Fe,\theta)
-\chi_\zeta(\theta)\OMEGA(\XX,\Fe,\theta)\,.
\label{Euler-thermodynam3-test+}\end{align}
Altogether, testing the regularized heat equation by $\chi_\zeta(\theta_\EPS)$ gives
\begin{align}\nonumber
&\frac{\d}{\d t}\int_\varOmega\![{\cal X}_\zeta]^{\bm\xi_\EPS}(\Fe_\EPS,\theta_\EPS)\,\d\xx
+\int_\varOmega
\chi_\zeta'(\theta_\EPS)\kappa^{\bm\xi_\EPS}(\Fe_\EPS,\theta_\EPS)|\nabla\theta_\EPS|^2\,\d\xx
\\&\nonumber\ \ 
=		\!\int_\varOmega\!\bigg(\frac{\nu_0|\EE(\vv_\EPS))|^p
+\nu_1|\nabla\EE(\vv_\EPS)|^p+\partial_{\Lp}\!\zeta(\theta_\EPS;\LL_{{\rm p}\,\EPS}){:}\LL_{\rm p\,\EPS}+\nu_2|\nabla\LL_{\rm p\,\EPS}|^q}{1{+}\EPS|\EE(\vv_\EPS)|^p{+}\EPS|\nabla\EE(\vv_\EPS)|^p{+}\EPS\partial_{\Lp}\!\zeta(\theta_\EPS;\LL_{{\rm p}\,\EPS}){:}\LL_{\rm p\,\EPS}{+}\EPS|\nabla\LL_{\rm p\,\EPS}|^q}
\chi_\zeta(\theta_\EPS)\nonumber
\\[-.3em]\nonumber
&\qquad+\OMEGA^{\bm\xi_\EPS}(\Fe_\EPS,\theta_\EPS)\chi_\zeta'(\theta_\EPS)\vv_\EPS{\cdot}\nabla\theta_\EPS
+[(\mathscr{X}_\zeta)_\Fe']^{\bm\xi_\EPS}(\Fe_\EPS,\theta_\EPS){:}\pdt{\Fe_\EPS}
+[(\mathscr{X}_\zeta)_\XX']^{\bm\xi_\EPS}(\Fe_\EPS,\theta_\EPS){\cdot}\pdt{\bm\xi_\EPS}
\\&\qquad
+\frac{\pi_\LAM(\Fe_\EPS)\COUPLING'_{\Fe}(\Fe_\EPS,\theta_\EPS)\Fe_\EPS^\top
{:}\ee(\vv_\EPS)}{\det_\LAM(\Fe_\EPS)(1{+}\EPS|\theta_\EPS|^{\alpha}
)}\chi_\zeta(\theta_\EPS)\bigg)\,\d\xx
+\!\int_\varGamma\!\frac {h(\theta_\EPS)}{1+|h(\theta_\EPS)|}\chi_\zeta(\theta_\EPS)\,\d S\,.
\label{Euler-thermodynam3-test++-}\end{align}
We realize that
$\chi_\zeta'(\theta)=\zeta/(1{+}\theta)^{1+\zeta}$ as used
already in %
\eq{8-**-+} and that ${\cal X}_\zeta(\Fe_\EPS,\theta_\EPS)\ge
c_K\theta_\EPS$ with some $c_K$ for $\theta_\EPS\ge0$ due to
\eq{Euler-ass-c-Euler-ass-W+}; again $K$ is a compact subset of ${\rm GL}^+(d)$
related here with the already proved estimates \eq{bounds2+} on $\FF_{\rm e\,\EPS}$.
The convective term in \eq{Euler-thermodynam3-test++-} can be estimated as
\begin{align}
\int_\varOmega\W_\EPS\chi_\zeta'(\theta_\EPS)\vv_\EPS{\cdot}\nabla\theta_\EPS\,\d\xx
=\int_\varOmega\varpi^{\bm\xi_\EPS}(\FF_{\rm e\,\EPS},\theta_\EPS)\chi_\zeta'(\theta_\EPS)\vv_\EPS{\cdot}\nabla\theta_\EPS\,\d\xx+\varsigma(\alpha{-}1)\int_\varOmega\theta^\alpha\chi_\zeta'(\theta_\EPS)\vv_\EPS{\cdot}\nabla\theta_\EPS\,\d\xx.
\label{est-of-convective0}\end{align}
The first integral on the right-hand side of \eqref{est-of-convective0} can
be estimated as follows:
\begin{align}
\nonumber
&\int_\varOmega\varpi^{\bm\xi_\EPS}(\FF_{\rm e\,\EPS},\theta_\EPS)\chi_\zeta'(\theta_\EPS)\vv_\EPS{\cdot}\nabla\theta_\EPS\,\d\xx
\le\|\vv_\EPS\|_{L^\infty(\varOmega)}\int_\varOmega{(\varpi^{\bm\xi_\EPS}(\FF_{\rm e\,\EPS},\theta_\EPS))^2}{\chi_\zeta'(\theta_\EPS)}\,\d\xx\\
&\quad+\frac1{\TWO\zeta}\int_\varOmega\chi_\zeta'(\theta_\EPS)|\nabla\theta_\EPS|^\TWO\,\d\xx\le
C\big(1+{\|\theta_\EPS\|_{L^1(\varOmega)}^{}}\big)
+\frac1{\TWO\zeta}\int_\varOmega\chi_\zeta'(\theta_\EPS)|\nabla\theta_\EPS|^\TWO\,\d\xx,
\label{est-of-convective}\end{align}
where we have used the linear growth of $\varpi$ with respect to temperature
and the {decay $\chi_\zeta'(\theta)=\mathscr{O}(1/\theta)$}.

The second integral on the right-hand side of \eqref{est-of-convective0} can
be estimated through integration by parts and an application of the Young's
inequality:
\begin{align}
\nonumber
&\varsigma(\alpha{-}1)\int_\varOmega\theta_\EPS^\alpha\chi_\zeta'(\theta_\EPS)\vv_{\EPS}\cdot\nabla\theta_\EPS\dx=%
-\varsigma(\alpha{-}1)\int_\varOmega\theta_\EPS^{\alpha}\chi_\zeta''(\theta_\EPS)\nabla\theta_\EPS\cdot\vv_\EPS+\theta_\EPS^{\alpha+1}\chi_\zeta'(\theta_\EPS){\rm div}\vv_\EPS\,\dx
\\
&\hspace{0em}-\varsigma\alpha(\alpha{-}1)\int_\varOmega\theta_\EPS^\alpha\chi_\zeta'(\theta_\EPS)\nabla\theta_\EPS\cdot\vv_\EPS
\nonumber
\nonumber
=-\varsigma\frac{\alpha{-}1}{\alpha{+}1}\int_\varOmega\theta_\EPS^{\alpha}\chi_\zeta''(\theta_\EPS)\nabla\theta_\EPS\cdot\vv_\EPS+\theta_\EPS^{\alpha+1}\chi_\zeta'(\theta_\EPS){\rm div}\vv_\EPS\,\dx
\nonumber
\\
&\hspace{0em}\le C \frac{\delta}2\|\vv_\EPS\|_{L^\infty(\varOmega;\R^d)}\int_\varOmega\frac{(\theta_\EPS^\alpha\chi_\zeta''(\theta_\EPS))^2}{\chi'_\zeta(\theta_\EPS)}+\frac 1 {2\delta}\int_\varOmega\chi'_\zeta(\theta_\EPS)|\nabla\theta_\EPS|^2
\nonumber
+\|{\rm div}\vv_\EPS\|_{L^\infty(\varOmega)}\|\theta_\EPS\|^{\alpha+1}_{L^{\alpha+1}(\varOmega)}
\nonumber
\\
&\hspace{0em}\le C \frac{\delta}2\|\vv_\EPS\|_{L^\infty(\varOmega;\R^d)}(1+\|\theta_\EPS\|^{\alpha+1}_{L^{\alpha+1}(\varOmega)})+\frac 1 {2\delta}\int_\varOmega\chi_\zeta'(\theta_\EPS)|\nabla\theta_\EPS|^2
\nonumber
+\|{\rm div}\vv_\EPS\|_{L^\infty(\varOmega)}\|\theta_\EPS\|^{\alpha+1}_{L^{\alpha+1}(\varOmega)},
\end{align}
where we have used $\chi_\zeta''(\theta_\EPS)/\chi_\zeta'(\theta_\EPS)\le C/(1+\theta_\EPS)$ so that, thanks to the assumption $\alpha\le 2$, altogether $(\theta^\alpha\chi_\zeta(\theta_\EPS)/\chi_\zeta'(\theta_\EPS))^2\le C(1+\theta^{2(\alpha-1)})\le C(1+\theta^{\alpha+1})$ . %
Denoting by $0<\kappa_0=\inf_{\XX,\FF,\theta}\kappa(\XX,\FF,\theta)$, and using
\eq{Euler-thermodynam3-test+++} integrated over $I=[0,T]$,
we further estimate:
\begin{align}\nonumber
&
I_{\zeta}^{(2)}(\theta_\EPS)=
\frac1\zeta\int_0^T\!\!\!\int_\varOmega
\chi_\zeta'(\theta_\EPS)|\nabla\theta_\EPS|^\TWO\,\d\xx\d t
\le\frac1{\kappa_0\zeta}\int_0^T\!\!\!\int_\varOmega\!\kappa^{\bm\xi_\EPS}(\Fe_\EPS,\theta_\EPS)
\nabla\theta_\EPS{\cdot}\nabla\chi_\zeta(\theta_\EPS)\,\d\xx\d t 
\\&\nonumber
\le\frac1{\kappa_0\zeta}\bigg(
\int_0^T\!\!\!\int_\varOmega\!\kappa^{\bm\xi_\EPS}(\Fe_\EPS,\theta_\EPS)
\nabla\theta_\EPS{\cdot}\nabla\chi_\zeta(\theta_\EPS)
\,\d\xx\d t+\int_\varOmega\![{\cal X}_\zeta]^{\bm\xi_\EPS(T)}(\Fe_\EPS(T),\theta_\EPS(T))\,\d\xx\bigg)
\\&\nonumber
=\frac1{\kappa_0\zeta}\bigg(\int_\varOmega\![{\cal X}_\zeta]^{\bm\xi_0}(\Fe_0,\theta_{0,\EPS})\,\d\xx
\\ \nonumber &\hspace{2em}+\!\int_0^T\!\!\!\int_\varOmega\!\bigg(\frac{\nu_0|\EE(\vv_\EPS))|^p{+}\nu_1|\nabla\EE(\vv_\EPS)|^p+\partial_{\Lp}\!\zeta(\theta_\EPS;\LL_{{\rm p}\,\EPS}){:}\LL_{\rm p\,\EPS}+\nu_2|\nabla\LL_{\rm p\,\EPS}|^q}{1{+}\EPS|\EE(\vv_\EPS)|^p{+}\EPS|\nabla\EE(\vv_\EPS)|^p{+}\EPS\partial_{\Lp}\!\zeta(\theta_\EPS;\LL_{{\rm p}\,\EPS}){:}\LL_{\rm p\,\EPS}{+}\EPS|\nabla\LL_{\rm p\,\EPS}|^q}
\chi_\zeta(\theta_\EPS)
\\&\nonumber\hspace{3em}
+[(\mathscr{X}_\zeta)_\Fe']^{\bm\xi_\EPS}(\Fe_\EPS,\theta_\EPS){:}\pdt{\Fe_\EPS}
+[(\mathscr{X}_\zeta)_\XX']^{\bm\xi_\EPS}(\Fe_\EPS,\theta_\EPS){\cdot}\pdt{\bm\xi_\EPS}
+\OMEGA^{\bm\xi_\EPS}(\Fe_\EPS,\theta_\EPS)\chi_\zeta'(\theta_\EPS)\vv_\EPS{\cdot}\nabla\theta_\EPS
\\\nonumber&\hspace{3em}
+\frac{\pi_\LAM(\Fe_\EPS)[\COUPLING'_{\Fe}]^{\bm\xi_\EPS}(\Fe_\EPS,\theta_\EPS)\Fe_\EPS^\top
{:}\ee(\vv_\EPS)}{\det_\LAM(\Fe_\EPS)(1{+}\EPS|\theta_\EPS|^{\alpha})}
\chi_\zeta(\theta_\EPS)\bigg)\,\d\xx
+\!\int_0^T\!\!\!\int_\varGamma\!\frac{h(\theta_\EPS)}{1+\EPS |h(\theta_\EPS)|}
\chi_\zeta(\theta_\EPS)\,\d S\bigg)
\\&\nonumber
\le\frac1{\kappa_0\zeta}\bigg(
\big\|\nu_0|\EE(\vv_\EPS)|^p{+}\nu_1|\nabla\EE(\vv_\EPS)|^p+\partial_{\Lp}\!\zeta(\theta_\EPS;\LL_{{\rm p}\,\EPS}){:}\LL_{\rm p\,\EPS}\!+\nu_2|\nabla\LL_{\rm p\,\EPS}|^q\big\|_{L^1(I\times\varOmega)}\!
\\&\ \nonumber
+\big\|
[(\mathscr{X}_\zeta)_\Fe']^{\bm\xi_\EPS}(\Fe_\EPS,\theta_\EPS)
\big\|_{L^{\min(p,q)'}(I;L^{r'}(\varOmega;\R^{d\times d}))}^{}
\Big\|\pdt{\FF_{\rm e\,\EPS}}\Big\|_{L^{{\rm min}(p,q)}(I;L^r(\varOmega;\R^{d\times d}))}^{}
\\&\ \nonumber
+\big\|
[(\mathscr{X}_\zeta)_\XX']^{\bm\xi_\EPS}(\Fe_\EPS,\theta_\EPS)
\big\|_{L^{p'}(I;L^{1}(\varOmega;\R^d))}^{}
\Big\|\pdt{\bm\xi_{\EPS}}\Big\|_{L^p(I;L^\infty(\varOmega;\R^d))}^{}
+\big\|h_{\max}\big\|_{L^1(I\times\varGamma)}\!
\\&\nonumber\ 
+\Big\|\frac{\pi_\LAM(\Fe_\EPS)[\COUPLING'_{\Fe}]^{\bm\xi_\EPS}(\Fe_\EPS,\theta_\EPS)\Fe_\EPS^\top}{{\det}_\lambda\Fe_\EPS}
{:}\ee(\vv_\EPS)\Big\|_{L^1(I\times\varOmega)}\!\!
\\[-.2em]&\ 
+C\big(1{+}{\|\theta_\EPS\|_{L^1(\varOmega)}^{}}\big)+
C \frac{\delta}2\|\vv_\EPS\|_{L^\infty(\varOmega;\R^d)}\big(1{+}\|\theta_\EPS\|^{\alpha+1}_{L^{\alpha+1}(\varOmega)}\big)
\nonumber
\\[-.4em]&\ 
+\|{\rm div}\vv_\EPS\|_{L^\infty(\varOmega)}\|\theta_\EPS\|^{\alpha+1}_{L^{\alpha+1}(\varOmega)}
+\big\|{\cal X}_\zeta(\Fe_0,\theta_{0,\EPS})\big\|_{L^1(\varOmega)}\!
+\frac1{\delta}\!\int_0^T\!\!\!\int_\varOmega
\chi_\zeta'(\theta_\EPS)|\nabla\theta_\EPS|^\TWO\,\d\xx\d t\!\bigg)\,.
\nonumber
\end{align}
The second term in the last element of the above chain of inequalities if bounded because we already proved that dissipation in bounded, thanks to %
\eq{bounds0+}. Furthermore, in view of assumption \eq{Euler-ass-primitive-c+}, both
$|[({\cal X}_\zeta)_\Fe']^{\bm\xi_\EPS}(\Fe,\theta)|\le C(1+\theta^{(1+\alpha)/r'-1})$ and
$|\chi_\zeta(\theta_\EPS)[\OMEGA_\Fe']^{\bm\xi_\EPS}(\Fe,\theta)|\le C(1+\theta^{(1+\alpha)/r'-1})$. Thus, $[({\cal X}_\zeta)_\Fe']^{\bm\xi_\EPS}(\Fe,\theta)-\chi_\zeta(\theta_\EPS)[\OMEGA_\Fe']^{\bm\xi_\EPS}(\Fe,\theta)|$ is bounded in
$L^\infty(I;L^{((1+\alpha)r')/(1+\alpha-r')}(\varOmega))\subset L^\infty(I;L^{r'}(\varOmega))$. Likewise thanks to \eqref{Euler-ass-primitive-c+}, we have $|[({\cal X}_\zeta)_\XX']^{\bm\xi_\EPS}(\Fe,\theta)|\le C(1+\theta^\alpha)$ and $|\chi_\zeta(\theta_\EPS)[\OMEGA_\XX']^{\bm\xi_\EPS}(\Fe,\theta)|\le C(1+\theta^{\alpha})$, so that altogether
$[({\cal X}_\zeta)_\XX']^{\bm\xi_\EPS}(\Fe,\theta)-\chi_\zeta(\theta_\EPS)[\OMEGA_\XX']^{\bm\xi_\EPS}(\Fe,\theta)|$ is bounded in $L^\infty(I;L^{1+1/\alpha}(\varOmega))\subset L^{p'}(I;L^1(\varOmega))$. Due to the assumption \eq{Euler-ass-adiab+}, we can estimate the adiabatic rate as
\begin{equation}
\begin{aligned}
\int_\varOmega\!\! \frac{\pi_\LAM(\Fe_\EPS)\COUPLING'_{\Fe}(\Fe_\EPS,\theta_\EPS)\Fe_\EPS^\top}{{\det}_\lambda\Fe_\EPS}{:}\ee(\vv_\EPS)\,\d\xx &\le C \int_\varOmega\!C(1+\theta_\EPS^{\alpha%
})|\ee(\vv_\EPS)|\,\d\xx\\
&\le C(1+\|\theta_\EPS\|_{L^1(\varOmega)}+\|\ee(\vv_\EPS)\|_{L^p(\varOmega;\R^{d\times d})})\le C.
\end{aligned}\nonumber
\end{equation}
Here we exploited the boundedness of $\Fe\mapsto 1/{\det}_\lambda\Fe$. Recall that the truncation parameter $\lambda$ ultimately depends on the data of the problem.
Altogether, we have $I_{\zeta}^{(2)}(\theta_\EPS)$ estimated. 
Choosing $\sigma:=(\TWO{-}\EXP)/(1{+}\zeta)$ for \eq{8-cond+}
with $\zeta>0$ arbitrarily small, one gets after some algebra,
the condition $\EXP<(d{+}2\ALPH)/(d{+}\ALPH)$.

Altogether, we proved
\begin{subequations}\label{est-W-eps+}\begin{align}
&\|\theta_\EPS\|_{L^\infty(I;L^1(\varOmega))\,\cap\,L^\EXP(I;W^{1,\EXP}(\varOmega))}^{}\le C_\EXP
\ \ \text{ with }\ 1\le\EXP<\frac{d{+}2\ALPH}{d{+}\ALPH}\,.
\intertext{Exploiting the calculus
$\nabla\W_{\EPS}=[\OMEGA_\theta']^{\bm\xi_\EPS}(\Fe_{\EPS},\theta_{\EPS})\nabla\theta_{\EPS}
+[\OMEGA_\Fe']^{\bm\xi_\EPS}(\Fe_{\EPS},\theta_{\EPS})\nabla\Fe_{\EPS}+[\OMEGA_\XX']^{\bm\xi_\EPS}(\Fe_{\EPS},\theta_{\EPS})\nabla\bm\xi_{\EPS}$
with $\nabla\Fe_{\EPS}$ bounded in $L^\infty(I;L^r(\varOmega;\R^{d\times d\times d}))$ and $\nabla\bm\xi_\EPS$ bounded in $L^\infty(I;W^{1,r}(\varOmega;\R^{d\times d}))$
and relying on the assumption \eq{Euler-ass-primitive-c+},
we have also the
bound on $\nabla\W_{\EPS}$ in $L^\EXP(I;L^{\EXP^*d/(\EXP^*+d)}\varOmega;\R^d)$, so
that %
}
&\|\W_\EPS\|_{L^\infty(I;L^{\alpha+1}(\varOmega))\,\cap\,
L^\mu(I;L^{{\mu(1+\alpha)}/({\alpha\mu+\alpha+1})}(\varOmega))}\le C\,.
\end{align}\end{subequations}

\medskip

\noindent{\it Step 8: Limit passage for $\EPS\to0$}.
We use the Banach selection principle as in Step~5. We extract some subsequence
of $(\vv_{\EPS},\bm\xi_{\EPS},\FF_{{\rm p},\EPS},\LL_{{\rm p},\EPS},\theta_{\EPS},w_{\EPS})$
and a limit $(\vv,\bm\xi,\Fp,\Lp,\theta,w)
:I\to W^{2,p}(\varOmega;\R^d)\times W^{2,r}(\varOmega;\R^{d})\times W^{1,r}(\varOmega;\R^{d\times d}_{\rm dev})\times W^{1,q}(\varOmega;\R^{d\times d}_{\rm dev})\times W^{1,\mu}(\varOmega)\times W^{\mu^*d/(\mu^*+d)}(\varOmega)$
such that%
\begin{subequations}%
\begin{align}
&\!\!\vv_{\EPS}\to\vv&&\text{weakly in $\ L^p(I;W^{2,p}(\varOmega))$}\,,\label{Euler-weak-sln-v+}
\\\nonumber
&\!\!\bm\xi_{\EPS}\to\bm\xi&&\text{weakly* in $\
L^\infty(I;W^{2,r}(\varOmega;\R^d))\cap W^{1,p}(I;W^{1,r}(\varOmega;\R^d))$}\!\!&&
\\
&&&\hspace{11em}\text{and strongly in }C(I\times\barvarOmega;\R^d)\,,
\\\nonumber
&\!\!\FF_{{\rm p}\,\EPS}\to\Fp\!\!\!&&\text{weakly* in $\ L^\infty(I;W^{1,r}(\varOmega;\R^{d\times d}))\,\cap\,W^{1,q}(I;L^r(\varOmega;\R^{d\times d}))$}\!\!
\\
& &&\hspace{11em}\text{and strongly in }C(I\times\barvarOmega;\R^{d\times d})\,,
\\
\label{bounded}
&\!\!\LL_{{\rm p}\,\EPS}\to\Lp&&\text{weakly* in $\ L^\infty(I;W^{1,q}(\varOmega;\R^{d\times d}_{\rm dev}))$,}\!\!&&
\\
&\!\!\theta_{\EPS}\to\theta&&
\text{weakly in $L^\mu(I;W^{1,\mu}(\varOmega))$}\,,
\label{Euler-weak-sln-theta++++}
\\
&\!\!\W_{\EPS}\to\W&&
\text{weakly in $\ L^\mu(I;W^{\mu^*d/(\mu^*+d)}(\varOmega))$}\,.
\label{Euler-weak-sln-w++++}
\intertext{By the analogous argumentation as in \eq{strong-hyper+} with $C_\EPS$ now indepedent of $\EPS$, we prove}
&{\ee(\vv_{\EPS})\to\ee(\vv)}&&\text{strongly in 
$\,L^p(I;W^{1,p}(\varOmega;\R_{\rm sym}^{d\times d}))\,$}.
\intertext{Then, arguing as in the derivation of \eqref{w-conv+}, we deduce
}
&\!\!w_\EPS\to w &&\text{strongly in }L^c(I{\times}\varOmega),\quad
1\le c<\frac{2\alpha{+}d}{\alpha d}\,,
\intertext{and then, using again the continuity of $(\XX,\Fe,w)\mapsto [\omega(\XX,\Fe,\cdot)]^{-1}(w)$ as in \eqref{z-conv+}, we also have
}
&\!\!\theta_\EPS\to\theta=[\omega^{\bm\xi}(\Fe,\cdot)]^{-1}(w)\hspace*{-3em}
&&\hspace*{3em}\text{strongly in }L^c(I{\times}\varOmega),\quad 1\le c<1+2\alpha/d\,.
\intertext{By the continuity of $\varphi_{\Fe}'$, $\gamma_{\Fe}'$, ${\det}_\lambda$, and $\kappa$, we have also %
}
&\!\!\kappa^{\bm\xi_\EPS}(\Fe_{\EPS},\theta_\EPS)\to\kappa^{\bm\xi}(\Fe,\theta)\hspace*{-3em}
&&\hspace*{3em}\text{strongly in }L^c(I{\times}\varOmega),\quad 1\le c<\infty,\text{ and}
\end{align}
\vskip-2em
\begin{align}
&\!\!\TT^{\bm\xi_\EPS}_{\lambda,\EPS}\to \TT^{\bm\xi}_\lambda=\frac{[(\pi_\lambda\varphi)']^{\bm\xi}(\Fe)+\pi_\lambda(\Fe)[\gamma'_{\Fe}]^{\bm\xi}(\Fe,\theta)}{{\det}\Fe}\Fetop &&\ \text{strongly in }L^c(I{\times}\varOmega)
\nonumber\\[-.5em]
&&&\text{ for any }1\le c<%
\frac{2\alpha{+}d}{\alpha d}.
\end{align}
\end{subequations}
Using the same argument as that used to obtain \eqref{strong-conv+} and \eqref{strong-conv-Lp} we deduce the strong convergence of $\ee(\vv_\EPS)$, $\nabla\ee(\vv_\EPS)$, $\LL_{\rm p\,\EPS}$ and $\nabla\LL_{\rm p\,\EPS}$ to the corresponding limits. The result is
\begin{subequations}\label{weak223}
\begin{align}
&&&\vv_\EPS\to\vv&&\text{strongly in }L^p(I;W^{2,p}(\varOmega;\R^d))\,,	
\\&&&\LL_{\rm p\EPS}\to\LL_{\rm p}&&\text{strongly in }L^2(I{\times}\varOmega;\R_{\rm dev}^{d\times d})\,,\text{ and}&&&&&&&&\\
&&&\nabla\LL_{\rm p\EPS}\to \nabla\LL_{\rm p}&&\text{strongly in }L^p(I{\times}\varOmega;\R^{d\times d\times d})\,.&&
\end{align}
\end{subequations}
The first of the above convergences allow us to pass to the limit in the momentum balance equation. From \eqref{weak223} we obtain that, up to a subsequence, $\LL_{\rm p\EPS}$ converges to $\Lp$ almost everywhere in $I{\times}\overline\varOmega$. Because, of \eqref{bounded}, $\|\LL_{\rm p\EPS}\|_{L^\infty(I\times\varOmega;\R^{d\times d})}\le C$, and thus also $\|\zeta(\theta_\EPS,\LL_{\rm p\EPS})\|_{L^\infty(I\times\varOmega)}\le \|\zeta_\text{\sc m}(\LL_{\rm p\EPS})\|_{L^\infty(I\times\varOmega)}\le C$ so that by the bounded convergence theorem $\zeta(\theta_\EPS;\LL_{\rm p\EPS})$ converges to $\zeta(\theta;\LL_{\rm p})$ a.e.\ in $I{\times}\varOmega$\INSERT{ here
$\zeta_\text{\sc m}$ comes from \eq{Euler-ass-zeta}.}
. A similar argument applies to the right-hand side of the flow rule. In fact, the strong convergence of $\FF_{\rm e\EPS}$ and $\FF_{\rm p\EPS}$ and their uniform boundedness in $I{\times}\varOmega$ entail that
\begin{align}
{\rm dev}\Big(\FF^{\top}_{\rm e\,\EPS}[(\pi_\lambda\varphi)_{\Fe}']^{\bm\xi_\EPS}(\FF_{\rm e\,\EPS})+\pi_\lambda(\FF_{\rm e\,\EPS})\FF_{\rm e\,\EPS}^{\top}\frac{[\gamma_{\Fe}']^{\bm\xi_\EPS}(\FF_{\rm e\,\EPS},\theta_\EPS)}{1+\EPS|\theta_\EPS|^\alpha}-[\phi'_{\Fp}]^{\bm\xi_\EPS}(\FF_{\rm p\,\EPS})\FF_{\rm p\,\EPS}^\top\Big)
\end{align}
converges a.e.\ in $I{\times}\varOmega$ and is bounded
\INSERT{ in $L^\infty(I{\times}\varOmega;\R_{\rm dev}^{d\times d})$} by a constant.
This allows the passage to the limit also in the flow rule.

\medskip\noindent{\it Step 9: The original problem and the energy balance}.
Let us note that the limit $\Fe$ lives in
$L^{\infty}(I;W^{1, r}(\varOmega;\R^{d \times d}))\cap W^{1,\min(p,q)}(I;L^r(\varOmega;\R^{d\times d}))$,
cf.\ \eqref{bounds2+} and \eqref{est+dF/dt+}, and this space is embedded into
$C(I{\times}\barvarOmega;\R^{d\times d})$ since $r>d$. Therefore $\Fe$ and its determinant
evolve continuously in time, being valued respectively in $C(\barvarOmega;\R^{d \times d})$ and
$C(\barvarOmega)$. Let us recall that at the initial time $\Fe(0)=\nabla\bm\xi_0^{-1}\Fp^{-1}$
complies with the bounds \eqref{cutoff} and that this holds also for the $\lambda$-regularized system. Therefore $\Fe$ satisfies these bounds not only at $t=0$ but also at least for small times. Yet, in view of the choice of $\lambda$, this means that the $\lambda$-regularization is nonactive and $(\vv,\bm\xi,\Lp,\FF_{\mathrm{p}},\theta)$ solves, at least for a small time, the original nonregularized problem for which the a priori $L^{\infty}$-bounds \eqref{cutoff3}--\eqref{cutoff2} hold. By the continuation argument, we may see that the $\lambda$-regularization remains therefore inactive within the whole evolution of $(\vv,\boldsymbol{\xi},\LL_{\rm p},\FF_{\rm p},\theta)$ on the whole time interval $I$. In particular,
the calculations which led to \eqref{energetics1}
and \eqref{tot-energy-balance} are legitimate.
\end{proof}

\begin{remark}[{\sl Role of inertia}]\label{rem-inert}\upshape
Involving inertia in the model is essential to bootstrap the process that leads to the formal estimates. A key part of our proof is the derivation of the integrability estimates on the velocity field $\vv$ and on the plastic-strain rate $\Lp$. These estimates are naturally obtained by bounding the dissipative terms in the energy-dissipation balance \eqref{energetics1}. However, we cannot resort to \eqref{energetics1} right away, because of the estimation of the adiabatic term on the right-hand side. Thus, we must first resort to usage of the total-energy balance. To use the latter, however, we need to control the power of the applied forces, which involves the velocity, and the only way (at least to our knowledge) is to exploit the kinetic energy term on the left-hand side. Altogether, the optically simpler quasistatic variant, which works in isothermal situations
\cite{Roub22QHLS}, would be troublesome in the full anisothermal model.  
\end{remark}

\begin{remark}[Assumptions on $\gamma(\Fe,\theta)$]\upshape
The requirement \eqref{Euler-ass-adiab+} on $\gamma'_{\Fe}$  guarantees that {adiabatic effects} {have not} bigger growth {in $\theta$ than} the function $\omega$. This is needed to apply a Gronwall-type argument in order to derive a priori estimates on the temperature, \emph{cf.} \eqref{basic-est-of-theta-eps+} {above}. On the other hand, assumption \eqref{Euler-ass-primitive-c+}, is essentially technical, related to our proof strategy {controlling} the partial derivatives of the enthalpy when estimating the temperature field (see \eqref{est-of-convective}).
\end{remark}

\begin{remark}[Assumptions on $r$]\upshape
Beside having $r>d$, the particular requirement $r>\alpha$ is needed to establish the estimate \eqref{Euler-weak-sln-w} {above}, which {was} needed to prove the weak convergence ofthe adiabatic terms.
\end{remark}

{\small
\bibliographystyle{plain}
\bibliography{tr-gt-Eulerian-thermoplast}
}

\end{sloppypar}\end{document}